\documentclass{amsart}
\usepackage{mathrsfs}
\usepackage[all]{xy}
\usepackage{amssymb,amsmath,amsthm,amscd,hyperref}

\usepackage{changepage,comment}
\usepackage{algorithm}
\usepackage[noend]{algpseudocode}
\makeatletter
\def\BState{\State\hskip-\ALG@thistlm}
\makeatother

\usepackage{ifpdf}
\ifpdf 
  \usepackage[pdftex]{graphicx}
  \DeclareGraphicsExtensions{.pdf,.png,.jpg,.jpeg,.mps}
  \usepackage{pgf}
\else 
  \usepackage{graphicx}
  \DeclareGraphicsExtensions{.eps,.bmp}
  \usepackage{pgf}
  \usepackage{pstricks}
\fi
\usepackage{wrapfig}
\usepackage{tikz}

\usepackage{soul}
\usepackage{url}

\newtheorem{thm}{Theorem}[section]
\newtheorem{prop}[thm]{Proposition}
\newtheorem{eg}[thm]{Example}
\newtheorem{cor}[thm]{Corollary}
\newtheorem{lem}[thm]{Lemma}

\newtheorem*{claim}{Claim}

\theoremstyle{definition}
\newtheorem{defn}[thm]{Definition}

\theoremstyle{remark}
\newtheorem{quest}[thm]{Question}
\newtheorem*{rem}{Remark}

\newtheorem*{examples}{Examples}

\newcommand{\isom}{\textrm{Isom}} 
 
\newcommand{\mcg}{\textrm{Mod}(\Sigma_g)}  

\newcommand{\Gx }{\mathscr{G} (G, S)}

\newcommand{\fl }{\len_\lambda}

\newcommand{\act}{\curvearrowright}

\newcommand{\Gf}{\overline{G}_\lambda}
\newcommand{\pGf}{\partial_\lambda{G}}

\newcommand{\pG}{\Lambda G}

\newcommand{\cG}{\Lambda^{\mathrm c}{G}}
\newcommand{\ncG}{\Lambda^{\mathrm nc}{G}}
\newcommand{\mG}{\Lambda^{\mathrm m}{G}}
\newcommand{\uG}{\Lambda^{\mathrm u}{G}}

\newcommand{\collar}{\operatorname{\mathbf{collar}}}

\newcommand\ep{\epsilon}

\newcommand{\I}{{\mathbf I}}

\newcommand{\U}{X}
\newcommand{\bU}{\overline{\U}}

\newcommand{\pU}{\partial{\U}}

\newcommand\Z{{\mathbb Z}}
\newcommand\R{{\mathbb R}}

\newcommand\Hyp{{\mathbb H}}
\newcommand\til{\widetilde}

\newcommand{\e}[1]{\omega_{#1}}

\newcommand{\HD}{\mathrm{Hdim}}
\newcommand{\ax}{\mathrm{Ax}}
\newcommand{\Ar}{\mathrm{Arc}}
\newcommand{\Ax}{\mathbf{Ax}}
\newcommand{\Dc}{\mathrm{Dbc}}

\newcommand{\pmf}{\mathscr {PMF}}

\newcommand{\T}{\mathcal {T}}

\newcommand{\natls}{{\mathbb N}}

\newcommand\AAA{{\mathcal A}}

\newcommand\CC{{\mathcal C}}

\newcommand\HH{{\mathcal H}}
\newcommand\II{{\mathcal I}}

\newcommand\LL{{\mathcal L}}

\newcommand\PP{{\mathcal P}}

\newcommand\TT{{\mathcal T}}

\newcommand{\diam }[1]{{\textbf{diam}\big(#1\big)}}

\newcommand{\proj}{\textbf{d}}

\newcommand{\len }{\ell}

\newcommand{\inj }{\mbox{Inj}}

\usepackage[margin=1.1in]{geometry}
\usepackage[normalem]{ulem}

\newtheoremstyle{query}%
{}{}
{\color{red}}
{}
{\sffamily\bfseries}{:}{12pt}
{}
\theoremstyle{query}

\begin{document}

\title[Hausdorff Dimension of non-conical and Myrberg limit sets]{Hausdorff Dimension of non-conical and Myrberg limit sets}

\author{Mahan Mj} 
\address{School of Mathematics,
	Tata Institute of Fundamental Research. 1, Homi Bhabha Road, Mumbai-400005, India}

\email{mahan@math.tifr.res.in}
\email{mahan.mj@gmail.com}

\author{Wenyuan Yang}

\address{Beijing International Center for Mathematical Research\\
Peking University\\
 Beijing 100871, China
P.R.}
\email{wyang@math.pku.edu.cn}
\thanks{}


\subjclass[2000]{Primary 20F65, 20F67, 37D40}

\date{\today}

\dedicatory{}

\keywords{Non-conical points,  Myrberg points, Hausdorff dimension, amenability, geometric limits}

\begin{abstract} 
In this paper, we develop techniques to study the   Hausdorff dimensions of  non-conical 
and Myrberg limit sets for groups acting on negatively curved spaces. We establish maximality of the Hausdorff dimension of the non-conical limit set of $G$ in the following cases.
\begin{itemize}
\item  $M$ is a finite volume complete Riemannian manifold of pinched negative curvature and  $G$ is an
infinite normal subgroups of infinite index in $\pi_1(M)$.
\item $G$ acts on a regular tree $X$ with $X/G$ infinite and amenable
(dimension 1).
\item $G$ acts on the hyperbolic plane $\Hyp^2$ such that  
$\Hyp^2/G$ has Cheeger constant zero (dimension 2).
\item  $G$ is a finitely generated geometrically infinite 
Kleinian group (dimension 3).
\end{itemize}

We also show that the Hausdorff dimension  of the Myrberg limit set is the same as the critical exponent, confirming a conjecture of Falk-Matsuzaki.

\end{abstract}

\maketitle

\tableofcontents

\section{Introduction}

Let $M$ be a complete Riemannian manifold with pinched negative sectional curvature. Fix a point $o\in M$. A geodesic ray $\gamma$ issuing from $o$ is called \textit{recurrent} if it returns to a fixed compact set of $M$  infinitely often. Otherwise, $\gamma$ is called \textit{escaping}. Denote by $\mathcal R(o)$ and $\mathcal E(o)$ the  set of recurrent and escaping geodesics. Their size is measured   in terms of the Hausdorff dimension of their limit sets.

The goal of this paper is thus to study the  behavior of geodesic rays  in terms of limit sets. Much of the discussion works in the general framework of   Gromov hyperbolic spaces. Let  $X$ be Gromov hyperbolic.
Let  $\pU$ be its  Gromov boundary.  
A point $\xi\in \pU$ is called a \textit{limit point} if it is an accumulation point of the orbit $Go$ for some (and hence any) $o\in  X$. The set of limit points of $Go$ is called  the \textit{limit set} of $G$ denoted as $\Lambda G$. A \emph{non-wandering geodesic ray} is a geodesic ray in $X$ ending at a  point in $\Lambda G$. We say that a limit point $\xi\in \Lambda G$ is \textit{conical} if there exists a sequence  $g_n\in G$ and a  geodesic ray $\sigma \subset X$ ending at $\xi$
so that $\{g_no\} $ is contained in a finite neighborhood of $\sigma$. Hence, $\sigma$ projects to a recurrent geodesic  in the quotient $M= X/G$. If $\xi\in \pU$ is \textit{non-conical}, then any geodesic ray $\sigma \subset X$ ending at $\xi$ projects to an escaping geodesic.



There are two important and complementary sub-classes of  conical limit points: uniformly conical points and {Myrberg limit points}. The former class corresponds  exactly to geodesic rays with compact closure on $M$. The latter exhibits opposite behavior. For $M= X/G$ a negatively curved manifold the corresponding geodesic rays are dense in the unit tangent bundle of $M$. For this reason, the corresponding geodesic rays are sometimes called transitive geodesic rays.  The definition in terms of limit points is a bit involved, but intuitively suggestive:  $\xi\in \pU$ is a \textit{Myrberg limit point} if there exists a geodesic ray starting at some $o\in X$ ending at $\xi$ so that the set $G(o,\xi)=\{(go,g\xi): g\in G\}$ is dense in the ordered pairs of distinct points in $\pG$. We refer to Table \ref{tbl-geodesic-limitpoints} for a summary of  limit points and geodesic rays considered in this paper.

The conical (resp.\ non-conical) limit sets will be denoted as $\cG$ (resp.\ $\ncG$). 
We denote by  $\uG$ and $\mG$ the sets of uniformly conical points and Myrberg limit points respectively. 

\subsection{Statement of results: escaping geodesics}


Let $X$ be a proper Gromov hyperbolic space. We equip the Gromov boundary with a canonical class of visual metrics with parameter $\epsilon$
(see \cite[Chapter III.H]{bridson-haefliger} or \cite{GhH} for details). If $X$ is CAT(-1), the visual metric could be explicitly written (with $\epsilon=1$) as 
$$
\rho_o(\xi, \eta) = \mathrm{e}^{-\langle\xi,\eta\rangle_o}
$$
where $\langle\xi,\eta\rangle_o$ is  the continuous extension to $\pU$ of Gromov product $\langle x,y\rangle_o=d(x,o)+d(y,o)-d(x,y)/2$. 
We shall denote the  Hausdorff of a set $A$ by $\HD (A)$ and the limit set of a group $\Gamma$ acting on $X$ by $\Lambda \Gamma$.

  The following definition defines the framework we explore in this paper in the context of non-conical limit sets.
  \begin{defn}\label{def-max}
  Let $X$ be a proper Gromov hyperbolic space and $\pU$ its boundary equipped with a visual metric. Let $\Gamma$ be a group acting properly on $X$ and 
  $G<\Gamma$ be a subgroup.
  
  If $\HD (\ncG) = \HD (\Lambda \Gamma)$, we shall say that $\ncG$ has \emph{maximal Hausdorff dimension in $\Lambda \Gamma$.}
  If $\HD (\ncG) = \HD (\pU)$, we shall say that $\ncG$ has \emph{maximal Hausdorff dimension in $\pU$.}
  \end{defn}
  
  A substantial part of this paper is devoted to obtaining positive answers to the following question.
  
  \begin{quest}\label{qn-max}
  Let $X, \Gamma, G$ be as in Definition~\ref{def-max}. Find conditions 
  on $X, \Gamma, G$ such that
  \begin{enumerate}
  \item $\ncG$ has maximal Hausdorff dimension in $\Lambda \Gamma$.
  \item $\ncG$ has {maximal Hausdorff dimension in $\pU$.}
  \end{enumerate}
  \end{quest}
 
 We start with the following theorem that provides a positive answer to 
 Question~\ref{qn-max} (see Corollary~\ref{cor-normalsubgp}).

\begin{thm}\label{nonconical-normal-covering-intro}
	Let $N$ be a complete finite volume Riemannian manifold of pinched
	negative curvature. Let $\Gamma=\pi_1(N)$ and $G$ an infinite normal subgroup of $\Gamma$ with $\Gamma/N$ infinite.
	Let $M$ be the cover of   $N$ corresponding to the subgroup $\Gamma$.
	Let $X = \til N$. Then 
	  Then $\ncG$ has {maximal Hausdorff dimension in $\pU$.}
\end{thm}

The above follows from the next result which holds in a general setting (see Theorem~\ref{NormalSubgroupCase}). To state the result, let us introduce the   \textit{critical exponent} of a group $G$ as follows $$\e{G}=\limsup_{n\to\infty}\frac{\log \sharp \{go: d(o,go)\le n\}}{n} $$

Let $\e \Gamma, \e G$ denote the critical exponents of $\Gamma, G$ respectively.
The above result is of interest when  $\e \Gamma=\e G$. Recall that  $\Gamma/G$ is amenable if and only if $\e \Gamma=\e G$ (\cite[Theorem 1.1]{CDST}).

\begin{thm}
Suppose $\Gamma$ is a discrete    group acting on a Gromov hyperbolic space $X$. If $G$ is an infinite normal subgroup of infinite index, then 
$\HD (\Lambda^{nc}G) \ge \e G/\epsilon$. 
\end{thm}

\noindent \emph{Hyperbolic 3-manifolds:} In dimension 3, we prove the following result using the
model manifold technology of Minsky \cite{minsky-elc1} and Brock-Canary-Minsky \cite{minsky-elc2} as adapted by the first author in \cite{mahan-ibdd,mahan-split} (see Theorem~\ref{thm-kl-nonconical} and Corollary~\ref{cor--kl-nonconical}).

\begin{thm}\label{thm-kl-nonconical-intro}
Let $G<\isom(\mathbb H^3)$ be a finitely generated geometrically infinite Kleinian group. Then 
$$\HD (\Lambda^{nc}G) =2.$$  
\end{thm}

There are some precursors to Theorem~\ref{thm-kl-nonconical-intro} in the literature, all for bounded geometry manifolds. We say that $M=\Hyp^3/G$ has \emph{bounded geometry} if there exists $\ep >0$ such that any closed geodesic in $M$ has length bounded below
by $\ep$. In \cite{BJ97bddgeo}, Bishop and Jones proved that $\HD (\Lambda^{nc}G) =2$ provided that $G$ is a finitely generated geometrically infinite Kleinian group, $M=\Hyp^3/G$ has {bounded geometry}
and $\Lambda G \neq S^2$. This was sharpened by G\"{o}nye \cite{gonye}.
Kapovich and Liu \cite[Theorem 1.6]{KL20} proved that $\HD (\Lambda^{nc}G) >0 $  provided that $G$ is a finitely generated, non-free, torsion-free geometrically infinite Kleinian group
and  $M=\Hyp^3/G$ has {bounded geometry}. In \cite[Remark 1.8]{KL20}, the authors comment,
`It is very likely that the conclusion of this theorem can be strengthened
to $\HD (\Lambda^{nc}G) =2$, but proving this would require considerably more work.'

We deduce a number of consequences by combining Theorem~\ref{thm-kl-nonconical-intro} with existing theorems in the literature.
By work of Bishop and Jones \cite{BJ97}, if $G$ is geometrically infinite, then $\e G=2$.  Sullivan's formula  implies  that the  bottom of the spectrum for the Laplacian satisfies $\lambda_0(\mathbb H^3/G)=0$. Hence  the Cheeger constant $h(\mathbb H^3/G)$ of $\mathbb H^3/G$ is $0$ by the Cheeger-Buser inequality.  Hence by Theorem~\ref{thm-kl-nonconical-intro} we have the following.
\begin{cor}\label{cor-cheeger3-intro}
Let $M$ be a  complete hyperbolic 3-manifold with finitely generated fundamental group. Then the Cheeger constant $h(M)$ is equal to 0 if and only if $\HD (\Lambda^{nc}G) =2.$
\end{cor}

We set up some notation. Let $G$ be a Kleinian group and $M =\Hyp^3/G$.
As a consequence of Corollary~\ref{cor-cheeger3-intro}, we obtain the following trichotomy on geodesic flows on 3-dimensional hyperbolic manifolds.  
\begin{cor}\label{cor-trich-intro}
Let $M$ be a complete 3-dimensional hyperbolic manifold with finitely generated fundamental group. Then exactly one of the following statements hold  
\begin{enumerate}
    \item 
    $M$ has finite volume and there are only countably many escaping geodesic rays from any fixed point.
    \item 
    $M$ has infinite volume with $\Lambda G \subsetneq S^2$, and the set of   escaping geodesic rays has full Lebesgue measure.
    \item 
    $M$ has infinite volume, $\Lambda G = S^2$, and the set of   escaping geodesic rays has Hausdorff dimension 2 with null Lebesgue measure.
\end{enumerate}
\end{cor} 
Beardon-Maskit \cite{BeaMas} showed that	a complete hyperbolic 3-manifold $M$   is geometrically finite if and only if there are  countably many escaping and non-wondering geodesic rays starting from a fixed but arbitrary basepoint. This gives the first alternative. When $G$ is finitely generated and $\Lambda G \subsetneq S^2$, the Ahlfors measure zero theorem \cite{minsky-elc2} shows that $\Lambda G$ has zero Lebesgue measure. This 
gives the second alternative.
The main new content  in Corollary~\ref{cor-trich-intro} is  contained in item (3) which now follows from Corollary~\ref{cor-cheeger3-intro} and Theorem~\ref{thm-kl-nonconical-intro}. 

A word about the connection to the Hopf-Tsuji-Sullivan dichotomy on recurrent and escaping geodesics. This dichotomy says that generic  geodesic rays are either recurrent or escaping in the sense of the Bowen-Margulis-Sullivan measure on the geodesic flows.  These two possibilities correspond precisely to the  dichotomy of completely conservative/dissipative geodesic flows or equivalently to    the divergence/convergence of the   Poincar\'e series associated with the action of $G:=\pi_1(M)$ on $X:=\widetilde M$ 
$$
\forall s\ge 0,\; o\in  X:\quad \sum_{g\in G} \mathrm{e}^{-sd(o,go)} 
$$ at its {critical exponent}.

It follows from tameness of 3-manifolds \cite{gab-cal,agol-tameness} and earlier work of Thurston, Bonahon and Canary \cite{thurstonnotes,bonahon-bouts,canary} that the geodesic flow on $M$ is ergodic in the third case of Corollary~\ref{cor-trich-intro}. This is more generally true for the geodesic flow  restricted to the convex core of any $M$ with finitely generated fundamental group. Finally, there is an intimate connection between  ergodicity of the geodesic flow on $M$ and recurrence of Brownian motion on $M$ (see \cite{ls} for instance).\\

\noindent \emph{Hyperbolic 2-manifolds and trees:} 
This leads us to a similar trichotomy  for hyperbolic surfaces
that was proved by Fernandez-Melian \cite{FM01}.  The key result they proved was that if $\Sigma$ is a hyperbolic surface with recurrent Brownian motion and infinite area, then the Hausdorff dimension of non-conical points is 1. By \cite[Theorem 2.1]{HP97},  the Brownian motion  is recurrent on $\Sigma$ if and only if  $\pi_1(\Sigma)\act \widetilde \Sigma$ is of divergent type with critical exponent 1. The same conclusion holds for  rank-1 locally symmetric manifolds and trees.

We call a Riemannian manifold amenable if its Cheeger constant is 0. This is consistent with the terminology for amenable graphs;  equivalently  the graph admits a Folner sequence. 
Our methods prove the following,  improving  Fernandez-Melian's result, see Theorem~\ref{thm-cheeger-surface}.
\begin{thm}\label{thm-cheeger-surface-intro}
Let $\Sigma$ be a hyperbolic surface with possibly infinitely generated fundamental group. If $\Sigma$ is amenable, then the Hausdorff dimension of non-conical points is 1. 
\end{thm}
It is not hard to construct an amenable hyperbolic surface with transient Brownian motion. For instance, we cut out half of a cyclic cover of a closed surface and then glue a funnel along the resulted boundary. It is clearly amenable by computing $h=0$, and the existence of the funnel makes the Poincar\'e series  convergent at $1$. 

An analog for groups acting on trees seems not be recorded in literature, see Theorem~\ref{thm-amenable-graph}.
\begin{thm}\label{thm-amenable-graph-intro}
Let $G$ be a discrete group acting on a $d$-regular tree $X$ with $d\ge 3$ so that the quotient graph is amenable. Then the Hausdorff dimension of non-conical  points for $G$ is $\log(d-1)$.   
\end{thm}


\subsection{Statement of results: Myrberg geodesics}
We now turn to the Myrberg limit set.
Our first general result is as follows. Let $\epsilon$ be the parameter for the visual metric on the Gromov boundary of a hyperbolic space $X$, see Theorem~\ref{MyrbergHdimHypThm}. 
\begin{thm}\label{MyrbergHdimThm-intro}
Let $X$ be  a Gromov hyperbolic space equipped with a  proper and non-elementary action of $G$.  Then, 
$$
\HD (\cG) = \HD(\uG)=\HD(\mG)=\e G/\epsilon.
$$
\end{thm}

By definition, the uniformly conical point set $\uG$ is disjoint from the Myrberg limit set $\mG$ unless the action $G\act X$ is convex-cocompact. The equality $\HD(\uG)=\HD(\mG)$ was conjectured  by Falk-Matsuzaki \cite[Conjecture 2]{FM20}, where they confirmed it for Kleinian groups with finite Bowen-Margulis-Sullivan (BMS) measure on the geodesic flow.   Their proof relies on a conjecture of  Sullivan \cite[after Corollary 19]{Sul} about generic sublinear limit sets. The conjecture is known to be true when the BMS measure is finite (\cite[Corollary 19]{Sul}). If Sullivan's conjecture is true for any divergent action, then Theorem \ref{MyrbergHdimThm-intro} would  follow from it in this case. Thus the above result could be thought of as positive  evidence for  Sullivan's conjecture.   

Combining Theorems \ref{MyrbergHdimThm-intro} and \ref{thm-kl-nonconical-intro}, we obtain the following fact about the limit set of Kleinian groups.
\begin{cor}\label{cor-geominf-limitset-intro}
Let $G$ be a  finitely generated geometrically infinite Kleinian group. Then the uniformly conical limit set, the Myrberg limit set and the non-conical limit set ($\uG$, $\mG$ and $\ncG$ respectively) are mutually disjoint, and have  the same Hausdorff dimension 2.
\end{cor}

Myrberg limit points could be defined more generally for a convergence group action  on a compact metric space (Definition \ref{MyrbergDefn}). 
Our method in proving the above theorem is rather  general  and in particular allows us to compute the Hausdorff dimension of the Myrberg limit set in the Floyd boundary. 

In \cite{Floyd}, Floyd introduced a way of compactifying any infinite locally finite graph $\Gamma$. Fixing a parameter $0<\lambda<1$ and a basepoint $o\in \Gamma$, one assigns each edge $e$ a new length $\lambda^n$ with $n=d(o,e)$. The induced length metric $\rho_\lambda$ on $\Gamma$ is called the \emph{Floyd metric}. It  is incomplete. We take the Cauchy completion $\overline \Gamma$. Then $\partial_\lambda \Gamma:=\overline\Gamma\setminus \Gamma$ is called 
the \emph{Floyd boundary}. The Floyd boundary $\partial_\lambda \Gamma$ can be equipped with a natural \emph{Floyd metric} as well.
We say that $\partial_\lambda \Gamma$ is non-trivial if $|\partial_\lambda \Gamma|>3$. If $\Gamma$ is Gromov hyperbolic, then the visual metric on $\partial \Gamma$ is bi-Lipschitz homeomorphic to $\partial_\lambda \Gamma$ equipped the Floyd metric with $\epsilon=-\log\lambda$. If $\Gamma$ is the Cayley graph of a finitely generated group $G$ and $|\partial_\lambda \Gamma|>3$,  Karlsson \cite{Ka} proved that the action of $G$ on its Floyd boundary $\partial_\lambda \Gamma$ is a convergence group action. We have the following (see Theorem~\ref{MyrbergHdimFloyd}).  
\begin{thm}\label{MyrbergHdimFloyd-intro}
Let $\Gamma$ be the Cayley graph of a group $G$ with a finite generating set $S$. Let $\e G$ be the critical exponent of the action of $G$ on $\Gamma$. Assume that $|\partial_\lambda \Gamma|>3$  for parameter $\lambda$. Then the Hausdorff dimension of the Myrberg limit set has full dimension $\frac{\e G}{(-\log\lambda)}$.    
\end{thm}

Gerasimov proved that the Floyd boundary of any non-elementary relatively hyperbolic group is nontrivial (\cite{Ge2}). In \cite{PYANG}, Potyagailo and the second author showed that for these groups, the Hausdorff dimension of $\partial_\lambda\Gamma$ is $\frac{\e G}{(-\log\lambda)}$.    Indeed, in \cite{PYANG}, the  dimension is computed precisely for uniformly conical points. Thus, Theorem \ref{MyrbergHdimFloyd-intro} is complementary to the main results of \cite{PYANG}. 

We give a brief history about the problem dealing with Hausdorff dimensions of limit sets of Kleinian groups. In 1971, Beardon \cite{Beardon} proved that the critical exponent gives an upper bound on the Hausdorff dimension for any finitely generated Fuchsian group. The lower bound was later established by  Patterson \cite{Patt} in 1976. In this work, Patterson introduced  what we  now call the Patterson-Sullivan (PS) measures in the critical dimension on the limit set. He identified PS measures with Hausdorff measures when  the Fuchsian group has no parabolic elements.  Subsequently, Sullivan \cite{Sul} generalized this to geometrically finite Kleinian groups. In \cite{BJ97}, Bishop and Jones proved that the Hausdorff dimension of the conical limit set equals the critical exponent for any finitely generated Kleinian group. This  generalized Patterson and Sullivan's works, where the groups  were geometrically finite, and hence   contain only countably many non-conical (parabolic) points.  Bishop-Jones' techniques are very general and were developed by many authors later on \cite{Paulin2,FSU}, to prove Hausdorff dimension results   for uniformly conical points.  The corresponding result for  Myrberg limit sets, i.e.\ for non-uniformly conical limit sets, remained open. Theorem \ref{MyrbergHdimThm-intro} completes the picture for non-uniformly conical limit sets. This is new even for Kleinian groups \cite[Conjecture 2]{FM20}.

\subsection{Proof ingredients: quasi-radial trees, amenability, and geometric limits}
To address Question \ref{qn-max} on the maximal Hausdorff dimension, we focus on finding a lower bound. 
Curiously, though non-conical  and Myrberg limit sets  are complementary, the strategy in getting the correct lower bound  is similar.  A key tool is the following notion of 
a \emph{quasi-radial tree}.
\begin{defn}\label{def-qrtree-intro}
A rooted metric tree $(\TT, v_0, d_\TT)$  is said to be  \emph{quasi-radially embedded} in a geodesic metric space $(X,d)$ via $\Phi$, if $\Phi: \TT \to X$ is injective and satisfies  the following. 
There exists $c \geq 1$ such that $\Phi|_{[v_0,v]}$ is a $c-$quasi-geodesic for every vertex $v$ of $\TT$.
We refer to the image of $\Phi$ as a \emph{quasi-radial tree}.
\end{defn}
Let $X$ be a Gromov hyperbolic space. The Gromov boundary of  $\TT$ is  a Cantor set $\partial \TT$. Let $T=\Phi(\TT)$.
We shall provide criteria such that $\Phi: \TT \to T \subset X$
extends continuously to give an embedding of $\partial \TT$
 in $\pU$. Further, we shall obtain 
  a lower bound on $\HD(\partial T)$. Towards this,  we construct $T$ from the following prescribed data: 
\begin{enumerate}
    \item a sequence of integers $K_n$ called \emph{repetitions}, and  a divergent sequence of real numbers $L_n>0$.
    \item   A sequence of finite sets $A_n$   with  $|A_n|\ge \mathrm{e}^{L_n \omega_n}$, where $\omega_n\to \HD(\Lambda)$. 
    \item 
     A sequence $b_n$ of arcs in $X$ called \emph{bridges}. Let $B_n>0$ be the length of $B_n$.
\end{enumerate}
The quasi-radial tree is constructed inductively  in two stages
(see Figure~\ref{fig:loop-bridge-intro}).    
\begin{enumerate}
    \item[\textbf{Step 1}] For each set $A_n$, we choose  $K_n$ elements $(a^{(1)},\cdots, a^{(K_n)})$ in $A_n$ and concatenate them in order. 

\item[\textbf{Step 2}] We  append the bridge $b_n$ to the resulting word in Step (1), and then repeat Step 1 for $A_{n+1}$. 
\end{enumerate}
More precisely, we consider the set $\mathcal W$ of words of the form  $$W=\prod_{i=1}^{K_1}a_1^{(i)} b_1 \prod_{i=1}^{K_2}a_2^{(i)}  b_2 \cdots \prod_{i=1}^{K_m}a_m^{(i)}  b_m,$$ where
$a_j^{(i)} \in A_2$ and each $b_j$ is a bridge.
Let $v_0$ denote the empty word. We construct naturally a  tree $\mathcal T$  rooted at $v_0$  with $\mathcal W$ as its vertex set. Endow   $\mathcal T$ with a metric $d_\TT$ so that the edges corresponding to $a\in A_n$ are assigned length $L_n$ and the edge  $b_n$ is assigned length $B_n$.  

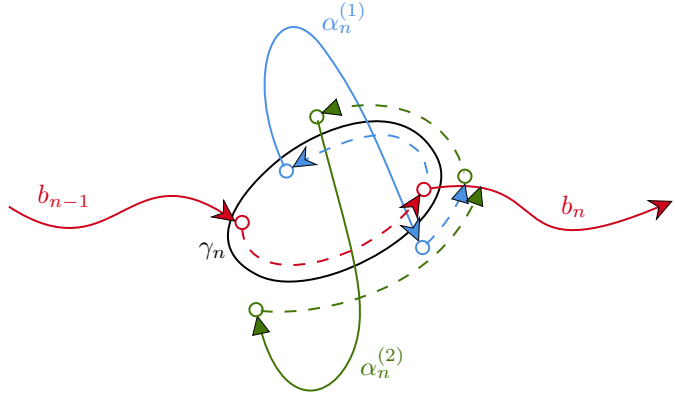
\begin{figure}
    \centering

\tikzset{every picture/.style={line width=0.75pt}} 

\begin{tikzpicture}[x=0.75pt,y=0.75pt,yscale=-1,xscale=1]

\draw  [line width=0.75]  (229.67,114.67) .. controls (248.67,99.33) and (291.67,81.33) .. (308.67,113.33) .. controls (325.67,145.33) and (249.67,188.67) .. (218.67,173.33) .. controls (187.67,158) and (210.67,130) .. (229.67,114.67) -- cycle ;
\draw [color={rgb, 255:red, 74; green, 144; blue, 226 }  ,draw opacity=1 ]   (231.95,118.51) .. controls (209.58,71.32) and (228.67,27.79) .. (251,58) .. controls (273.21,88.05) and (281.34,113.62) .. (298.6,153.54) ;
\draw [shift={(299.67,156)}, rotate = 246.47] [fill={rgb, 255:red, 74; green, 144; blue, 226 }  ,fill opacity=1 ][line width=0.08]  [draw opacity=0] (10.72,-5.15) -- (0,0) -- (10.72,5.15) -- (7.12,0) -- cycle    ;
\draw [shift={(233,120.67)}, rotate = 243.43] [color={rgb, 255:red, 74; green, 144; blue, 226 }  ,draw opacity=1 ][line width=0.75]      (0, 0) circle [x radius= 3.35, y radius= 3.35]   ;
\draw [color={rgb, 255:red, 208; green, 2; blue, 27 }  ,draw opacity=1 ]   (305.51,129.49) .. controls (372.84,119.52) and (337.44,174.93) .. (424.65,137.17) ;
\draw [shift={(427.33,136)}, rotate = 156.15] [fill={rgb, 255:red, 208; green, 2; blue, 27 }  ,fill opacity=1 ][line width=0.08]  [draw opacity=0] (10.72,-5.15) -- (0,0) -- (10.72,5.15) -- (7.12,0) -- cycle    ;
\draw [shift={(302.33,130)}, rotate = 350.32] [color={rgb, 255:red, 208; green, 2; blue, 27 }  ,draw opacity=1 ][line width=0.75]      (0, 0) circle [x radius= 3.35, y radius= 3.35]   ;
\draw [color={rgb, 255:red, 65; green, 117; blue, 5 }  ,draw opacity=1 ]   (219.14,197.53) .. controls (230.01,242.14) and (255.93,238.02) .. (266.33,212.67) .. controls (276.89,186.93) and (263.28,153.35) .. (248.77,95.11) ;
\draw [shift={(248.33,93.33)}, rotate = 256.12] [color={rgb, 255:red, 65; green, 117; blue, 5 }  ,draw opacity=1 ][line width=0.75]      (0, 0) circle [x radius= 3.35, y radius= 3.35]   ;
\draw [shift={(218.33,194)}, rotate = 77.93] [fill={rgb, 255:red, 65; green, 117; blue, 5 }  ,fill opacity=1 ][line width=0.08]  [draw opacity=0] (8.93,-4.29) -- (0,0) -- (8.93,4.29) -- cycle    ;
\draw [color={rgb, 255:red, 208; green, 2; blue, 27 }  ,draw opacity=1 ]   (93,138.67) .. controls (147.18,173.8) and (152.84,108.67) .. (204.6,144.3) ;
\draw [shift={(207,146)}, rotate = 216.07] [fill={rgb, 255:red, 208; green, 2; blue, 27 }  ,fill opacity=1 ][line width=0.08]  [draw opacity=0] (10.72,-5.15) -- (0,0) -- (10.72,5.15) -- (7.12,0) -- cycle    ;
\draw [color={rgb, 255:red, 74; green, 144; blue, 226 }  ,draw opacity=1 ] [dash pattern={on 4.5pt off 4.5pt}]  (239,117.07) .. controls (256.3,106.98) and (280.93,97.73) .. (292.33,105.33) .. controls (304.33,113.33) and (301,112) .. (305.67,127.33) ;
\draw [shift={(236.33,118.67)}, rotate = 328.39] [fill={rgb, 255:red, 74; green, 144; blue, 226 }  ,fill opacity=1 ][line width=0.08]  [draw opacity=0] (10.72,-5.15) -- (0,0) -- (10.72,5.15) -- (7.12,0) -- cycle    ;
\draw [color={rgb, 255:red, 74; green, 144; blue, 226 }  ,draw opacity=1 ] [dash pattern={on 4.5pt off 4.5pt}]  (303.47,157.64) .. controls (310.07,151.31) and (318.1,141.68) .. (322.15,130) ;
\draw [shift={(323,127.33)}, rotate = 105.95] [fill={rgb, 255:red, 74; green, 144; blue, 226 }  ,fill opacity=1 ][line width=0.08]  [draw opacity=0] (8.93,-4.29) -- (0,0) -- (8.93,4.29) -- cycle    ;
\draw [shift={(301.67,159.33)}, rotate = 317.73] [color={rgb, 255:red, 74; green, 144; blue, 226 }  ,draw opacity=1 ][line width=0.75]      (0, 0) circle [x radius= 3.35, y radius= 3.35]   ;
\draw [color={rgb, 255:red, 65; green, 117; blue, 5 }  ,draw opacity=1 ] [dash pattern={on 4.5pt off 4.5pt}]  (254.17,90.35) .. controls (282.81,82.14) and (314.62,90.81) .. (322.55,121.43) ;
\draw [shift={(323,123.33)}, rotate = 77.59] [color={rgb, 255:red, 65; green, 117; blue, 5 }  ,draw opacity=1 ][line width=0.75]      (0, 0) circle [x radius= 3.35, y radius= 3.35]   ;
\draw [shift={(251,91.33)}, rotate = 341.57] [fill={rgb, 255:red, 65; green, 117; blue, 5 }  ,fill opacity=1 ][line width=0.08]  [draw opacity=0] (8.93,-4.29) -- (0,0) -- (8.93,4.29) -- cycle    ;
\draw [color={rgb, 255:red, 65; green, 117; blue, 5 }  ,draw opacity=1 ] [dash pattern={on 4.5pt off 4.5pt}]  (220.08,191.21) .. controls (236.93,193.85) and (306.81,188.68) .. (329.97,131.32) ;
\draw [shift={(331,128.67)}, rotate = 110.22] [fill={rgb, 255:red, 65; green, 117; blue, 5 }  ,fill opacity=1 ][line width=0.08]  [draw opacity=0] (8.93,-4.29) -- (0,0) -- (8.93,4.29) -- cycle    ;
\draw [shift={(217.67,190.67)}, rotate = 18.43] [color={rgb, 255:red, 65; green, 117; blue, 5 }  ,draw opacity=1 ][line width=0.75]      (0, 0) circle [x radius= 3.35, y radius= 3.35]   ;
\draw [color={rgb, 255:red, 208; green, 2; blue, 27 }  ,draw opacity=1 ] [dash pattern={on 4.5pt off 4.5pt}]  (210.67,149.41) .. controls (212.42,182.93) and (276.22,167.3) .. (298.69,135.17) ;
\draw [shift={(300.33,132.67)}, rotate = 121.61] [fill={rgb, 255:red, 208; green, 2; blue, 27 }  ,fill opacity=1 ][line width=0.08]  [draw opacity=0] (10.72,-5.15) -- (0,0) -- (10.72,5.15) -- (7.12,0) -- cycle    ;
\draw [shift={(210.67,146.67)}, rotate = 92.51] [color={rgb, 255:red, 208; green, 2; blue, 27 }  ,draw opacity=1 ][line width=0.75]      (0, 0) circle [x radius= 3.35, y radius= 3.35]   ;

\draw (106,126.07) node [anchor=north west][inner sep=0.75pt]    {$\textcolor[rgb]{0.82,0.01,0.11}{b_{n-1}}$};
\draw (370,131.4) node [anchor=north west][inner sep=0.75pt]    {$\textcolor[rgb]{0.82,0.01,0.11}{b_{n}}$};
\draw (187.33,155.4) node [anchor=north west][inner sep=0.75pt]    {$\gamma _{n}$};
\draw (249.33,34.07) node [anchor=north west][inner sep=0.75pt]    {$\textcolor[rgb]{0.29,0.56,0.89}{\alpha _{n}^{( 1)}}$};
\draw (268.67,208.07) node [anchor=north west][inner sep=0.75pt]  [color={rgb, 255:red, 65; green, 117; blue, 5 }  ,opacity=1 ]  {$\alpha _{n}^{( 2)}$};

\end{tikzpicture}    
    \caption{Looping  with $K_n=2$ and bridging. We slide the endpoints of shortest arcs $\alpha_n^{(i)}$ on $\gamma_n$, and the terminal point of $b_{n-1}$ to $(b_n)_-$. }
    \label{fig:loop-bridge-intro}
\end{figure}
Depending on the specific setup, the proof will proceed by finding a sequence $A_n$ and  quasi-radially embedding $\mathcal T$ into $X$. The idea of constructing quasi-radially embedded trees (in our sense) first appeared in work of Bishop and Jones \cite{BJ97} to give a lower bound on the Hausdorff dimension of uniformly conical points for Kleinian groups. It was later adapted by Fern\'{a}ndez and Meli\'{a}n \cite{FM01} to study non-conical points in Fuchsian groups with recurrent Brownian motion (cf. Theorem \ref{thm-cheeger-surface-intro}). Our work is particularly inspired by the construction in \cite{FM01,MRT19} and generalizes its key aspects to a broader setup.  \\

\noindent \emph{Non-conical points}. 
Let $M$ be a regular cover as in Theorem \ref{nonconical-normal-covering-intro} or let $M=\mathbb H^3/G$ be a geometrically infinite 3-manifold as in Theorem \ref{thm-kl-nonconical-intro}.  We shall  find a sequence of \emph{oriented} escaping closed geodesics $\gamma_n$ on $M$, and construct $A_n$ from these. Further,
the bridge $b_n$ will be a  shortest arc from $\gamma_n$ to $\gamma_{n+1}$. Choose the set $A_n$ of \emph{oriented} shortest arcs from $\gamma_n$ to itself. We may slide the starting and terminal points of each arc in $A_n$ (respecting the orientation on $\gamma_n$) to the starting point of bridge $b_n$. A \emph{$K_n$-looping}  in $A_n$  means a concatenation of $K_n$ arcs in $A_n$ following   their orientation.    The construction of $\mathcal T$ is best  carried out in $M$ itself: we take a $K_n$-looping  in $A_n$, and then pass though $b_n$ to  $\gamma_{n+1}$ where we do the next $K_{n+1}$-looping. In the end,   lifting of all so-produced paths gives the  desired quasi-radial tree $\mathcal T$.  See Figure \ref{fig:loop-bridge-intro} for illustrating the construction.

In the setup of Theorem \ref{nonconical-normal-covering-intro}, finding  $\gamma_n$ and  corresponding shortest arcs $A_n$ with length about $L_n$ is relatively straightforward. We simply use the deck transformations of $\Gamma/G$ acting on $M$. We deduce 
the cardinality lower bound $|A_n|\ge \mathrm{e}^{L_n \omega_n}$  from the following counting result that may be of independent interest, see Lemma~\ref{ShortestArcsonMfd}. 
\begin{lem}\label{ShortestArcsonMfd-intro}
Let $M$ be a complete Riemannian manifold with pinched negative  curvature. Let $\e G$ be the critical exponent for the action of $G:=\pi_1(M)$
on $\widetilde M$. Let $\gamma$ be a closed geodesic on $M$.  Then there exist  $c, \Delta>0$ depending on $\gamma$ so that the following holds. Let  $\mathrm{Arc}(\gamma,t,\Delta)$ denote the
collection of shortest arcs
 from $\gamma$ to $\gamma$  with length in  $[t-\Delta,t+\Delta]$. Then for any $\epsilon>0$, and for all large $t>0$, 
$$
|\mathrm{Arc}(\gamma,t,\Delta)|\ge c\mathrm{e}^{(\e G-\epsilon) t}.
$$    
\end{lem}
We remark that, when $M$ is geometrically finite, a precise counting of shortest arcs has been well-studied in literature (see survey \cite{PP16}) and the above one follows from it in this case. In our applications, however, we need to consider geometrically infinite manifolds.  

In Theorem \ref{thm-kl-nonconical-intro}, if  $M=\mathbb H^3/G$ is a  general geometrically infinite hyperbolic 3-manifold, locating the desired sets $A_n$ of shortest arcs directly is quite 
subtle. We will use an indirect  approach based on the
model manifold technology of Minsky \cite{minsky-elc1} and Brock-Canary-Minsky \cite{minsky-elc2} as adapted by the first author in \cite{mahan-ibdd,mahan-split}. We prove  the following result, see Theorem \ref{thm-kl-nonconical}. We refer to a complete hyperbolic manifold $M^h$ minus a small neighborhood of its cusps as the truncation $M$ of $M^h$.
\begin{thm} 
	Let $\Gamma$  denote a finitely generated geometrically infinite Kleinian group, $M^h = \Hyp^3/\Gamma$, and  $M$ denote the associated truncated 3-manifold. Then there exists an unbounded sequence of points $x_n \in M$, such that $(M,x_n)$ converges geometrically to  a geometrically infinite 
	truncated hyperbolic 3-manifold $N$.  Further, if $\Gamma_\infty$ is the associated Kleinian group, then the limit set of $\Gamma_\infty$ is the entire 2-sphere.
\end{thm}

Indeed, fixing a closed geodesic $\tilde \gamma$ on $N$ we apply 
Theorem \ref{ShortestArcsonMfd-intro} to
obtain adequate shortest arcs $\tilde A$ in $N$ with end-points on $\tilde \gamma$. Finally, using the fact that $N$ is a  geometric limit, we pull back $\tilde A$ to a sequence of shortest arcs $A_n$ on $M$. This furnishes the estimates on $|A_n|$ as we wanted. We summarize this geometric limit argument in a general Theorem \ref{NonConicalFromGeometricLimit}.\\

\noindent \emph{Completion of the proofs of Theorem \ref{nonconical-normal-covering-intro} and Theorem \ref{thm-kl-nonconical-intro}}:\\
Let $\Lambda=\Lambda\Gamma$ or $\Lambda=\partial \mathbb H^3=S^2$.
Recall we use the parameters $(L_n,B_n,K_n)$ to construct the quasi-radial tree $\Phi: \mathcal T\to X$ and $\omega_n\to \HD(\Lambda)$. The repetitions $K_n$ are of primary importance.
In practice, the bridge length $B_n$ is typically not  fixed  at
the outset. In  the course of the construction,  we will have to choose $K_n$ large enough to compensate the effect 
of $B_n$ on the critical exponent of $T$. The technical  Lemma \ref{QuasiRadialTree} and Lemma \ref{HDLargeTree} show that $\HD(\partial T)=\HD(\Lambda)$. Moreover, one can verify that each infinite radial ray in $T$ projects to an escaping geodesic in $M$. Thus $\partial T$ consists of non-conical points. This completes the proof of Theorem \ref{nonconical-normal-covering-intro} or Theorem \ref{thm-kl-nonconical-intro}. See  Corollary \ref{cor-normalsubgp} and Theorem \ref{NonConicalFromGeometricLimit} for details.\\

\noindent \emph{About the proofs of Theorem \ref{thm-cheeger-surface-intro} and Theorem \ref{thm-amenable-graph-intro}}:\\
Now, let $M$ denote a hyperbolic surface or a $d$-regular graph $X/G$. Amenability of  $M$ enters  the proofs at the stage where $A_n$'s are constructed. The Folner sequence characterization of 
amenability ensures that $M$  contains a sequence $S_n$ of compact subsurfaces or subgraphs with $vol(\partial S_n)/vol(S_n)\to 0$. The 2-dimensional or 1-dimensional geometry of $M$ allows  us to complete $S_n$ to obtain a \textit{geometrically finite}  surface or a $d$-regular graph $\tilde S_n$ with \textit{finite core} respectively. 

The inequality $vol(\partial S_n)/vol(S_n)\to 0$ in conjunction with the Patterson formula (\ref{EPSC}) or the Grigorchuk co-growth formula (\ref{Grigorchuk}) implies that the critical exponent of $\tilde S_n$ tends to $1$ or $\log(d-1)$ respectively. Finally, we construct $A_n$  in $\tilde S_n$ with the desired estimates using Lemma \ref{ShortestArcsonMfd-intro} or the analog Lemma \ref{ShortestArcsonGraphs} in graphs. The rest of the proof  is completed exactly as  above. See Theorems \ref{thm-cheeger-surface} and \ref{thm-amenable-graph} for details.

In Example \ref{eg-2dsurf} we construct   an infinite type surface $\Sigma$  with zero Cheeger constant, so that Theorem \ref{thm-cheeger-surface-intro} applies. However, a geometric limit argument as  in Theorem \ref{thm-kl-nonconical-intro}  fails: any geometric limit $(\Sigma,x_n)$ with unbounded $x_n\in \Sigma$ is the hyperbolic plane.
\\

\noindent \emph{Myrberg limit set.} Let $X$ be the Gromov hyperbolic space in Theorem  \ref{MyrbergHdimThm-intro}. We perform a similar construction of a quasi-radial tree $\Phi: \TT\to X$. But the scenario is much simpler.

Here, $A_n$ is given by the annular set $A(n,\Delta,o):=\{go: |d(o,go)-L_n|\le \Delta\}$. The estimates $|A_n|\ge \mathrm{e}^{L_n\omega_n}$ follow immediately from the definition of $\omega_G$. The bridges $b_n$ are given by the set of \emph{all} loxodromic elements in $G$ in some order. We do not need to repeat the looping, i.e.\ $K_n=1$
for all $n$. So the quasi-radial tree  $\TT$ is constructed from the set $\mathcal W$ of words of the form  $$W=a_1  b_1 a_2   b_2 \cdots a_m   b_m.$$
By the characterization of Myrberg limit points (Lemma \ref{lem:CharMyrberg}), each radial ray in $\TT$ labeled by $W$ will terminate at a Myrberg point.
This is because $W$ contains every loxodromic element as 
a subword. This proves that  the  quasi-radial tree $T$ accumulates to Myrberg points in $\partial X$. Lemma \ref{QuasiRadialTree} and Lemma \ref{HDLargeTree} concludes the proof of Theorem  \ref{MyrbergHdimThm-intro}; see   Theorem \ref{MyrbergHdimHypThm} for details.  

It turns out that the above sketch   works for  groups with   contracting elements. This class of groups
includes relatively hyperbolic groups, groups with rank-1 elements and mapping class groups. Hence,  Theorem \ref{MyrbergHdimFloyd-intro} on  the Myrberg limit set in the Floyd boundary is proved along the same lines with somewhat different ingredients; see Theorem \ref{MyrbergHdimFloyd} for details. 

In Section~\ref{sec-contracting} we 
carry out   the above construction for Myrberg limit sets for actions on general metric spaces with contracting elements; see Theorem \ref{thm-qrtree-Myrberg-general}.  To end the introduction, let us mention a sample application to mapping class  groups, see Theorem~\ref{MyrberginMCG}.

\begin{thm}\label{MyrberginMCG-intro}
Let $G=\mcg$ denote the mapping class group of a closed orientable surface $\Sigma_g$ with $g\ge 2$. Consider the proper action of $G$ on the Teichm\"uller space $\T_g$. Fix a point $o\in \T_g$. Then there exists a quasi-radial tree $T$ rooted at $o$ with vertices contained in $Go$ so that $\omega_T=(6g-6)$ and all accumulation points of $T$ in the Thurston boundary  consists of Myrberg limit points.     
\end{thm}

\noindent\textbf{Organization of the paper}. The paper is organized as follows. Section \ref{sec-prel} introduces the basics of  Gromov hyperbolic spaces, and discusses various classes of limit points  with their relation to geodesic rays. In Section \ref{sec-hausd} we develop  general procedures  to build a quasi-radial tree  from group actions (\textsection\ref{subsec-qtreegroup}) and from prescribed patterns (\textsection\ref{subsec-qtreepattern}). Section \ref{sec-arcs} provides another ingredient on counting shortest arcs between geodesics.  Sections \ref{sec-nonconical} and \ref{sec-nonconical-kl}  are the bulk of the paper. In \textsection\ref{sec-nonconical}, we explain the concrete realization of constructions given in Section \ref{sec-hausd} on Riemannian manifolds and Gromov hyperbolic spaces:  Theorem \ref{nonconical-normal-covering-intro} for normal coverings, Theorem \ref{thm-cheeger-surface-intro} for surfaces and Theorem \ref{thm-amenable-graph-intro} for graphs are proved.  Section \ref{sec-nonconical-kl} is devoted to the proof of Theorem \ref{thm-kl-nonconical-intro} in Kleinian groups. In last two sections, the Hausdorff dimension of Myrberg limit sets are computed  on the Gromov boundary of hyperbolic spaces (Theorem \ref{MyrbergHdimThm-intro}), and  on the Floyd boundary of finitely generated groups (Theorem \ref{MyrbergHdimFloyd-intro}). The latter contained in Section \ref{sec-contracting}  is proved by generalizing Section \ref{sec-hausd} to groups with  contracting elements, which also have applications to mapping class groups in Theorem \ref{MyrberginMCG-intro}.  

\subsection*{Acknowledgments}
We are grateful for helpful discussions with Xiaolong Han, Beibei Liu and Tianyi Zheng. MM is partly supported by a DST JC Bose Fellowship,  the Department of Atomic Energy, Government of India, under project no.12-R\&D-TFR-5.01-0500, and by an endowment of the Infosys Foundation. MM also acknowledges support of the Institut Henri Poincare (UAR 839 CNRS-Sorbonne University) and LabEx CARMIN, grant number ANR-10-LABX-59-01.
WY is partially supported by National Key R \& D Program of China (SQ2020YFA070059) and  National Natural Science Foundation of China (No. 12131009 and No. 12326601).

\section{Preliminaries}\label{sec-prel}
Let $(X,d)$ be a metric space.  A geodesic segment in $X$ is an isometrically embedded closed interval. Geodesic rays and  bi-infinite geodesics are isometrically embedded  copies of $[0,\infty)$ and $(-\infty, \infty)$ respectively. The space $X$ is geodesic if every pair of points in $X$ can be joined by a possibly non-unique geodesic segment. For $x, y \in X$, $[x, y]$ will denote a geodesic segment between  $x$ and $y$.   

\begin{defn}\label{def-hyp}
A geodesic metric space $X$ is (Gromov) \textit{hyperbolic}  if there exists  $\delta\ge 0$ so that for  $x, y, z \in X$, $[x, y] \subseteq N_\delta ([y,z] \cup [z,x])$.    
\end{defn}

\begin{defn}\label{def-qgeo}
Given  $c\geq 1$, a map between two metric spaces $f: (X,d_X)\rightarrow (Y,d_Y)$ is called a \textit{$c$-quasi-isometric embedding} if the following  holds
\begin{equation*}
\frac{d_X(x,x')}{c}-c\leq d_Y(f(x),f(x'))\leq cd_X(x,x')+c,
\end{equation*}
for all $x,x'\in X$. Furthermore, if there exists $R>0$ such that $Y\subset N_R(f(X))$, then $f$ is called a \textit{$c$-quasi-isometry}.

More generally, given $K \geq 1, \ep \geq 0$, $f: (X,d_X)\rightarrow (Y,d_Y)$ is  a \textit{$(K,\ep)$-quasi-isometric embedding} if 
\begin{equation*}
	\frac{d_X(x,x')}{K}-\ep\leq d_Y(f(x),f(x'))\leq Kd_X(x,x')+\ep.
\end{equation*}
\end{defn}

A $c$-quasi-isometric embedding $\gamma: I \subseteq (-\infty,+\infty)\to X$ of an interval $I$ into $X$ shall be called \textit{$c$-quasi-geodesic}. Similarly, a $(K,\ep)$-quasi-isometric embedding $\gamma: I \subseteq (-\infty,+\infty)\to X$  shall be called \textit{$(K,\ep)$-quasi-geodesic}.
 Since $\gamma$ is not necessarily continuous, we actually work with a continuous version of quasi-geodesics. A path $\gamma$ is  a (continuous) \textit{$c$-quasi-geodesic} for some $c\ge 1$ if any finite subpath $\beta$ is rectifiable and $\ell(\beta)\le cd(\beta_-,\beta_+)+c$. If $\gamma: I\to X$ is given by arc-length parametrization, then it is a $c$-quasi-isometric embedding. Conversely, one could construct a continuous quasi-geodesic from the image $\gamma(I)$ of a $c$-quasi-isometric embedding in a finite neighborhood. In what follows,  the two notions are used  interchangeably without explicit mention.     Recall \cite[Ch. III.H]{bridson-haefliger} that hyperbolicity for geodesic metric spaces is invariant under quasi-isometry.

\begin{lem}
Suppose $X$ is  $\delta$-hyperbolic. Then, given $c \geq 1$ there exists $D=D(\delta,c)$ such that  any two $c$-quasi-geodesics  with the same endpoints are contained in a $D$-neighborhood of each other.    
\end{lem}

A path is called an \textit{$L$-local $c$-quasi-geodesic} if any subpath of length $L$ is a  $c$-quasi-geodesic.  
\begin{lem}\cite[Ch. III.H, Thm 1.13]{bridson-haefliger}\label{localtoglobal}
For any $\tau\ge 1$ there exist $L=L(\tau,\delta)$ and $c=c(\tau,\delta)$  so that any $L$-local $\tau$-quasi-geodesic is a $c$-quasi-geodesic.     
\end{lem}

For any $x,y,z\in X$, the  \textit{Gromov product} $\langle x,z\rangle_y$   is given by the following.
\begin{equation*}
\langle x,z\rangle_y=\frac{d(x,y)+d(y,z)-d(x,z)}{2}.
\end{equation*}

Two geodesic rays $r_1, r_2 : [0,+\infty) \to X$ are said to be \textit{asymptotic} if
$$\sup_{t\in [0,+\infty)} d(r_1(t), r_2(t)) < \infty$$
The \emph{Gromov boundary} $\partial X$ of $X$ consists of all asymptotic classes of geodesic rays. 
It is endowed with the topology induced by the topology of uniform convergence on  compact subsets of $[0,+\infty)$. The group $Isom(X)$ acts on $\pU$ by homeomorphisms.
If $X$ is a proper metric space, then $\partial X$ with the above topology is compact. Moreover, it is 
a visibility space: any two distinct points $\xi,\eta$ in $\pU$ are connected by a bi-infinite geodesic denoted by $[\xi,\eta]$, i.e.\ $[\xi,\eta]$ is the union of two geodesic half rays asymptotic to $\xi,\eta$.

We now endow Gromov boundary with a family of visual metrics $\rho^o_\epsilon$
\cite[p. 434-6]{bridson-haefliger}. The visual metrics $\rho^o_\epsilon$
depend on a basepoint $o\in X$ and a (small) parameter $\epsilon$.
\begin{lem}\label{VisualMetric} Given  $\delta \geq 0$
there exists $\epsilon_0>0$ such that for all  $\epsilon\in (0,\epsilon_0)$, there exists a visual metric $\rho_o$ on $\pU$ satisfying the following:  for all $\xi\ne\eta\in\pU$, 
$$\rho_\epsilon^{o}(z, w) \asymp \mathrm{e}^{-\epsilon L}$$ where $L = d(o, [\xi,\eta])$ and the implicit constant depends only on $\delta$.    
\end{lem}  Visual metrics remain in the same Holder class under changing the parameter $\epsilon$  and the basepoint $o$. We often write $\rho_\epsilon$ if the basepoint is understood.   

A large class of Gromov hyperbolic spaces is provided  by  CAT$(-1)$-spaces. 
In the definition below, \emph{triangle} refers to an embedded 2-simplex.
\begin{defn}\label{def-cat-1} Let $X$ be a geodesic metric space.
Let $\mathbb H^2$ be the real hyperbolic plane (of constant curvature $-1$). Given a triangle $\Delta$ in $X$ with  geodesic edges, a  comparison triangle $\bar \Delta$  is a  geodesic triangle in  $\mathbb H^2$ with edges isometric to the corresponding edges  of $\Delta$. Then $X$ is a \textit{CAT$(-1)$-space} if  every geodesic triangle in $X$ is thinner than the comparison triangle,\ i.e. the edge identification map $\Delta\to  \bar \Delta$ 
sending edges isometrically to edges is  (globally) $1$-Lipschitz.    
\end{defn}

Thanks to the Alexandrov comparison theorem, any simply connected complete Riemannian manifold of sectional curvature   $\le -1$ is a CAT$(-1)$-space.

The   \emph{$r-$shadow} of $y$ seen from $x$ is given by  
$$\Pi_{x}(y, r) := \{\xi\in \pU: \exists [x,\xi]\cap B(y,r)\ne\emptyset\}$$

\begin{lem}\label{ShadowApproxBalls}\cite[Section 6]{coornaert}
	There exist constants $r, C>0$ so that the following holds.
Let $\gamma$ be a geodesic ray starting at $o$ and ending at $\xi\in \pU$.  For any $x\in \gamma$, one has
$$B_{\rho_\epsilon}(\xi, C^{-1} \cdot \mathrm{e}^{-\epsilon d(o,x)}) \subset \Pi_o(x,r) \subset  B_{\rho_\epsilon}(\xi, C \cdot \mathrm{e}^{-\epsilon d(o,x)}),$$
where $\rho_\ep$ is a visual metric with basepoint $o$.    
\end{lem}

We now study the action of  a discrete group   on Gromov boundary and introduce various classes of conical limit points,  which are the key objects studied in this paper.

Assume that $G$ acts properly on a hyperbolic space $X$. The \textit{limit set} $\pG$    consists of accumulation points of $Go$ in the Gromov boundary $\partial X$ of $X$  for some (or any) $o\in X$. Alternatively, the limit set $\pG$ is the same as the set of  accumulation points of all  orbits in $\partial X$.  We say the action is \textit{non-elementary} if $\Lambda G$ contains at least three points. 

The action of $G$ on $\partial X$  by homeomorphism is a \textit{convergence group action} in the following sense. Any infinite set of elements  $\{g_n\in G:n\ge 1\}$ has a \textit{collapsing} sequence $\{g_{n_i}\}$ with a pair of (possibly same) points $a,b\in \pG$: the sequence of maps $g_{n_i}$ converges to the constant map $\delta_a$ locally uniformly on $\pG\setminus b$. Here   $\delta_a$ sends everything to $a$. Moreover, the defining properties of $\{g_{n_i}\}$ and $a, b$ are such that $g_no\to a$ and $g_n^{-1}o\to b$ for some $o\in X$.

The limit set satisfies the following  \textit{duality condition} due to Chen-Eberlein.
\begin{lem}\label{dualitycondi}
Assume that $|\pG|\ge 2$. Then for any distinct pair $(\xi,\eta)$ in $\pG$, there exists a sequence of elements $g_n\in G$ so that $g_no\to \xi$ and $g_n^{-1}o\to \eta$ for some or any $o\in X$.     
\end{lem}\begin{proof}
If $|\pG|= 2$, then $G$ is virtually cyclic and the conclusion follows immediately in this case. Let us now assume $|\pG|> 2$ and thus $\pG$ is uncountable. In particular, $G$ has no global fixed point in $\pG$. By definition, let us take $h_no\to \xi$  and $k_no\to \eta$. Up to taking subsequence, assume $h_n^{-1}o\to a$ and $k_n^{-1}o\to b$. We may assume $a\ne b$;  otherwise if $a=b$, we find $f\in G$ so that $fa\ne a$ and then replace $k_n$ with $k_nf$: $(k_nf)^{-1}o=f^{-1}k_n^{-1}o \to fb\ne a$.   Then $g_n:=h_nk_n^{-1}$ is the desired sequence: $h_nk_n^{-1} o \to \xi$ and $k_nh_n^{-1} o \to \eta$.    
\end{proof}

\begin{defn}\label{ConicalDefn}
A limit point $\xi$ in $\pG$ is called \textit{conical} if there exists a sequence of elements $g_n\in G$  so that $g_no\to\xi$  and  $g_no$ lies within an $R$-neighborhood of a geodesic ray  $[o,\xi)$ for some number $R>0$. If, in addition, $\sup_{n\ge 1}\{d(g_no,g_{n+1}o)\}<\infty$, then $\xi$ is called \textit{uniformly conical}.    
\end{defn} 
\begin{rem}
It is useful to give an equivalent formulation of conical points using only boundary actions. Namely, $\xi\in \pG$ is conical if and only if there exist a sequence $g_n\in G$ and a pair of distinct points $a\ne b\in \pG$ so that for any $\eta\ne \xi$, we have $g_n^{-1}(\xi,\eta)\to (a,b)$. This definition works in any convergence group action.    
\end{rem} 

Except for uniformly conical points, several other classes of conical points have been studied in literature.  The following class of points was introduced by P. Myrberg \cite{Myr31} in 1931 in his approximation theorem for Fuchsian groups. The    geodesic ray ending at Myrberg point was called there ``quasi-ergodic".

\begin{defn}\label{MyrbergDefn}
A  limit point $\xi$ is  called a \textit{Myrberg point}  if  for any distinct pair $a\ne b\in \pG$, there exist a sequence of elements $g_n\in G$ so that $g_n(o,\xi)\to (a,b)$ for some (or any) $o\in X$.   
\end{defn}

\begin{rem}  
By the convergence group action, one could equally replace the basepoint $o$ with any point in $X\cup \pU\setminus \xi$.     Indeed, we could take a sequence  of points $x_n\in X \to x\ne\pU\setminus \xi$, for which the statement works, and thus conclude that $g_nx\to \zeta$ and $g_n\xi\to \eta$.
\end{rem}

\begin{lem}\label{lem:CharMyrberg}
A limit point $\xi\in \pG$ is a Myrberg point if and only if the following holds.

There exist a universal constant $r>0$ depending on hyperbolicity constant. For any  loxodromic element  $h\in G$ there is a sequence of   distinct axis $g_n\ax(h)$ with $g_n\in G$ so that for any $x\in X$  the intersection $[x,\xi]\cap N_r(g_n\ax(h))$ has diameter tending to $\infty$ as $n\to\infty$. 
\end{lem}
\begin{proof}
As two geodesic rays ending at the same point are eventually contained in the universal neighborhood of the other, we only need to very the conclusion for some $[x,\xi]$ with $x\in X$.

$\Leftarrow$: We apply the definition of Myrberg limit point to the pair of fixed points $(h^-,h^+)$ of $h$. We thus have a sequence of elements $g_n^{-1}\in G$ so that $g_n^{-1}x\to h^-$ and $g_n^{-1}\xi\to h^+$. From the visual topology, we know that $g_n[x,\xi]$ projects to the axis $\ax(h)$ as a subset with  diameter tending to $\infty$. The axis $\ax(h)$ is a $c$-quasi-geodesic for a universal constant $c$, so we obtain a constant $r$ depending on $c$ and hyperbolicity constant  that the intersection $g_n[x,\xi]\cap N_r(\ax(h))$ tends to $\infty$. This concludes the proof of this direction.

$\Rightarrow$: the above argument is reversible: $g_n^{-1}(x,\xi)$ tends the fixed points $(h^-,h^+)$ of the loxodromic element $h$. The proof is then finished by the fact that the fixed point pairs of all loxodromic elements are dense in $\pG\times \pG$.  
\end{proof}
\begin{rem}
Myrberg limit points could be defined in a much larger context with  contracting elements, in class of convergence boundary (\textsection \ref{sub-conv-bdry}) which includes visual boundary of CAT(0) spaces and Teichm\"uller spaces, and horofunction boundary of any metric space with contracting elements.  See \textsection \ref{sub-Myrberg} for the details and  a characterization of Myrberg limit points (Lemma \ref{CharMyrberg-general}) in this context.     
\end{rem}

Following \cite{AHM}, we say that a point $\xi\in \pU$ is a \textit{controlled concentration point} if it has a neighborhood $U$ so that for any neighborhood $V$ of $p$ there exists $g\in G$ so that $\xi\in gU\subset V$.  \cite[Theorem 2.3]{AHM}  characterizes the endpoint of a Poincar\'e-recurrent ray (defined below)  as a controlled concentration point. Moreover, Myrberg limit points are controlled concentration points, but the converse is not true.

These notions of limit points are closely related to the asymptotic behaviors of geodesic rays on the quotient manifold. To be concrete, we assume that $X$ is the universal covering of a complete negatively pinched Riemannian manifold $M$  and $G=\pi_1(M)$ acts by deck transformation on $X$. 

Consider the geodesic flow $\mathfrak  g^t: T^1(M)\to T^1(M)$ with $t\in \mathbb R$ on the unit tangent bundle $T^1(M)$. Fix a basepoint $p\in M$. A vector $v\in T_p^1(M)$ is called \textit{wandering} if there exists an open neighborhood $U$ of $p$ so that $\mathfrak g^t(U)\cap U= \emptyset$ for all sufficiently large   time $|t|> 0$. Otherwise, it is called \textit{non-wandering}: for any open neighborhood $U$ of $p$, there exists a sequence of times $t_n\to \infty$  so that $\mathfrak g^{t_n}(U)\cap U= \emptyset$. The non-wandering set thus forms a closed subset of $T^1(M)$. Thanks to the duality property of limit points (Lemma \ref{dualitycondi}),   the trajectory $\{\mathfrak g^t(v):t\in \mathbb R\}$ lifts  to a  bi-infinite geodesic with  endpoints in the limit set $\pG$. The non-wandering set is thus a subset of the unit tangent bundle to the quotient of the convex hull of the limit set. It corresponds to vectors $v$ for which $\{\mathfrak g^t(v):t\in \mathbb R\}$ lies in the quotient of the convex hull of the limit set. 

A vector $v\in T_p^1(M)$ is called \textit{Poincar\'e-recurrent}   if there exists a sequence of times $t_n\to \infty$ so that $\mathfrak  g^{t_n}(v)\to v$. It is called \textit{transitive} if the semi-infinite trajectory $\{\mathfrak g^t(v):t\ge 0\}$ is dense in the non-wandering set of $T^1(M)$. Equivalently, $v$ is transitive if and only if the oriented geodesic with tangent vector $v$ lifts to an oriented geodesic ending at a Myrberg point. By definition, a transitive geodesic ray is recurrent, but the converse is false: a periodic geodesic is recurrent but  of course not transitive. In general, the set of Myrberg points is disjoint from the set of uniformly conical points unless $M$ is convex-compact.

\begin{table}[!h]
        \centering
        
\begin{tabular}{|p{0.40\textwidth}|p{0.40\textwidth}|}
\hline 
 Non-wandering geodesics & Limit points \\
\hline 

Recurrent geodesics   & Conical  points \\
\hline  
\quad Bounded geodesics   & \quad Uniformly conical  points \\

\quad Transitive geodesics   & \quad Myrberg  points \\
 
\quad  Poincar\'e-recurrent geodesics & \quad Controlled concentration  points  \\
\hline 
 Non-wandering escaping geodesics & Non-conical  points \\
\hline
\end{tabular}\\
\hfill
\caption{Correspondence between geodesic rays and limit points}   
\label{tbl-geodesic-limitpoints}
        \end{table}

In the sequel, $\cG$, $\uG$, $\mG$, $\ncG$ denote respectively  the conical limit set,  the uniform conical limit set, the Myrberg limit set and  the non-conical limit set.

\section{Hausdorff dimension of ends of large trees}\label{sec-hausd}
We start by recalling the notion of Hausdorff measures in a metric space.
\begin{defn}
	Let $W$ be a subset in a metric space $(Y, d)$. Given 
	$\epsilon, s \ge 0$, define
	$$
	\mathcal H^s_{\epsilon}(W) = \inf\left\{ \sum \diam{U_i}^s: W \subset
	\bigcup_{i=1}^{\infty} U_i, U_i \subset Y, \diam{U_i} \le \epsilon \right\}.
	$$
	Define 
	the \textit{$s$-dimensional Hausdorff measure} of $W$
	to be $\mathcal H^s(W) = \lim\limits_{\epsilon \to 0} \mathcal H^s_{\epsilon}(W)$. The
	\textit{Hausdorff dimension} of $W$ is given by
	$$\HD_{d}(W) =
	\inf\{s\ge 0: \mathcal H^s(W) =0 \} = \sup\{s\ge 0: \mathcal H^s(W) = \infty \}. $$ 
\end{defn}
By
convention, set $\inf\emptyset = \sup\{s \in \mathbb R_{\ge 0}\}
= \infty$. Thus, $\HD_d{W} \in [0, \infty]$. Note that $\mathcal H^s(W)$ may
be zero for $s=\HD_d{W}$. 

For the purposes of this paper, the space $Y$ will be the Gromov boundary endowed  with visual metric of a geodesic hyperbolic metric space $(X,d)$. To give a lower bound on Hausdorff dimension,  we need the notion of a \emph{quasi-radial tree}:
\begin{defn}\label{def-qrtree}
A rooted metric tree $(\TT, v_0, d_\TT)$  is said to be  \emph{quasi-radially embedded} in a geodesic metric space $(X,d)$ via $\Phi$, if $\Phi: \TT \to X$ is injective and satisfies  the following. 
There exists $c \geq 1$ such that $\Phi|_{[v_0,v]}$ is a $c-$quasigeodesic for every vertex $v$ of $\TT$.
We refer to the image of $\Phi$ as a \emph{quasi-radial tree}.
\end{defn}

In this section, we explain a general procedure to build  large quasi-radial trees in the sense that their growth  is exponential with a large exponent. This will turn out to be intimately related to the Hausdorff dimension of their boundary.

\begin{rem}
Note that a quasi-radial tree is not necessarily  quasi-isometrically embedded globally. Only ``radial" geodesics in $\TT$ starting at the root
$v_0$ are required to be uniformly quasi-isometrically embedded.
\end{rem} 
\subsection{Construction of quasi-radial trees from group actions}\label{subsec-qtreegroup}
We start by introducing the data we need to build a quasi-radial tree. Recall that $G$ acts isometrically and properly discontinuously on a geodesic hyperbolic space $X$. Fix a basepoint $o \in X$.
\begin{defn}\label{def-annularset}
$A (L,\Delta,o) \subset G$ will denote the \emph{annular set with parameters
	$L,\Delta$} given by $$A (L,\Delta,o) := \{ g \in G: |d(o,go)-L|\le \Delta\}.$$
\end{defn}

\noindent {\bf Conditions on a sequence of annular sets:}\\
We will need a constant $R>0$, and a sequence $\{A_n\subseteq A (L_n,\Delta_n,o), \; n\ge 1\}$  of annular sets, parameters $L_n,\Delta_n$ and a sequence of non-negative real numbers 
$\omega_n$ such that
\begin{align}
\label{LargeGrowth1}\tag{L1} A_n & \subseteq A (L_n,\Delta_n,o)\\
\label{LargeGrowth2}\tag{L2}   |A_n|& \ge  \mathrm{e}^{L_n\omega_n}\\
\label{Separation}\tag{S0}  \forall a\ne a'\in A_n:\;\;& d(ao,a'o)>2\Delta_n+2R
\end{align}
In what follows, $\Delta_n, L_n$ may tend to $\infty$;  however, $L_n$ will be large relative to $\Delta_n$. The constant $R$ shall be a uniform constant furnished by Lemma~\ref{localtoglobal}, and depending  on $\tau$ introduced below. 

Condition (S0) ensures that $a\ne a'\in A_n$ are \emph{well-separated}. The letter S here connotes large separation.\\

\noindent {\bf Conditions on auxiliary elements and straightness:}\\
 Let  $\{b_n\in G:n\ge 1\}$ be a sequence of auxiliary elements. Let $B_n:=d(o,b_no)$. 
 \begin{defn}\label{def-straight}
 We say that a sequence of annular sets $A_n$ and auxiliary elements 
 $\{b_n\}$ satisfies a \textit{local $\tau$-straight} condition for some $\tau>0$, if  for each $n\ge 1$, 
\begin{align}
	\label{LocalStraight1}\tag{S1} \forall a, a'\in A_n:\;\; & d(o,[a^{-1}o, a'o])\le \tau\\
	\label{LocalStraight2}\tag{S2} \forall a\in A_n, a'\in A_{n+1}:\;\; & d(o,[a^{-1}o, b_no]), d(o,[b_n^{-1}o, a'o])\le \tau 
\end{align}
 \end{defn} 
\begin{rem}
The letter S in  conditions (S1) and (S2) connotes local straightness. They guarantee that the concatenations
 $[a^{-1}o,o]\cup[o, a'o]$ and $[a^{-1}o,o]\cup[o,b_no]$  are $(1+2\tau)$-quasi-geodesics (in the sense of Definition~\ref{def-qgeo}). 
 Equivalently, $[o,ao]\cup[ao, aa'o]$ and $[o,ao]\cup[ao,ab_no]$  are $(1+2\tau)$-quasi-geodesics. 
\end{rem}

 Let $\mathcal K=\{K_n: n\ge 1\}$
be a sequence of positive integers. We shall refer to   $\mathcal K=\{K_n: n\ge 1\}$ as a sequence of \emph{repetitions} (the reason for this terminology will become clear below). 

For a set $A \subset G$,
 $A^{K}$ will denote the set of  $K$-tuples $a=(a^{(1)},\cdots,a^{(K)})$  in $A$. Under evaluation as an element of $G$, a $K$-tuple $a=(a^{(1)},\cdots,a^{(K)})$ will be written as a product $\prod_{i=1}^{K}a^{(i)}$. \\

\noindent {\bf Admissible words and tree-representation:}\\
Let $m\ge 0$ be an integer. Then 
 $\prod_{n=1}^m A_n^{K_n} b_n$ denotes the set of  words of the  form $$W=\prod_{i=1}^{K_1}a_1^{(i)} b_1 \prod_{i=1}^{K_2}a_2^{(i)}  b_2 \cdots \prod_{i=1}^{K_m}a_m^{(i)}  b_m.$$  Words such as $W$ are referred to as \textit{admissible words}. Thus, an admissible word is
  a concatenation (with $n$ ranging from 1 to $m$) of $K_n$ elements $a_n^{(i)}$ ($1\le i\le K_n$) of $A_n$, and the letter $b_n$ in the natural order. The last letter $b_m$ could  be absent. We allow $W$ to be the empty word when $m=0$.
Let $\mathcal W$ be the set of all such admissible words, that is, $$\mathcal W=\bigcup_{m\ge 0} \left(\prod_{n=1}^m A_n^{K_n} b_n\right)$$
The \emph{length} of $W$ as above is defined to be
\begin{enumerate}
\item $\sum_{n=1}^m (K_n+1)$  when $b_m$ is non-trivial,
\item $\sum_{n=1}^m (K_n+1)-1$, otherwise.
\end{enumerate}  
Let $\mathcal W_m$  denote the set of admissible words of length  $m\ge 0$.
We can write $\mathcal W=\cup_{m\ge 0}^{\infty} \mathcal W_m$.

It will be helpful to represent $\mathcal W$ as the vertices of a rooted tree   $\mathcal T$ with the root vertex  given by the empty word denoted as $W_0$. The vertex set $\mathcal W$ is partitioned    according to \textit{generations} (length of admissible words): $$\mathcal W = \bigcup_{m=0} \mathcal W_m.$$ 
In this tree-representation,  
$\mathcal W_m$ will be referred to as the \textit{$n$-th generation}. For each vertex $W\in \mathcal T$, let
\begin{itemize}
\item $\hat W$ denote the unique parent of $W$,
\item $[\check W]$ denote  the set of children of $W$, and
\item $[\overset\leftrightarrow W]$ denote the set of siblings of $W$.
\end{itemize}   
Instead of the simplicial metric, we equip $\mathcal T$ with a different metric $d_{\mathcal T}$ as follows. Each edge $[W,W']$ is assigned  length $L_n+\Delta_n$ (resp. $B_n$) when $W'$ is obtained from $W$ by adding $a\in A_n$ (resp. $a=b_n$). For example, the vertex corresponding to the above word $W$  has distance to the root $W_0$ given by 
$$d_{\mathcal T}(W_0,W) = \sum_{n=1}^{m-1} (K_n(L_n+\Delta_n)+B_n)$$

Any admissible  word $W$ furnishes a sequence of  points in $Go$, given by the vertices of the geodesic in $\TT$ from $W_0$ to $W$. These  vertices correspond to sub-words $A_n^{K_n}b_n$, $1\le n\le m$  of the following form:
$$
\begin{aligned}
\left(\prod_{i=1}^{n-1} A_i^{K_i} b_i\right) \cdot \left(\overbrace{\underbrace{o,\;  a_n^{(1)}o, \; a_n^{(2)}o,\;  \cdots, \;  a_n^{(K_n)}o}_{K_n+1},\; \left(\prod_{i=1}^{K_n}a_n^{(i)}\right)  b_n o}^{K_n+2}\right)
\end{aligned}
$$  
The path  obtained by connecting consecutive points is said to be  \textit{labeled} by $W$ and is denoted as $p(W)$.  This defines a map as follows.
$$
\begin{aligned}
\Psi:\quad &\mathcal W \longrightarrow X\\
&W\longmapsto Wo     
\end{aligned}
$$
The image  $\Psi(\mathcal W)$  will then have the structure of a tree  induced from $\mathcal T$ and give a quasi-radial tree as in Definition \ref{def-qrtree}, provided we can prove  that $\Psi$ is injective (this is established in Lemma~\ref{QuasiRadialTree} below). 
We shall  use lowercase notation to denote  points $v, w\in \Psi(\mathcal W)$. Further, $[\check v]$ and $[\overset\leftrightarrow v]$ are 
then defined  as before.  A sequence of $v_l\in \Psi(\mathcal W)$ ($l\ge 0$) shall be refereed to as a \textit{family path} if $v_l=\hat v_{l+1}$ is the parent of $v_{l+1}$ and $v_0$ is the basepoint $o$. In this terms, $p(W)$ is exactly given by a family path by connecting consecutive points.

We now record  the main consequence of Conditions (\ref{LocalStraight1},\ref{LocalStraight2}) in Definition \ref{def-straight}.

\begin{lem}\label{lem:localqginhyp}
For any $\tau>0$, there exist $c, L, R_0>0$ with the following property. If $L_n\ge L$,  every path $p(W)$ labeled by $W\in\mathcal W$ is a $c$-quasi-geodesic and $[o,Wo]$ passes through the $R_0$-neighborhood of each $W_n o$, where    $W_n$ is the prefix  of $W$ of length $n$.    
\end{lem}
\begin{proof}
Every path labeled by $W\in\mathcal W$ is an $L$-local  $(1+2\tau)$-quasi-geodesic $p(W)$. By Lemma \ref{localtoglobal}, there exists $L\gg 0$ and $c \geq 1$ so that whenever $L_n\ge L$ for all $n\ge 1$,  $p(W)$ is  a $c$-quasi-geodesic (in the sense of Definition~\ref{def-qgeo}). By the Morse Lemma, there exists $R_0=R_0(c)$ so that the corresponding geodesic $[o,Wo]$ passes through the $R_0$-neighborhood of each $W_n o$, where    $W_n$ is the prefix  of $W$ of length $n$.      
\end{proof}

Let $T(\mathcal W)$ denote the graph obtained as the union of all paths $p(W)$ labeled by words $W$ in $\mathcal W$. 
Let $\partial X$ denote the Gromov boundary of $X$.
Let  $\epsilon, C>0$ be given by Lemma \ref{VisualMetric} and we endow $\partial X$ with the visual metric $\rho_\epsilon$.
With the above notation and setup in place we can now begin to establish a number of properties. 
\begin{lem}\label{QuasiRadialTree}
For any $\tau>0$, there exist $c, L, R>0$ with the following property. If  $L_n\ge L$ for all $n\ge 1$, then  the    map $\Psi$ is injective and  {each $p(W)$ labels a $c$-quasi-geodesic for $W\in\mathcal W$}. Further, 
\begin{enumerate}
    \item 
    the shadows $\Pi_{v_0}(v,R)$ with $v\in T(\mathcal W)$ are either disjoint or nested; the latter happens exactly when one is a descendant of the other.
    \item 
    If $w\ne  w'$ are children of $v$ associated with the set $A_m$ in $T(\mathcal W)$, $\Pi_{v_0}(w,R)$ and $\Pi_{v_0}(w',R)$ are at $\rho_\epsilon$-distance greater than $C\mathrm{e}^{-\epsilon \widetilde L}$ where $\widetilde L=d(v_0,v)+L_m+\Delta_m$.
\end{enumerate} 
In particular, the image $T(\mathcal W)$ is a quasi-radial tree in Definition \ref{def-qrtree}.
\end{lem}
\begin{rem}[on further generalizations in \textsection \ref{sec-contracting}]
The injectivity of $\Phi$ uses only Lemma \ref{lem:localqginhyp} which follows from Gromov's hyperbolicity. The same property holds for  admissible paths (Definition \ref{AdmDef}) in general metric space with strongly contracting elements (\textsection \ref{sec-contracting}).  In (2), the visual metric separation  between shadows $\Pi_{v_0}(w,R)$'s  uses the estimates  in Lemma \ref{ShadowApproxBalls}, which hold for Floyd metrics along certain $w$ as stated in Lemma \ref{lemma3.16PY}. This lemma shall be used in the proof of Theorem \ref{thm-qrtree-Myrberg-general} and \ref{MyrbergHdimFloyd}.
\end{rem}
\begin{proof}
Let $c, L, R_0>0$ be given by Lemma \ref{lem:localqginhyp}. Then $p(W)$ is  a $c$-quasi-geodesic and $d(W_no,[o,WO])\le R_0$, where    $W_n$ is the prefix  of $W$ of length $n$.  

{We first prove that $\Psi$ is injective. Indeed, if not, assume that $Wo=W'o$ but $W\ne W'\in\mathcal W$. As $b_n$ is uniquely chosen, the first  place where $W, W'$ differ are in $A_n$ for some $n$. Assume therefore that    $a_n\ne a_n'\in A_n$ occurring in $W,W'$ are different. By the above discussion, if $\gamma$ denotes the geodesic between $o$ and $Wo=W'o$, we have $d(a_no, \gamma)\le R_0$ and $d(a_n'o, \gamma)\le R_0$. Choose $x, y\in \gamma$ so that $d(a_no,o)=d(x,o)$ and $d(a_n'o,o)=d(y,o)$. Then $d(a_no,x)\le 2R_0$ and $d(a_n'o,y)\le 2R_0$. As $|d(o,a_no)-d(o,a_n'o)|\le 2\Delta_n$, we see that $d(x,y)\le 2\Delta_n$ and $d(a_no,a_n'o)\le 4R_0+2\Delta_n$. Setting $R>2R_0$, this contradicts (\ref{Separation}) for $a_n\ne a_n'\in A_n$, completing the proof for  the injectivity of $\Psi$. }

Next, we prove that if $W\ne W'$ have the same parent $V\in\mathcal W$,  the shadows $\Pi_{v_0}(w,R_0)$ and $\Pi_{v_0}(w',R_0)$ at $w:=Wo$ and $w':=W'o$ are disjoint. 
Indeed, if not, let us choose $\xi\in \Pi_{v_0}(w,R_0) \cap \Pi_{v_0}(w',R_0)$, so that we have $d(w,[v_0,\xi])\le R_0$ and $d(w',[v_0,\xi])\le R_0$. At $v:=Vo$,   the two uniform quasi-geodesics from $v_0$ to $w$ and $w'$ branch off from each other. Hence, $v$ lies in the $R_0$-neighborhood of the two geodesics starting at $v$ and ending at $w$ and $w'$. Up to increasing $R_0$ by a uniformly bounded amount, since $d(w,[v_0,\xi])\le R_0$, we have that $[v_0,w]$ is contained in the $R_0$-neighborhood of $[v_0,\xi])$. Thus $d(v,[v_0,\xi])\le R_0$. A similar reasoning as in the second paragraph of this proof proves $d(w,w')\le 6R_0+2\Delta_n$. This  contradicts (\ref{Separation}) again  when $R>3R_0$. The statement (1) thus follows.

We now prove statement (2). See Figure~\ref{fig:quasitree}.  Assume that $w\ne w'$ are children of $v\in T(\mathcal W)$ and are associated with  elements in $A_{n}$ for some $n$. Then, by the triangle inequality, $\widetilde L_n:=d(v_0,v)+L_{n}+\Delta_{n}$ gives an upper bound on $d(v_0, w)$ for any child $w\in [\check v]$. So the $\rho_\epsilon$-diameter of  $\Pi_{v_0}(w,R)$ is at most $C\mathrm{e}^{-\epsilon\widetilde L_n}$ for some universal $C$ as per Lemma \ref{VisualMetric}. 

On the contrary, if  statement (2) fails,  let us choose $\xi\in \Pi_{v_0}(w,R_0), \xi'\in \Pi_{v_0}(w',R_0)$ so that $\rho_\epsilon(\xi,\xi')\le C\mathrm{e}^{-\epsilon\widetilde L_n}$. 
\begin{figure}
    \centering

\tikzset{every picture/.style={line width=0.75pt}} 

\begin{tikzpicture}[x=0.75pt,y=0.75pt,yscale=-1,xscale=1]

\draw    (75,141) -- (214,140) ;
\draw    (214,140) -- (276.5,124.5) ;
\draw [shift={(276.5,124.5)}, rotate = 346.07] [color={rgb, 255:red, 0; green, 0; blue, 0 }  ][fill={rgb, 255:red, 0; green, 0; blue, 0 }  ][line width=0.75]      (0, 0) circle [x radius= 3.35, y radius= 3.35]   ;
\draw    (214,140) -- (278.5,160.5) ;
\draw [shift={(278.5,160.5)}, rotate = 17.63] [color={rgb, 255:red, 0; green, 0; blue, 0 }  ][fill={rgb, 255:red, 0; green, 0; blue, 0 }  ][line width=0.75]      (0, 0) circle [x radius= 3.35, y radius= 3.35]   ;
\draw    (75,141) .. controls (206,132.5) and (314,108) .. (357,84) ;
\draw    (75,141) .. controls (142,143.5) and (317,167) .. (355.5,194) ;
\draw    (355.5,194) .. controls (320,157.5) and (276,145.5) .. (357,84) ;
\draw    (311.5,140.5) -- (300.63,107.9) ;
\draw [shift={(300,106)}, rotate = 71.57] [color={rgb, 255:red, 0; green, 0; blue, 0 }  ][line width=0.75]    (10.93,-3.29) .. controls (6.95,-1.4) and (3.31,-0.3) .. (0,0) .. controls (3.31,0.3) and (6.95,1.4) .. (10.93,3.29)   ;
\draw [shift={(311.5,140.5)}, rotate = 251.57] [color={rgb, 255:red, 0; green, 0; blue, 0 }  ][fill={rgb, 255:red, 0; green, 0; blue, 0 }  ][line width=0.75]      (0, 0) circle [x radius= 3.35, y radius= 3.35]   ;
\draw    (311.5,140.5) -- (303.03,171.57) ;
\draw [shift={(302.5,173.5)}, rotate = 285.26] [color={rgb, 255:red, 0; green, 0; blue, 0 }  ][line width=0.75]    (10.93,-3.29) .. controls (6.95,-1.4) and (3.31,-0.3) .. (0,0) .. controls (3.31,0.3) and (6.95,1.4) .. (10.93,3.29)   ;
\draw   (262.26,121.95) .. controls (263.14,113.98) and (270.31,108.23) .. (278.29,109.11) .. controls (286.26,109.98) and (292.01,117.16) .. (291.13,125.13) .. controls (290.25,133.11) and (283.08,138.86) .. (275.1,137.98) .. controls (267.13,137.1) and (261.38,129.92) .. (262.26,121.95) -- cycle ;
\draw   (265.51,161.47) .. controls (265.4,153.86) and (271.48,147.6) .. (279.09,147.48) .. controls (286.7,147.37) and (292.96,153.45) .. (293.07,161.06) .. controls (293.18,168.68) and (287.1,174.94) .. (279.49,175.05) .. controls (271.88,175.16) and (265.62,169.08) .. (265.51,161.47) -- cycle ;
\draw  [line width=0.75]  (270.33,89.67) .. controls (269.55,85.07) and (266.86,83.16) .. (262.26,83.94) -- (176.52,98.57) .. controls (169.95,99.69) and (166.27,97.95) .. (165.49,93.35) .. controls (166.27,97.95) and (163.37,100.81) .. (156.8,101.93)(159.76,101.43) -- (80.72,114.91) .. controls (76.12,115.7) and (74.21,118.39) .. (75,122.99) ;

\draw (69.5,141.9) node [anchor=north west][inner sep=0.75pt]    {$v_{0}$};
\draw (208.5,140.4) node [anchor=north west][inner sep=0.75pt]    {$v$};
\draw (242,92.9) node [anchor=north west][inner sep=0.75pt]    {$B( w,R)$};
\draw (241,176.9) node [anchor=north west][inner sep=0.75pt]    {$B( w',R)$};
\draw (319.5,129.9) node [anchor=north west][inner sep=0.75pt]    {$z$};
\draw (345.5,70.9) node [anchor=north west][inner sep=0.75pt]    {$\xi \in \Pi _{v_{0}}( w,R)$};
\draw (346.5,171.9) node [anchor=north west][inner sep=0.75pt]    {$\xi '\in \Pi _{v_{0}}( w',R)$};
\draw (111.33,66.07) node [anchor=north west][inner sep=0.75pt]    {$d( v_{0} ,w) \leq \ \widetilde L_{n}$};

\end{tikzpicture}
    \caption{Lemma \ref{QuasiRadialTree}}
    \label{fig:quasitree}
\end{figure}
Let $z\in [\xi,\xi']$ be a nearest point projection point of $v_0$ to $[\xi,\xi']$.
On account of the inequality $C^{-1}\mathrm{e}^{-\epsilon d(v_0,[\xi,\xi'])}\le \rho_\epsilon(\xi,\xi')$ (Lemma \ref{VisualMetric}),   we have $d(v_0,z)\ge \widetilde L_n-\log (C^2/\epsilon)$. Moreover, by  the thin-triangle property for the triangle with vertices $(v_0,\xi,\xi')$, the point  $z$ lies within distance $C$ of the  two sides $[v_0,\xi]$ and $[v_0,\xi']$ (up to increasing $C$ by a constant depending only on $\delta$). See Fig. \ref{fig:quasitree}. Recall that $v$ is  within distance $R_0$ of $[v_0,w]$ and $w$ is within distance $R_0$ of $[v_0,\xi]$. Up to increasing $C$ again depending also on $R_0$, and noting that $$d(v_0, w),d(v_0,w')\le \widetilde L_n\le d(v_0,z)+\log (C^2/\epsilon)$$ we have $d(v,[v_0,z])\le C$. The thin-triangle property again shows that $ d(w,[v_0,z]),d(w',[v_0,z])\le C$. That is, $v, w,w'$ are contained in a $C$-neighborhood of the same geodesic $[v_0,z]$. Since $|d(v,w)-d(v,w')|\le 2\Delta_n$,   a similar argument as  in the proof of injectivity of $\Psi$ yields  $d(w,w')\le 2\Delta_n+4C$. This contradicts (\ref{Separation}) for $R\gg 4C$.  The proof of (2) is complete.
\end{proof}



We fix the local straight constant $\tau>0$ in Definition \ref{def-straight}. Let  $L, R>0$ be given by Lemma \ref{QuasiRadialTree} for this $\tau$.  We write $T=T(\mathcal W)$ in the sequel. Denote $\omega:=\liminf_{n\to \infty}\omega_n$.

\begin{lem}\label{LargeTreeGrpVersion}
If $K_n$ is chosen so that the parameters $(L_n,\Delta_n,B_n)$  satisfy
\begin{align}\label{ChoiceKnEq}
\frac{\Delta_n}{L_n}+\frac{B_n}{K_nL_n}\to 0    
\end{align} then the growth rate $\e {\mathcal W}$ of  $\Psi(\mathcal W)$ is greater than or equal $\omega$.  
\end{lem}
\begin{proof}
We may assume $\omega_n\to \omega$ in what follows.
 For each $W=a_1  b_1\cdots a_m  b_m\in\mathcal W$ with $K_n$-tuples $a_n=(a_n^{(i)})\in A_n^{K_n}$, we have 
$$
d(o, Wo) \le \sum_{n=1}^m \left(\left(\sum_{i=1}^{K_n} d(o,a_n^{(i)}o) \right)+K_n\Delta_n+B_n\right)
$$ by the triangle inequality.
We estimate the Poincar\'e series associated to $\mathcal W$. Note first that 
$$
\begin{aligned}
\sum_{W\in\mathcal W} \mathrm{e}^{-s d(o,Wo)} 
\ge \sum_{m\ge 1}^\infty \left(\prod_{n=1}^m    \mathrm{e}^{-s B_n} \left(\sum_{a\in A_n}\mathrm{e}^{-s (d(o,ao)+\Delta_n)}\right)^{K_n}\right)
\end{aligned}
$$
where the lower bound follows by injectivity of $\Phi$.
Fix any $0<s<\omega$.  We claim that there exists $n_0>1$    so that for $n> n_0$,
\begin{align}\label{LowerBDonAi}
\left(\sum_{a\in A_n} \mathrm{e}^{-s  (d(o,ao)+\Delta_n)}\right)^{K_n} >   \mathrm{e}^{s B_n}      
\end{align}

We conclude the proof assuming ~(\ref{LowerBDonAi}). Set $q=\prod_{n=1}^{n_0} (\sum_{a\in A_n} \mathrm{e}^{-s (d(o,ao)+\Delta_n)})^{K_n}\cdot  \mathrm{e}^{-s B_n}>1$.  Then $$\sum_{W\in\mathcal W} \mathrm{e}^{-s d(o,Wo)}\ge \sum_{m\ge n_0+1} q=\infty.$$  It follows that $\e {\mathcal W}\ge s$. As this holds for any $s<\omega$, we have $\e {\mathcal W}\ge \omega$.

We now establish (\ref{LowerBDonAi}).  The conditions (\ref{LargeGrowth1}) and (\ref{LargeGrowth2}) give us the following:
$$
\sum_{a\in A_n} \mathrm{e}^{-s d(o,ao)-s \Delta_n} \ge   \mathrm{e}^{\omega_n L_n } \mathrm{e}^{-s (L_n+2\Delta_n)}.
$$ 
In order to prove (\ref{LowerBDonAi}), it suffices  to show that
$$\begin{aligned}
\mathrm{e}^{\omega_n K_nL_n } \mathrm{e}^{-s K_n (L_n+2\Delta_n)} \ge   \mathrm{e}^{s B_n}.   
\end{aligned}$$
Equivalently, it suffices  to show that $\omega_n K_nL_n \ge s(K_nL_n+2K_n\Delta_n+B_n)$. that is, 
\begin{align}\label{LowerBDonwi}
\omega_n \ge s  \left(1+\frac{2\Delta_n}{L_n}+\frac{B_n}{K_nL_n}\right)
\end{align}
By assumption, $\omega_n\to \omega$ and $\frac{\Delta_n}{L_n}+\frac{B_n}{K_nL_n}\to 0$. As   $s$ is a fixed number less than $\omega$, there exists $n_0$ so that the inequality (\ref{LowerBDonwi}) is satisfied  for all large $n\ge n_0$. Hence, (\ref{LowerBDonAi}) holds as desired, and the proof of the lemma is complete. 
\end{proof}

\begin{lem}\label{HDLargeTree} Let $\epsilon$ be the parameter for the visual metric in Lemma \ref{VisualMetric}. 
We continue with Condition~(\ref{ChoiceKnEq}) of Lemma \ref{LargeTreeGrpVersion}, and assume further that $K_n$ is chosen  to  satisfy the condition
\begin{align}\label{ChoiceKnEq2}
\frac{L_{m+1}+\Delta_{m+1}}{\sum_{n=1}^m K_n (L_n+\Delta_n)+B_n} \to 0, \text{ as } m\to \infty.
\end{align} Then the Hausdorff dimension of the boundary $\partial T(\mathcal W)$ of the quasi-radial tree $T(\mathcal W)$ is greater than or equal to $\frac{\omega}{ \epsilon}$.
\end{lem}
\begin{proof}
Write $T=\cup  V_n$ for $n\ge 0$ with $V_n:=\Psi(\mathcal W_n)$ and $E_n=\cup_{v\in  V_n} \Pi_{v_0}(v,R)$. Then $\partial T=\cap_{n\ge 0} E_n$.   

Fix $0<s<\omega/\epsilon$.  We shall define a probability measure $\nu$ (depending on $s$) on $E_0=\Pi_{v_0}(v_0,R)$ that is supported on $\partial T$. 

Set $\nu(E_0)=\nu(\Pi_{v_0}(v_0,R))=1$. For $v\in T$, define  
\begin{equation}\label{nuShadowEq}
\nu(\Pi_{v_0}(v,R))=\frac{\mathrm{e}^{-s\epsilon d(v_0,v)}}{\sum_{w\in [\overset\leftrightarrow v]}\mathrm{e}^{-s\epsilon d(v_0,w)}}\nu(\Pi_{v_0}(\hat v,R))    
\end{equation}
Recall that $\rho_\epsilon$ denotes the visual metric.
Let $B_{\rho_\epsilon}(\xi,t) \subset \partial T$ denote the  
$\rho_\epsilon-$ball  centered at $\xi$ of radius $t$. We define $$\nu(B_{\rho_\epsilon}(\xi,t)) =\inf_{\mathcal U} \sum_{U\in\mathcal U} \nu(U)$$ where the infimum is taken over  covers  $\mathcal U$ of  $B_{\rho_\epsilon}(\xi,t)$ by a collection of shadows $\Pi_{v_0}(v,R)$ at $v\in T$.

\textbf{Step 1}. We first prove that  $\nu(\Pi_{v_0}(v,R))\le \mathrm{e}^{-s \epsilon d(v_0,v)}$ for any $v\in T$. 
A path $v_0, v_1, \cdots, v_l:=v$ in $ T$ for some $l\ge 1$, with $\hat v_i= v_{i-1}$ the parent of $v_i$ for $1\le i <l$ will be referred to as a
 \textit{family path}. Consider such a family path $v_0, v_1, \cdots, v_l:=v$.  Then
$$
\nu(\Pi_{v_0}(v_l,R)) = \frac{\nu(\Pi_{v_0}(v_l,R))}{\nu(\Pi_{v_{0}}(v_{l-1},R))}\frac{\nu(\Pi_{v_0}(v_{l-1},R))}{\nu(\Pi_{v_0}(v_{l-2},R))} \cdots \frac{\nu(\Pi_{v_0}(v_1,R))}{\nu(\Pi_{v_{0}}(v_0,R))} \nu(\Pi_{v_{0}}(v_0,R)).
$$
Unraveling the definition in  (\ref{nuShadowEq}), it thus suffices to  prove the following: 
\begin{align}\label{blancedproduct}
\prod_{i=0}^l\frac{\mathrm{e}^{-s\epsilon d(v_0,v_{l-i})}}{\sum_{w\in [\overset\longleftrightarrow {v_{l-i}}]} \mathrm{e}^{-s\epsilon d(v_0,w)}} \le \mathrm{e}^{-s\epsilon d(v_0,v_{l})}    
\end{align}
Condition~(\ref{blancedproduct}) is in turn equivalent to the following 
condition by canceling $ \mathrm{e}^{-s\epsilon d(v_0,v_{l})}$ from   the two sides:
\begin{align}\label{blancedproduct2}
\prod_{i=0}^l{\mathrm{e}^{-s\epsilon d(v_0,v_{l-i-1})}} \le \prod_{i=0}^l{\sum_{w\in [\overset\longleftrightarrow {v_{l-i}}]} \mathrm{e}^{-s\epsilon d(v_0,w)}}    
\end{align}

By triangle inequality, $d(v_0,w)\le d(v_0, v_{l-i-1})+d(v_{l-i-1},w)$ for any sibling $w$ of ${v_{l-i}}$.  Let  $A$ denote the set of children of $v_{l-i-1}$, i.e. $A$ is the
 set of siblings   $[\overset\longleftrightarrow {v_{l-i}}]$ of $v_{l-i}$. Then  
\begin{align}\label{ASeriesEq}
\mathrm{e}^{-s\epsilon d(v_0, v_{l-i-1})} \sum_{a\in A} \mathrm{e}^{-s\epsilon d(o, ao)} \le \sum_{w\in [\overset\longleftrightarrow {v_{l-i}}]} \mathrm{e}^{-s \epsilon d(v_0,w)}\end{align}
By the nature of the construction, $A$ is either the set $A_{n}$  or $\{b_n\}$ for some $n$. 

By the choice of $K_n$ in  (\ref{LowerBDonAi}), we have $$\left(\sum_{a\in A_n} \mathrm{e}^{-s\epsilon  (d(o,ao))}\right)^{K_n} >   \mathrm{e}^{s\epsilon B_n}$$
with the constraint $s<\omega$  replaced with $s\epsilon<\omega$ in the RHS of (\ref{LowerBDonwi}). Consequently, for any $m\ge 1$, \begin{align}\label{blancedproduct3}
1\le \prod_{n=1}^m \left(   \sum_{a\in A_n} \mathrm{e}^{-s\epsilon d(o, ao)}\right)^{K_n} \left( \mathrm{e}^{-s\epsilon B_n}\right)   
\end{align}
For concreteness, assume that $\sum_{n=1}^m (K_n+1)\le  l<\sum_{n=1}^{m+1} (K_n+1)$. We deal with the case $l=\sum_{n=1}^m (K_n+1)$; the other case  follows from this.  Now, if we take the product  of the two sides of (\ref{ASeriesEq}) over $0\le i\le l$:
$$
\prod_{i=0}^l{\mathrm{e}^{-s\epsilon d(v_0,v_{l-i-1})}} \prod_{n=1}^m  \left(  \sum_{a\in A_n} \mathrm{e}^{-s\epsilon d(o, ao)}\right)^{K_n} \left(\mathrm{e}^{-s\epsilon B_n}\right) \le \prod_{i=0}^l{\sum_{w\in [\overset\longleftrightarrow {v_{l-i}}]} \mathrm{e}^{-s\epsilon d(v_0,w)}} 
$$
then the condition (\ref{blancedproduct}) follows from  (\ref{blancedproduct3}).
Thus, $\nu(\Pi_{v_0}(v_n,R))\le \mathrm{e}^{-s \epsilon d(v_0,v_n)}$ is proved. 

\textbf{Step 2}. Fix any $0<s_0<s$.  Let $C$ be given by Lemma \ref{ShadowApproxBalls}.
We are   going to prove that $\nu(B_{\rho_\epsilon}(\xi,t)) \le (2t/C)^{s_0}$ for all $\xi\in \partial T$ and for all small $t>0$.    

Let $\Pi_{v_0}(v_l,R)$ be the shadow of a \textit{lowest generation} $v_l\in \Phi(\mathcal W_l)$ containing $B_{\rho_\epsilon}(\xi,t)$ for some $l\ge 1$ (i.e. $l$ is minimal). For definiteness, assume that the children of $v_l$ are given by the set $A_{m}$ for some $m$. Denoting $v_0=\Phi(W_0)$ and $v_l=\Phi(W_l)$ for words $W_0,W_l\in \mathcal W$, we have $$d_{\mathcal T}(W_0,W_l) = (L_m+\Delta_m)(l-\sum_{n=1}^{m-1}(K_n+1)) +\sum_{n=1}^{m-1} (K_n(L_n+\Delta_n)+B_n).$$  Then, $\widetilde L_m:=d(v_0,v_l)+L_{m}+\Delta_{m}$ gives an upper bound on $d(v_0, v)$ for any child $v\in [\check v_l]$. If $v\ne v'$ are siblings we have $\Pi_{v_0}(v,R)$ are at  $\rho_\epsilon$-distance at least $C\mathrm{e}^{-\epsilon\widetilde L_m}$ by Lemma \ref{QuasiRadialTree}.  Since $B_{\rho_\epsilon}(\xi,t)$ is not contained in the shadow $\Pi_{v_0}(v,R)$ of any descendant $v$ of $v_n$, $B_{\rho_\epsilon}(\xi,t)$ intersects at least two $\Pi_{v_0}(v,R)$ with $v\in [\check v_l]$. Hence $C\mathrm{e}^{-\epsilon\widetilde L_m} < 2t$.

Note that the map $\Phi$ sends each geodesic ray in $\mathcal T$ issuing at the root to a  $c$-quasi-geodesic at $v_0$ (Lemma \ref{QuasiRadialTree}). Thus, we have $$cd(v_0,v_l)+c\ge d_{\mathcal T}(W_0,W_l)\ge \sum_{n=1}^{m-1} (K_n(L_n+\Delta_n)+B_n).$$ Hence the assumption on $K_n$ in (\ref{ChoiceKnEq2}) implies  that    
$({L_{m}+\Delta_{m}})/{\widetilde L_m}\to 0$ and   
$$
\frac{d(v_0,v_l)}{d(v_0,v_l)+L_{m}+\Delta_{m}}\ge \frac{s_0}{s}
$$
for $d(v_0,v_l)\gg 0$.
Thus, by Step (1), for all $\xi\in \partial T$ and for all small $t>0$, $$\nu(B_{\rho_\epsilon}(\xi,t)) \le \nu(\Pi_{v_0}(v_l,r))\le \mathrm{e}^{-s\epsilon d(v_0,v_l)}  \le \left(\frac{2}{C}\right)^{s_0}t^{s_0}.$$
This proves $\HD(\partial T)\ge s_0$. As $s_0<s<\omega/\epsilon$ is arbitrary, it follows that $\HD(\partial T)\ge \omega/\epsilon$. 
\end{proof}
\subsection{Construction of quasi-radial trees from a  pattern}\label{subsec-qtreepattern}
In this subsection, we recast, in a form that will be relevant to us,  some of the material in \cite{FM01, MRT19} in terms of Poincar\'e series. This could be thought of as a purely geometric (not group theoretic) version of the previous section. This formulation shall be used  to estimate Hausdorff dimension of boundaries of trees. 

The following definition is an analog to the set of conditions (\ref{LargeGrowth1},\ref{LargeGrowth2},\ref{Separation},\ref{LocalStraight1}).  
\begin{defn}\label{defn-patteronset}
Let $v,\hat v\in X$ be a pair of points.   We say that  a finite set $A$ of points $x\in  X$ for  $(\hat v, v)$ has \textit{pattern with parameters} $(L, \Delta,\omega,R,\tau)$ if   the following conditions hold  
\begin{align}
\label{LargeGrowth1'}\tag{L1'} & |d(v,x)-L|\le \Delta\\
\label{LargeGrowth2'}\tag{L2'} & |A| \ge \mathrm{e}^{\omega L}\\
\label{Separation'}\tag{S0'} & d(a,a')>2\Delta+2R,\;\forall a\ne a'\in A\\
\label{LocalStraight1'}\tag{S1'} & d(v,[\hat v,x])\le \tau
\end{align}
Note that if $v=\hat v$, the last condition $d(v,[\hat v,x])=0$ is vacuous.
    
\end{defn}

\subsubsection*{\textbf{Quasi-radial tree from a pattern}} Fix a sequence of parameters $(L_n, \Delta_n, \omega_n,R,\tau)$, a sequence of repetitions $K_n$, and a sequence of bridge lengths $B_n$. 
We shall build a quasi-radial tree $T$    in $X$ by choosing  a sequence of subsets $A_n$ repeated  $K_n$-times followed with a bridge $b_n$ with length $B_n$ to the next $A_{n+1}$. This is similar to the construction of admissible words $\mathcal W$. However, since there are no group actions, we      inductively       build the quasi-radial tree by appropriately choosing points in $X$. We now explain the construction  subject to these parameters in the following way. 

We construct inductively a sequence of   finite subsets $V_l$ in $X$  for $l\ge 0$. Set $V_0=\{o\}$. Given $V_l$, we construct $V_{l+1}$.
\begin{enumerate}
    \item 
    Let $n\ge 1$ be the minimal integer with $$l\le \sum_{m=1}^{n} (K_m+1)-1$$ 
    
    For each element $v$ in $V_l$, we construct a finite set of children for the pair $(\hat v,v)$, denoted by $[\check v]$,  that has a pattern with parameters $(L_n, \Delta_n,\omega_n,R,\tau)$. Here we set $v=\hat v$ if $l=0$.
    
    Inductively, set $V_{l+1}=\cup_{v\in V_l} [\check v]$  at most $K_n$ times until $l+1=\sum_{m=0}^{n} (K_m+1)$.   
    \item 
    For each point $w\in [\check v]$, pick a point $\check w\in X$ satisfying  
    \begin{align}
    \nonumber &d(w,\check w)=B_n\\
    \label{LocalStraight2'}\tag{S2'} &d(w, [v,\check w]) \le \tau
    \end{align} 
    The resulting set of points denoted by  $V_{l+1}$  has the same cardinality as $V_l$ by construction.
    \item 
    We repeat the above steps (1) and (2). 
    
\end{enumerate} 
A sequence of points $v_l\in V_l$ ($l\ge 0$) with $v_{l}=\hat v_{l+1}$  the parent of $v_{l+1}$ is referred to as a \textit{family path}. Let $\mathcal T$ denote the  underlying tree structure of the sets $V_l$ ($l\ge 0$) induced by the parent-child relation. 
The resulting set $T=\cup_{l\ge 0} V_l$  will be a quasi-radial tree, once we establish that the    map 
$$
\begin{aligned}
\Psi:\;& \mathcal T \longrightarrow  X\\
&v\longmapsto v     
\end{aligned}
$$
is injective.

The same argument as in the proofs of Lemma \ref{QuasiRadialTree}, Lemma \ref{LargeTreeGrpVersion}, Lemma \ref{HDLargeTree}  proves the following. Set $\omega:=\liminf_{n\to \infty}\omega_n$.  
\begin{lem}\label{LargeTreePtsVersion}
For any $\tau$, there exist $L,R>0$ with the following property.  Let $\Psi$ be the map  constructed as above with parameters $(L_n,\Delta_n,\omega_n,K_n,B_n)$. If $L_n\ge L$ the map $\Psi$ is injective and every family path obtained by joining consecutive vertices by geodesic segments  is  a $c$-quasi-geodesic  in $X$. In particular, $T$ is a quasi-radial tree in the sense of Definition \ref{def-qrtree}.

Moreover, if $K_n>0$ is a sequence satisfying  conditions (\ref{ChoiceKnEq}) and (\ref{ChoiceKnEq2}), then the growth rate of $ T$ equals $\omega$ and the Hausdorff dimension of $\partial T$ is $\frac{\omega}{\epsilon}$.  
\end{lem}


\section{Counting geodesic arcs between two closed geodesics}\label{sec-arcs}

The goal of  this section is to present counting results about  shortest arcs between two closed geodesics in Riemannian manifolds and graphs. This follows from a more general result on the counting of double cosets in groups with contracting elements in \cite{HYZ}. Here we present the argument for the case in hyperbolic spaces, as it is relatively short and also facilitates the construction of an appropriate quasi-radial tree.

\subsection{Setup}
We recast the  setup in terms of group actions.
Assume that $G\act X$ is a proper isometric action on a proper hyperbolic geodesic  space. Let $\alpha,\beta$ be two quasi-geodesics in $X$. Let $H$ and $K$ be the stabilizers of $\alpha$ and $\beta$ respectively. Assume that $H$ and $K$ preserves $\alpha$ and $\beta$ co-compactly. 
In our applications, $\alpha,\beta$ are preserved by two loxodromic elements $h,k$ respectively  and $H=E(h), K=E(k)$ are the associated maximal elementary subgroups.

Denote the set of  $G$--translates of $\alpha,\beta$ as follows
$$
[\alpha]=\{g\alpha: g\in G\},\quad [\beta]=\{g\beta: g\in G\}.
$$
Thus $[\alpha]$ and $[\beta]$ could equivalently be thought of as the images of  two geodesics corresponding to  $\alpha$ and $\beta$ on the quotient space $X/G$.  

Since $G$ acts on $[\alpha]\times [\beta]$ by the diagonal action, the quotient denoted by $\Ar([\alpha],[\beta])=([\alpha]\times [\beta])/G$   records   the set of shortest arcs from $[\alpha]$ to $[\beta]$. To be precise, the elements in $\Ar([\alpha],[\beta])$ are of form $G(\alpha, g\beta)$ for $g\in G$.
These are  $G$-translates of the pair  $(\alpha, g\beta)$.  


Let $\Dc(H,K):=\{HgK: g\in G\}$ be the collection of double cosets.  Then we have the following one-to-one correspondence: 
$$\begin{aligned}
\Phi: \Ar([\alpha],[\beta])/G&\longrightarrow \Dc(H,K) \\
G(\alpha, g\beta)& \longmapsto HgK
\end{aligned}$$  

Let $\Ar(n,[\alpha],[\beta])$ be the set of $G(\alpha, g\beta)$'s satisfying $d(\alpha,\beta)\le n$. 
We have $\Ar([\alpha],[\beta])=\cup_{n>0} \Ar(n,[\alpha],[\beta])$.
Similarly, let $\Ar(n,\Delta, [\alpha],[\beta]) =\{G(\alpha, g\beta): |d(\alpha, g\beta)-n|\le \Delta\}$.
Simplifying notation, we shall write $\Ar(n,\Delta)=\Ar(n,\Delta, [\alpha],[\beta])$, when $[\alpha],[\beta]$ are understood from the context.


The quotient $\Ar(n,[\alpha],[\beta])/G$ is  the set of shortest arcs between $[\alpha]$ and $[\beta]$ so that $d([\alpha],[\beta])\le n$. The above correspondence allows us to estimate $\Ar(n,\Delta)$ from the cardinality of  $$\Dc(o, n,\Delta)=\{AgB: |d(Ao,gBo)-n|\le \Delta\}.$$

\begin{lem}\label{arc-db-cosets}
For any point $o\in X$, there exist constants $\Delta_0$ and  $c_1, c_2$ depending on $o$ so that for any large $\Delta>\Delta_0$ 
$$\begin{aligned}
 c_1|\Dc(o,n,\Delta-\Delta_0)|  \le |\Ar(n,\Delta)| \le c_2|\Dc(o, n, \Delta+\Delta_0)| 
\end{aligned}$$    
\end{lem}
\begin{proof}
As $H$ acts co-compactly on $\alpha$ and $Ho$, there exists a constant $R$ depending on $o$ so that $\alpha$ and $Ho$ are contained in an $R$-neighborhood of each other.  By the same reason, $\beta $ and $ Ko$ are contained in an $R$-neighborhood of each other.    
Let $c$ be a shortest arc from $\alpha$ to $\beta$
and $\ell(c)$ denote its length. As $Ho$ and $Ko$ are quasi-convex subsets,  there exists a constant $\Delta_0$ so that  $|d(Ho,Ko)-\ell(c)|\le \Delta_0$.  
\end{proof}

\subsection{Constructing shortest arcs}
Recall the annular set
$$A(n,\Delta,o)=\{go\in Go: |d(o,go)-n|\le \Delta\}$$

This section is devoted to the proof of the following Theorem. \begin{thm}\label{Thm:Growth-DoubleCosets}
Given $\Delta>0$, there exist $\Delta'=\Delta'(\Delta, o)$ and $c=c(\Delta, o)$ so that for any $n\gg 0$,
$$|\Ar( n, \Delta')|\ge c\cdot |A(o,n, \Delta)|.$$
\end{thm}


Let us note the following elementary fact.
\begin{lem}\label{BddProjLem}
Let  $F$ be a finite set of pairwise independent loxodromic elements in $G$. There exists some $\tau$ depending on $F$ with the following property:
$$\forall g\in G,\; \forall f_1\ne f_2\in F:\;\; \min\{\mathrm{diam}(\pi_{\mathrm{Ax}(f_1)}([o,go])),\mathrm{diam}(\pi_{\mathrm{Ax}(f_2)}([o,go]))\}\le \tau 
$$     
\end{lem}

As $G$ contains infinitely many independent loxodromic elements, we may choose a set   $F$ of three loxodromic elements so that the union $\{h,k\}\cup F$ are pairwise independent. That is, the axes of any distinct pair of elements in the set
$\{h,k\}\cup F$ have $\tau$-bounded projections for some $\tau>0$:
\begin{align}
\label{BddProjEQ1}
\forall f_1\ne f_2\in \{h,k\}\cup F:\quad \mathrm{diam}(\pi_{\mathrm{Ax}(f_1)}(\mathrm{Ax}(f_2)))\le \tau.
\end{align} 
This further implies the following (up to increasing $\tau$ if necessary):
\begin{align}
\label{BddProjEQ2} \forall g\in G,\; \forall f_1\ne f_2\in F:\quad & \min\{\mathrm{diam}(\pi_{\mathrm{Ax}(f_1)}([o,go])),\mathrm{diam}(\pi_{\mathrm{Ax}(f_2)}([o,go]))\}\le \tau.
\end{align} 
In particular, $d(o,ao)\le \tau$ for any $a\in (H\cup K) \cap E(f)$ with $f\in F$. 

\begin{lem}\label{DCExtension} Let $F$ be as above.
There exist $n_0$ and $c,\Delta_0>1$ with the following property.
For any $g\in G$ with $d(o,go)>n_0$ and $a\in H, b\in K$, we have $f_1,f_2\in F^{n_0}$ so that the word $(a,f_1,g,f_2,b)$ labels a $c$-quasi-geodesic. Moreover, $|d(Ho,f_1 g f_2 Ko)-d(o,go)|\le \Delta_0$.
\end{lem}
\begin{proof}
By the thin triangle property, we note that if two $c_0$-quasi-geodesics $\alpha, \beta$ with $\alpha_+=\beta_-$ have $\tau$-bounded projection, then $\alpha\beta$ is a $\tau'$-quasi-geodesic, for some $\tau'$ depending on $\tau,c_0$ and the hyperbolicity constant. Let $L, c\ge 1$ be  as in Lemma \ref{localtoglobal} so that an $L$-local $\tau'$-quasi-geodesic is a $c$-quasi-geodesic.  

By (\ref{BddProjEQ1}),    $[o,ao]a[o, f^no]$ is a $\tau'$-quasi-geodesic  with any $n\in \mathbb Z$ and $f\in F$.  Choose $n_0>L$ large enough so that $F^{n_0}$ consists of elements with length greater than $L$. For any $g\in G$ with $d(o,go)>n_0$, we may apply  (\ref{BddProjEQ2}) twice to choose $f_1, f_2\in F$ so that $\mathrm{diam}(\pi_{\mathrm{Ax}(f_1)}([o,go])),\mathrm{diam}(\pi_{\mathrm{Ax}(f_2)}([o,go]))\le \tau$. Then $[o,f_1o]f_1[o,go]$ and $[o,go]g[o,f_2o]$ are $\tau'$-quasi-geodesics. This implies that connecting   consecutive points in the sequence  $$(o, af_1o, af_1go, af_1gf_2o,af_1gf_2bo)$$ by geodesic segments one obtains an $L$-local $\tau'$-quasi-geodesic, and hence a $c$-quasi-geodesic.  This path is labeled by $(a,f_1,g,f_2,b)$. 

We now prove the ``moreover" statement. First, $d(Ho,f_1 g f_2 Ko)\le d(o,go)+2D$ where $D:=\max\{d(o,fo): f\in F\}$. For the other direction,  let $\alpha$ be a shortest arc  from $Ho$ to $f_1 g f_2 Ko$. We may assume that $\alpha$ starts at some point $a^{-1}o$ with $a\in H$ and ends at $f_1gf_2bo$ with $b\in K$.  As above, consider the $c$-quasi-geodesic $\gamma$ labeled by $(a,f_1,g,f_2,b)$, which has the same endpoints as $\alpha$.
By  the Morse Lemma,  $\alpha$ lies in the $R_0$-neighborhood of $\gamma$ for some $R_0$ depending on $c$. As $a^{-1}o\in Ho, bo\in Ko$, we obtain  $d(o,ao), d(o,bo)\le 2R_0$: indeed, if $d(o,a^{-1}o)>2R_0$, the fact $d(o,x)\le R_0$ for $x\in \alpha$ implies   $d(x,a^{-1}o)>R_0$,  contradicting that $\alpha$ is shortest arc. This implies $\ell(\alpha)\ge d(o,go)-2R_0-d(o,f_1o)-d(o,f_2o)\ge  d(o,go)-2R_0-2D$. Setting $\Delta_0=2R_0+2D$ completes the proof.   
\end{proof}
We now define a map as follows:
$$\begin{aligned}
\Phi:  A(o,n,\Delta) &\longrightarrow  \Dc(n,\Delta+\Delta_0)\\
g&\longmapsto  Hf_1gf_2K
\end{aligned}$$
where  the elements   $g\in G$, $(f_1,f_2)$ are chosen as per Lemma \ref{DCExtension}.

\begin{lem}\label{lem-finite-one}
There exists an integer $M$ such that for all $n$,  the above  map $\Phi$  is at most $M$-to-one. 
\end{lem}
\begin{proof}
Assume that $ g_1\ne g_2$ with $ Hf_1gf_2K = Hf_1'g_2f_2'K$ for two   pairs $(f_1,f_2)\ne (f_1',f_2')$ in $F$. Write explicitly, for some $a,a'\in H$ and $b,b'\in K$,
$$a f_1g_1f_2b =a'f_1'g_2f_2'b'$$  
First of all, we must have $a\ne a'$ or $b\ne b'$. Otherwise, if $a=a'$ and $b=b'$ then $f_1g_1f_2=f_1'g_2f_2'$. As $g_1\ne g_2$ we have $f_1\ne f_1'$ or $f_2\ne f_2'$. Assume $f_1\ne f_1'$ for concreteness. By the choice of $F$, $f_1$ and $f_2$ are independent. Thus, the word $(f_2^{-1},g_1^{-1},f_1^{-1},f_1',g_2,f_2')$ labels a $c$-quasi-geodesic with the same endpoints (i.e. a loop at $o$). The length is at most $c$, but this contradicts the choice of $f$ satisfying $d(o,f_1o)>c$. 

Now, let us assume  that  $a\ne a'$ (the argument is symmetric for $b\ne b'$).  Then, either   $f_1\ne f_1'$ or $f_1=f_1'=:f$ with $a'^{-1}a\notin E(f)$.
In both cases   the  word $$\left(b^{-1},f_2^{-1},g_1^{-1},f_1^{-1},a^{-1}a', f_1',g_2',f_2',b'\right)$$ labels   a $c$-quasi-geodesic, which is a loop at the basepoint $o$. This gives a contradiction as above.  Hence,   $f_1=f_1'$ and   $a'^{-1}a\in E(f_1)$ and similarly, $f_2=f_2'$ and  $b'^{-1}b\in E(f_2)$. By (\ref{BddProjEQ1}), there are at most $N$   choices of $a'^{-1}a$ and $b'^{-1}b$ with $N$ depending on $\tau$.    Once $g_1$ is chosen, $g_2$ is determined   up to $N^2\times |F|^2$-possibilities, so the map is at most $(9N^2)$-to-one. Setting $M=9N^2$ we are done.
\end{proof}

Theorem \ref{Thm:Growth-DoubleCosets} now follows from Lemma \ref{lem-finite-one} and Lemma \ref{arc-db-cosets}. $\hfill \Box$\\  

\subsection{Applications}
We end this section with an application to counting shortest arcs in Riemannian manifolds and graphs.

Let $M$ be a negatively curved Riemannian manifold. Let $[\alpha]$ and $[\beta]$ be two closed geodesics on $M$. 
We consider an arc $\sigma$ whose end-points are in 
$[\alpha]$ and $[\beta]$. Next, consider the \emph{constrained homotopy class} of $\sigma$ where the endpoints  are allowed to move in 
$[\alpha]$ and $[\beta]$.
Each such constrained homotopy class  contains a unique shortest representative, which we shall refer to as a \textit{shortest arc}. We denote by $\Ar([\alpha],[\beta])$ the set of all shortest arcs between $[\alpha]$ and $[\beta]$.  

\begin{lem}\label{ShortestArcsonMfd}
Let $M$ be a complete Riemannian manifold with pinched negative  curvature. Let $\e G$ be the critical exponent for the action of $G:=\pi_1(M)$
on $ \widetilde M$. Let $\gamma$ be a closed geodesic on $M$.  Then there exist  $c, \Delta>0$ depending on $\gamma$ so that the following holds. Let  $\mathrm{Arc}(\gamma,t,\Delta)$ denote the
collection of shortest arcs
 from $\gamma$ to $\gamma$  with length in  $[t-\Delta,t+\Delta]$. Then for any $\epsilon>0$, and for any $t\gg 0$:
$$
|\mathrm{Arc}(\gamma,t,\Delta)|\ge c\mathrm{e}^{(\e G-\epsilon) t}.
$$    
\end{lem}
\begin{proof}
Fix a lift $\tilde \gamma$ of $\gamma $ in $\widetilde M$ and denote by $H$ the stabilizer of $\tilde \gamma$ in $G$. Since $\widetilde M$ is a CAT(-1) space, $G$ has no nontrivial torsion and $H$ is an infinite cyclic group. We choose a basepoint $o$ on $\tilde \gamma$. Then there exists $R$ (depending on $o$) so that $Ho$ and $\tilde \gamma$ have Hausdorff distance at most  $R$. According to the discussion at the beginning of this section, a shortest arc $\alpha$ from $\gamma$ to itself lifts to a shortest arc $\tilde \alpha$ between $\tilde \gamma$ and $a\tilde \gamma$  for some $a\in G$. Further, the assignment $\alpha\mapsto HaH$ is bijective. It follows that $|d(Ho, aHo)-\ell(\tilde \alpha)|\le 2R$. By definition of critical exponent   $\e G$, for any $\epsilon>0$ we have $|A(t,\Delta,o)|\ge \mathrm{e}^{(\e G-\epsilon)t}$ holds for all sufficiently large $t\gg 0$. Thus the conclusion follows by Theorem \ref{Thm:Growth-DoubleCosets}.
\end{proof}

The following corollary for graphs will be useful in Theorem \ref{thm-amenable-graph}. In this setting, an immersed  (i.e. non back-tracking) path in a graph $\Gamma$ lifts to a geodesic in its universal cover $\widetilde \Gamma$. Conversely, any geodesic in $\widetilde \Gamma$ projects to an immersed path in $\Gamma$. A \textit{shortest arc} between two immersed loops $\alpha,\beta$ will refer to an immersed path $\gamma$ intersecting  $\alpha,\beta$ only at the endpoints. This terminology is justified by the fact that lifts of  $\gamma$ are shortest arcs between lifts of $\alpha,\beta$.  
\begin{lem}\label{ShortestArcsonGraphs}
Let $\Gamma$ be an infinite regular graph with degree $d$. Let $\e G$ be the critical exponent for the action of $G:=\pi_1(\Gamma)$
on $ \widetilde \Gamma$. Let $\gamma$ be  an immersed loop in $\Gamma$.  Then there exist  $c, \Delta>0$ depending on $\gamma$ so that the following holds. Let  $\mathrm{Arc}(\gamma,t,\Delta)$ denote the
collection of shortest arcs
 from $\gamma$ to $\gamma$  with length in  $[t-\Delta,t+\Delta]$. Then for any $\epsilon>0$, and for any $t\gg 0$:
$$
|\mathrm{Arc}(\gamma,t,\Delta)|\ge c\mathrm{e}^{(\e G-\epsilon) t}.
$$    
\end{lem} 

\section{Hausdorff dimension of non-conical points: graphs and surfaces}\label{sec-nonconical}

In this section, we describe two  constructions of  escaping geodesic  rays: one geometric for negatively curved manifolds, the other group theoretic for group actions on Gromov hyperbolic spaces. The resulting geodesic rays end at non-conical points. Subsequently, these are implemented in graphs and hyperbolic surfaces, leading to  proofs of Theorem \ref{thm-amenable-graph} and Theorem \ref{thm-cheeger-surface}.

\subsection{Escaping geodesics in negatively curved manifolds}
\label{sec-escape-mfld}

Let $M=X/G$ be a Riemannian manifold with  pinched  negative curvature. Let $\gamma_n$ $(n\ge 1)$ be a  sequence of closed geodesics  on $M$ that is \emph{escaping}, i.e.\ the sequence exits every compact set. If $M$ is geometrically infinite, such a sequence $\gamma_n$ must exist. In fact,
$M$ is  geometrically infinite if and only if there exists  an escaping  sequence of closed geodesics  by \cite{bonahon-bouts} and \cite[Theorem 1.5]{KL19}. 

We fix,   for each $n\ge 1$,   a shortest arc  $\mathfrak b_n$ from $\gamma_n$ to $\gamma_{n+1}$. We call such arcs
 \textit{bridges}.  Since $\gamma_n$'s are escaping, the sequence  $\mathfrak b_n$ is also escaping. Let $\widetilde L_n=d_M(o,\gamma_n)$. Then $\widetilde L_n$ tends to $\infty$
as $n\to \infty$.

Let $\Delta_n=\ell(\gamma_n)$ be the length of $\gamma_n$ and let $B_n=\ell(\mathfrak b_n)$ be the length of $\mathfrak b_n$. It is useful to keep in mind the  following special case of Lemma \ref{localtoglobal} in the current setup.

\begin{lem}\label{CAT(-1)Case}
Let $X$ be a hyperbolic space. Then there exist $c,L>0$ with the following property. Let $\gamma=\alpha_1\alpha_2\cdots \alpha_n$ be a piecewise  geodesic path so that $\alpha_n$ is a shortest arc between  $\alpha_{n-1}$ and  $\alpha_{n+1}$. If the length of each $\alpha_i$ is greater than $L$ then $\gamma$ is a $c$-quasi-geodesic.
\end{lem}

Let $\tau$ be a constant so that the intersection point of two orthogonal geodesics $\alpha, \beta$ is $\tau$-close to the corresponding geodesic between $\alpha_-$ and $\beta_+$. Let $R, L$ be given by Lemma \ref{QuasiRadialTree} for this $\tau$. Assume  $L$ also satisfies Lemma \ref{CAT(-1)Case}.

\subsubsection{Construction}\label{sub-constr-mfld}
Let $A_n$ be a set of \textit{oriented} shortest arcs  from $\gamma_n$ to itself with length in $[L_n-\Delta_n,L_n+\Delta_n]$, where  $\Delta_n$ depends on $\gamma_n$ by Lemma \ref{ShortestArcsonMfd}. For any $\omega_n<\e G$,  and we may take very large $L_n>L$ so that $|A_n|\ge \mathrm{e}^{\omega_n L_n}$ and $\Delta_n/L_n\to 0$. 

We place the basepoint $o$ at the starting point of $\mathfrak b_1$ on $\gamma_1$.  By increasing $\Delta_n$
if necessary, we may assume that  $ 99\Delta_n>\max\{L,R\}$.  

We  choose a definite proportion, say $0<\theta<1$, of  $A_n$ (still denoted by $A_n$ for simplicity) so  that $A_n$ is \textit{well separated}: given a lift $\tilde \gamma_n$ of $\gamma_n$, any  two distinct arcs in $A_n$ when lifted to have starting points on $\tilde \gamma_n$ have  terminal points at least $2(\Delta_n+R)$-separated. The value of $\theta$  depends on $\Delta_n$ (and $R$), but in order to  keep $|A_n|\ge \mathrm{e}^{\omega_n L_n}$, we take   even larger $L_n$. Compare with the condition (\ref{Separation}).\\

\textbf{Sliding the endpoints.} We move the starting and terminal points  of each $\alpha\in A_n$ along $\gamma_n$ so that  the resulting arc 
denoted by $\tilde \alpha$ satisfies the following.
\begin{itemize}
\item It starts and ends at the starting point of $\mathfrak b_{n}$ on $\gamma_n$, and
\item wraps 
about    $\gamma_n$  $100$ times (respecting the given orientation).
\end{itemize} 
Thus $\tilde \alpha$ is a loop and has   length  between $L_n+198\Delta_n$ and $L_n+202\Delta_n$.  Any lift of $\tilde \alpha$ in $X$ is   a concatenation of  three geodesic segments:
\begin{itemize}
\item two of these are contained in two distinct lifts  of $\gamma_n$, and
\item the lift of $\alpha$ is  the  shortest arc between them and has 
length lying in the interval $[99\Delta_n,101\Delta_n]$.
\end{itemize}  
Thus,  any lift of $\tilde \alpha$  is a $c$-quasi-geodesic in $X$ by Lemma \ref{CAT(-1)Case}. We refer to the above operation that converts 
$\alpha\in A_n$ to $\tilde \alpha$ as 
\emph{sliding  endpoints.}\\

\textbf{Looping  many times.} Denote by $\widetilde A_n$ the  set of  oriented  loops obtained from the arcs in $A_n$ by sliding their endpoints.   We now pick up an arbitrary  (not necessarily distinct)
collection of $K_n$  loops $(\tilde\alpha^{(1)},\cdots,\tilde\alpha^{(K_n)})$ from $\widetilde A_n$. Recall that they all have the same endpoints.  Concatenating  them in order while respecting their orientation gives a piecewise geodesic path  $\alpha_n$. (Since orientations have been chosen consistently, there is no  cancellation even when  consecutive pieces $\tilde\alpha^{(i)}, \tilde\alpha^{(i+1)}$ coincide). Thus, the pieces of $\alpha_n$  satisfy the hypothesis of Lemma \ref{CAT(-1)Case}: note that these pieces are  arcs that are either lifts of $\gamma_n$ or of $\alpha^{(i)}$. Hence,  any lift of $\alpha_n$  is a $c$-quasi-geodesic in $X$. \\

\textbf{Escaping to infinity.} The looping construction above guarantees that any $\alpha_n$ constructed as above begins and ends  
at the starting point of $\mathfrak b_{n}$. 
We next go through the bridge $\mathfrak b_n$ to the next $\gamma_{n+1}$.  Note that the bridge  $\mathfrak b_{n}$ may end at a point  of $\gamma_{n+1}$ that is different from the starting point of $\mathfrak b_{n+1}$. However, the distance between the end-point of $\mathfrak b_n$ and the starting point of $\mathfrak b_{n+1}$
is at most $\Delta_{n+1}=\ell(\gamma_{n+1})$.
We move the endpoint of $\mathfrak b_n$ to the starting point of  $\mathfrak b_{n+1}$ by sliding it along a distance of at most $\Delta_{n+1}$ on $\gamma_{n+1}$.  We retain the same notation for the modified $\mathfrak b_n$. Again,   the concatenation   $\alpha_{n}\cdot  \mathfrak b_n$  lifts to a $c$-quasi-geodesic in $X$ by Lemma \ref{CAT(-1)Case}. \\

To summarize, we perform the following operation for each $n$: 
\begin{enumerate}
\item loop around $K_n$ shortest arcs in $A_n$ union  $\gamma_{n}$,
\item go through the bridge $\mathfrak b_n$,  and
\item loop around $K_{n+1}$ arcs  in $A_{n+1}$   union   $\gamma_{n+1}$.
\end{enumerate}  
 The resulting piecewise geodesic paths lift  to  a family  $T$  of $c$-quasi-geodesic rays in $\widetilde M$. By construction,  $T$ has a natural structure of a rooted tree. By Lemma \ref{QuasiRadialTree}, $T$ is a quasi-radial tree with  pattern parametrized by $(L_n,\Delta_n,\omega_n,K_n, B_n)$
 (see Definition~\ref{defn-patteronset}).  Let $K_n$ be given by Lemma \ref{LargeTreePtsVersion}.

Recall that $\widetilde L_n=d_M(o,\gamma_n)$. We shall say that a
semi-infinite path
$\sigma:[0, \infty)\to M$ is \emph{escaping} if, for every compact subset
$K$ of $M$, $\sigma^{-1}(K)$ is compact.

\begin{lem}\label{EscapingRayCriterion}
If $L_n\ge L$ and $\widetilde L_n-L_n-\Delta_n\to \infty$, then the concatenation $\cup_{n\ge 1}    \alpha_n \mathfrak b_{n}$ is an escaping path in $M$.
Further, any  lift of the concatenation $\cup_{n\ge 1}    \alpha_n \mathfrak b_{n}$ is a quasi-geodesic ray ending at a non-conical limit point. 
\end{lem}
\begin{proof}
By construction, $\sigma=\cup_{n\ge 1} \alpha_n\mathfrak b_{n}$ contains an  escaping sequence  $\{\mathfrak b_n\}$ with $\widetilde L_n=d_M(o,\gamma_n)\to\infty$. Also, the length of the backtracking path due to $\alpha_n$ is at most  $L_n+\Delta_n$. By assumption, $\widetilde L_n-L_n-\Delta_n\to \infty$.
This implies that $\sigma$ is an escaping ray in $M$, i.e.\ it leaves every compact subset. 

By Lemma~\ref{CAT(-1)Case}, any lift $\tilde \sigma$  of $\sigma$ to $\tilde M$ is a $c$-quasi-geodesic ray.
By the Morse Lemma, $\tilde \sigma$ lies within a finite $R_0$-neighborhood of a geodesic ray $\tilde \gamma$, where $R_0=R_0(c)$
depends only on $c$. Thus, the projection $\gamma$ of the geodesic ray $\tilde \gamma$ to $M$  stays within the $R_0$-neighborhood of $\cup_{n\ge 1}   \alpha_n \mathfrak b_{n}$
as a parametrized path. Hence  $\gamma$  escapes every compact subset as well. 

By construction, $\tilde \sigma$ traces in turn an escaping sequence of lifts of $\gamma_n$ with length about $100\Delta_n$. Note that the endpoints $\gamma_n^\pm$ of lifts of $\gamma_n$ are fixed points of loxodromic isometries.
Further, $\gamma_n^\pm$  converge to the endpoint of $\tilde \gamma$. Hence, the endpoint of $\tilde \gamma$ is a non-conical limit point.       
\end{proof}

We summarize the above discussion as follows. Recall that $\widetilde L_n=d_M(o,\gamma_n)$.
\begin{prop}\label{CriterionHDNonConical}
Let $\gamma_n$ be an escaping sequence of closed geodesics  of length $\Delta_n$ on $M$. Let $A_n$ be a set  of shortest   arcs from $\gamma_n$ to itself with length in the interval 
$[L_n- \Delta_n, L_n + \Delta_n]$.  Assume that the cardinality 
$|A_n|$ satisfies $|A_n|\ge \mathrm{e}^{\omega_n L_n}$. Further, assume that $\Delta_n/L_n\to 0$ as $n\to\infty$.
Set $\omega=\liminf_{n\ge 1} \omega_n$. 

If $\widetilde L_n-L_n-\Delta_n\to \infty$, then  $\HD(\ncG)\ge \omega/\epsilon$, where $\epsilon$ is the parameter for the visual metric in Lemma \ref{VisualMetric}.     
\end{prop}
\begin{proof}
Let $R>R(c)$ be given as in the proof of Lemma \ref{QuasiRadialTree}. We may assume further that any two distinct arcs $\alpha,\alpha'\in A_n$ are  $R$-separated, i.e.\ their lifts starting at a common  point have endpoints at least $R$-apart. This only affects the cardinality of $|A_n|$ by a fixed fraction, depending only on $R$. For simplicity, we still assume $|A_n|\ge \mathrm{e}^{\omega_n L_n}$ up to modifying $L_n$ by a fixed amount.  

Choose a sequence of integers $K_n>0$ so that (\ref{ChoiceKnEq}) and (\ref{ChoiceKnEq2}) hold for the  parameters $(L_n,\Delta_n,B_n)$.

As mentioned above, we choose the basepoint $o$ to be the starting point of $b_1$ on $\gamma_1$. Let $\mathbb P: \til M \to M$ denote the covering projection. Let $\tilde  o$ be a point with $\mathbb P (\tilde o)=o$. We now lift each $\cup_{n\ge 1} \alpha_n \mathfrak b_{n}$ to $\til M$ to get a quasi-geodesic ray $\gamma$ starting at $\tilde o \in X$.  
The union $T$ of all such lifted quasi-geodesic rays $\gamma$ forms a quasi-radial tree, by Lemma \ref{QuasiRadialTree}.

By Lemma \ref{EscapingRayCriterion}, if $\widetilde L_n-L_n-\Delta_n\to \infty$, $\gamma$   ends at a non-conical limit point. The proof is then completed by Lemma \ref{HDLargeTree}.
\end{proof}
\begin{rem}\label{rem-Ln-Delta_n}
If $\Delta_n$ is uniformly bounded over $n$ (i.e. does not depend on $\gamma_n$), any divergent sequence of $L_n$ suffices to have $\Delta_n/L_n\to 0$. In general,  $\Delta_n$ may depend on $\gamma_n$ by Theorem \ref{ShortestArcsonMfd} (when $\gamma_n$ escapes to infinity). We have to take $L_n$ very large, but this will make the condition $\widetilde L_n-L_n-\Delta_n$ hard to be fulfilled. We are able to resolve this in surfaces (Theorem \ref{thm-cheeger-surface}) and graphs (Theorem \ref{thm-amenable-graph}).      
\end{rem}

As mentioned before, the existence of escaping sequence  $\gamma_n$ on $M$ is  very general  by Kapovich-Liu's result \cite{KL19}. However, the condition $\widetilde L_n - L_n - \Delta_n \to \infty$ presents a key challenge. Below, we give two approaches using geometric limits and amenability to satisfy this condition. 

\begin{defn}\label{def-geolt}
A sequence of manifolds with basepoints $\{(M_i, x_i)\}$ \textit{converges geometrically} to a manifold with basepoint $(N,x_\infty)$ if for any $R>0$, there exists $i_0$, and  compact submanifolds $C_i \subseteq N_i$  $C \subseteq N$ such that the following hold:
\begin{enumerate}
\item $C_i, C$ contain the $R$-balls centered at $x_i$ and $x_\infty$ respectively,
\item there exists a $K_i$-bi-Lipschitz map $h_i: C_i \to C$ for any $i \ge i_0(R)$,
\item  $K_i \to 1$ as $i \to \infty$. 
\end{enumerate}

 A sequence of Kleinian groups $(G_n)$ \textit{converges geometrically} to $\Gamma$ if and only if for a fixed basepoint $x\in \mathbb H^3$ and its projections $x_n\in \mathbb H^3/G_n$ and $x_\infty\in\mathbb H^3/\Gamma$, the sequence $\{(\mathbb H^3/G_n, x_n)\}$ converges geometrically to $(\mathbb H^3/\Gamma, x_\infty)$.
\end{defn}

We are now ready to prove the following.
\begin{thm}\label{NonConicalFromGeometricLimit}
Let $M$ be a Riemannian manifold with pinched negative   curvature. Let $G=\pi_1(M)$.  Let $x_n$ be an unbounded sequence of points on $M$. Assume that the sequence of pointed manifolds $(M,x_n)$ converges geometrically to a pointed Riemannian manifold $(N,x_\infty)$. Assume that $\pi_1(N)$ is non-elementary. Then $\HD(\ncG)\ge \omega_N/\epsilon$, where $\omega_N$ is the critical exponent of $\pi_1(N)$, and $\epsilon$ is the parameter for the visual metric in Lemma \ref{VisualMetric}. 
\end{thm}
\begin{proof}
To apply Proposition \ref{CriterionHDNonConical}, we need to specify the data $(A_n,L_n,\Delta_n,B_n)$ occurring in the hypotheses and explain how the  assumptions could be realized.

It is given that the geometric limit manifold $N$ is non-elementary. So $N$ contains infinitely many distinct closed geodesics.  Let us fix such a closed geodesic $\gamma$ and $\omega<\omega_N$.  Fix a sequence  $L_n\to\infty$.
By  Lemma \ref{ShortestArcsonMfd}, there exists for each $n\ge 1$ a set $\widetilde A_n$ of shortest arcs with length in $[L_n- \Delta, L_n+ \Delta]$ such that $|\tilde  A_n|\ge \mathrm{e}^{\omega L_n}$. The constant $ \Delta$ may depend on $\gamma$, but not on $L_n$. 

Next, $(M,x_n)$   converges geometrically to $N$, with
$\{x_n\}$ unbounded.
Geometric convergence (Definition~\ref{def-geolt}) implies the existence of an escaping sequence of closed geodesics $\gamma_n$ in $M$ such that
\begin{itemize}
\item each $\gamma_{n}$ is contained in a fixed $D-$neighborhood of $x_n$ for all $n\ge 1$, and 
\item  $\ell(\gamma_n) \leq 2\ell(\gamma)$.
\end{itemize}  

Moreover, we can choose a set $A_n$ of shortest arcs such that they 
\begin{itemize}
	\item are shortest arcs from $\gamma_n$ to itself,
\item have length   in $[L_n- \Delta, L_n+ \Delta]$,
\item have cardinality $|A_n|\ge \mathrm{e}^{\omega L_n}$.
\end{itemize}
Further, $L_n/ \Delta_n\to 0$. 
Indeed this is possible as $A_n$'s maybe chosen as pre-images under $(1+\ep)-$bi-Lipschitz maps sets of the family $\widetilde A_n$ of shortest arcs from $\gamma$ to itself in $N$
(see Definition~\ref{def-geolt}). Let $B_n=d_M(\gamma_n,\gamma_{n+1})$. Note that $B_n$  depends on $\gamma_n$.

As $x_n$ is unbounded, we see that $\widetilde L_n=d_M(o,\gamma_n)$ tends to infinity.  Since $L_n$ is fixed independent of $x_n$, we may extract a subsequence of $x_n$ and of $\gamma_n$ so that $\widetilde L_n-L_n-\Delta_n\to \infty$. Note that this may change the length $B_n$ of the bridge $\mathfrak b_n$ from $\gamma_n$ to $\gamma_{n+1}$ to larger values after passing to a subsequence.  We may then  choose a sufficiently large number $K_n$ of repetitions  of looping arcs in $A_n$ so that (\ref{ChoiceKnEq2}) is satisfied. This compensates for the effect of larger $B_n$.    Therefore,  $\HD(\ncG)\ge \omega/\epsilon$  by Proposition  \ref{CriterionHDNonConical}. As $\omega<\omega_N$ is arbitrary, the proof is complete.       
\end{proof}

\subsection{Escaping geodesics from group actions}\label{sec-escape-group}

 Assume that $G$ acts properly on a Gromov hyperbolic space $X$. 

\begin{defn}\label{def-qa}
Let $g\in G$ be a loxodromic element. 
We  define the \emph{quasi-axis}  $\ax(g)$ to be   the convex hull of the two fixed points of $g$ in the Gromov boundary of $X$.
Thus, $\ax(g)$ is the union of all bi-infinite geodesics between $g^-$ and $g^+$.
\end{defn}

Let $E(g)<G$ denote the maximal elementary subgroup containing $\langle g\rangle$.
Let  $H<G$ be a subgroup. We denote $$\Ax_H(g)=H\cdot \ax(g)=\cup_{h\in H} h\ax(g).$$  If $H=G$, we write $\Ax(g)=\Ax_G(g)$ for simplicity.

\begin{defn}
Let $\{g_n\in G: n\in \mathbb N\}$ be a sequence of elements.   We say that $\{g_n\}$ \textit{escapes to infinity} if $d(o, \Ax(g_n))\to\infty$ as $n\to\infty$.     
\end{defn}
This is equivalent to saying that the sequence  $\{\Ax(g_n)\}$ 
regarded as essential loops on the quotient space $X/G$ 
is escaping (that is, the sequence  $\{\Ax(g_n)\}$ leaves every bounded subset). 


\subsubsection{Construction}\label{sub-constr-actions}
Assume that $G$ contains
an  infinite sequence of  escaping loxodromic elements $g_n$. Then $d(o,\ax(g_n))\ge d(o,\Ax(g_n))\to\infty$.
For each $n$, we fix a shortest arc $\mathfrak b_n$ from $\ax(g_n)$ to $\ax(g_{n+1})$. We may assume that $d(o,G\mathfrak b_n)\to\infty$ up to taking a subsequence of $\Ax(g_n)$. That is, the projection of $\mathfrak b_n$ to $X/G$ is escaping. To be in line with the  construction on manifolds,  denote $$\Delta_n=\diam{\ax(g_n)/E(g_n)}\text{ and } B_n=\ell(\mathfrak b_n)$$   

\noindent \textbf{Convention.}
Since each $\ax(g_n)$ is quasi-isometric to a real line, we could fix an orientation on $\ax(g_n)$ so that we can talk about a coarse left-right order. That is, for any point $x$ in $\ax(g_n)$, we can specify  a point $y\in \ax(g_n)$ with $d(x,y)>10\Delta_n$ to  the \textit{left} or \textit{right} of $x$.\\

\textbf{Sliding the endpoints.} 
Let $\alpha$ be any shortest arc  between $\ax(g_n)$ and $a \ax(g_n)$ for some $a\in G$. On $\ax(g_n)$, we may choose  some $h\in E(g_n)$  so that the starting point $h\alpha_-$ of $h\alpha$ is to the right of $(\mathfrak b_n)_-\in \ax(g_n)$ and  $99\Delta_n\le d(h\alpha_-,(\mathfrak b_n)_-)\le 100\Delta_n$. Now, $h\alpha$ is a shortest arc between   $\ax(g_n)$ and $ha \ax(g_n)$.

On $ha\ax(g_n)$, we choose  some $h'\in E(g_n)$  so that the terminal point $h\alpha_+$ of $h\alpha$ is to the left of $h'a (\mathfrak b_n)_-\in ha'\ax(g_n)$ and  $99\Delta_n\le d(h\alpha_+,h'a(\mathfrak b_n)_-)\le 100\Delta_n$.

In the end, the resulting new path, still denoted by $\alpha$,  is composed of two segments with the original $\alpha$ in between. By Lemma \ref{CAT(-1)Case}, $\alpha$ is a $c$-quasi-geodesic, with the   length  $\ell(\alpha)$ in $[L_n+198\Delta_n,L_n+202\Delta_n]$. 

By Theorem \ref{Thm:Growth-DoubleCosets} and by sliding the endpoints for shortest arcs on a given $\ax(g_n)$, we produce   a set $A_n$ of such $c$-quasi-geodesics $\alpha$ with the following properties:
\begin{itemize}
\item $A_n$ has cardinality at least $\mathrm{e}^{\omega_n L_n}$,
\item for each $\alpha$, there exists some $a\in G$ so that the  path $\alpha$ has initial point  $(\mathfrak b_n)_-$ and terminal point at $a(\mathfrak b_n)_-$ and $|d((\mathfrak b_n)_-,a(\mathfrak b_n)_-)-L_n|\le 100\Delta_n$.  
\end{itemize}
The translate $gA_n$ for $g\in G$ will be referred to as the set of shortest arcs from $\Ax(g_n)$ to itself  \textit{lifted} at $g(\mathfrak b_n)_-$.

Fix a sequence of repetitions $K_n$.
We now give the formal construction of sets $V_n$ with pattern $(L_n,\Delta_n,\omega_n)$.

Let the root $V_0=\{(\mathfrak b_1)_-\}$ be  the starting point of $\mathfrak b_1$ on $\ax(g_1)$. Assume that the set $V_l$ with $l\ge 0$ is constructed.  We inductively construct sets $V_{l+1}$ as follows.  Let $n\ge 1$ be the minimal integer with \begin{equation}\label{generationEq}
l\le \sum_{m=1}^{n} (K_m+1) -1.   
\end{equation}

\textbf{Looping  many times.} Each point $v\in V_l$  on $a\ax(g_n)$  for some $a\in G$ is the starting point of $a\mathfrak b_n$. We consider the  set of  shortest arcs $aA_n$ on $\Ax(g_n)$ lifted at $v$. Then $[\check v]$
is the set of  terminal points of shortest arcs in $aA_n$. Moreover, $[\check v]$ has pattern with parameter  $(L_n,\Delta_n,\omega_n)$.  In this way, we inductively  define the next generation $V_{l+1}$: 
 $$V_{l+1}:=\bigcup_{v\in V_l}[\check v].$$  
at most $K_n$-times, until $l+1= (\sum_{m=1}^{n} K_m+1)$.\\

\textbf{Escaping to the infinity.} We now go from $\Ax(g_n)$ via the bridge $\mathfrak b_n$ to the next $\Ax(g_{n+1})$. Let $V_l$ be the last generation produced. Each $v\in V_{l}$ is the starting point $a(\mathfrak b_n)_-$ of $a\mathfrak b_n$ on some lift $a\ax(g_n)$ with $a\in G$. Note that  the bridge $a \mathfrak b_{n+1}$ from $a\ax(g_n)$ to $a\ax(g_{n+1})$  might not terminate at $a(\mathfrak b_{n+1})_-$. 
Hence we define $\check v$ to be the starting point $a(\mathfrak b_{n+1})_-$ of  $a\mathfrak b_{n+1}$ on $a\ax(g_{n+1})$.  
The set of such $\check v$ forms $V_{l+1}$. By construction 
$V_{l+1}$ has the same cardinality as $V_{l}$.\\

To summarize, we perform the following operation for each $n$: 
\begin{enumerate}
\item concatenate  $K_n$ times appropriately-translated shortest arcs in $A_n$ with $\Ax(g_n)$,
\item go through the corresponding translated bridge $\mathfrak b_n$,  and
\item concatenate similarly the next $K_{n+1}$ arcs  in $A_{n+1}$   with   $\Ax(g_{n+1})$.
\end{enumerate}  
The terminal points of translated shortest arcs in $A_n$ and of translated $\mathfrak b_n$ form the  generation $V_l$, where $n $ and $ l$ are  related  by Equation (\ref{generationEq}). That is, $V_l$ consists of $G$-translated copies of the initial or terminal points of $\mathfrak b_n$. Every family path  $v_l\in V_l$ ($l\ge 0$)  gives an $L$-local $c$-quasi-geodesic path $\gamma$ in $X$, and their union gives a quasi-radial tree $T$ by Lemma \ref{LargeTreePtsVersion}.

\begin{prop}\label{CriterionHDNonConical2}
Let $g_n$ be an escaping sequence of loxodromic elements in $G$. Let $A_n$ be a set  of shortest   arcs from $\ax(g_n)$ to itself with length in the interval 
$[L_n- \Delta_n, L_n + \Delta_n]$.  Assume that the cardinality 
$|A_n|$ satisfies $|A_n|\ge \mathrm{e}^{\omega_n L_n}$. Further, assume that $L_n/\Delta_n\to 0$ as $n\to\infty$.
Set $\omega=\liminf_{n\ge 1} \omega_n$. 

If $\widetilde L_n-L_n-\Delta_n\to \infty$, then  $\HD(\ncG)\ge \omega/\epsilon$, where $\epsilon$ is the parameter for the visual metric in Lemma \ref{VisualMetric}.   
\end{prop}
\begin{proof}
Construction  \ref{sub-constr-actions} outputs a quasi-radial tree $T$ with parameters $(L_n,\Delta_n, \omega_n, K_n,B_n)$. If we choose a sufficiently large number $K_n$ of repetitions  of looping arcs in $A_n$ so that (\ref{ChoiceKnEq2}) is satisfied, then 
the Hausdorff dimension of ends of $T$ is at least $\omega/\epsilon$ by Proposition \ref{LargeTreePtsVersion}.

It remains to see that each end of $T$ is a non-conical limit point.  
By construction, let $\gamma$ be a quasi-geodesic ray marked by a family path as before the Proposition. The projection of $\gamma$ travels close to the escaping bridge $\mathfrak b_n$ for any $n\to\infty$. And looping around shortest arcs on $\Ax(g_n)$ may trace back at most $L_n+\Delta_n$. Since $d(o,\Ax(g_n))-L_n-\Delta_n\to +\infty$,
$\gamma$ is escaping. Thus,  the endpoint of  $\gamma$ is non-conical.  The proof is  complete.
\end{proof}

\subsection{Hausdorff dimension of non-conical points for normal covering}
Here is a way to obtain an escaping sequence of loxodromic elements. 

\begin{lem}\label{EscapingLoxoInNormalSubgroup}
Suppose that  $X$ is   a   hyperbolic space equipped with a proper isometric and non-elementary action of  $\Gamma$. Let $G<\Gamma$ be an infinite normal subgroup of infinite index. Then $G$ has an infinite sequence of loxodromic elements $h_n$ that is escaping.   
In fact, $\Ax_G(h_n)$ can be chosen to be of the form $\Ax_G(h_n)=g_n\Ax_G(h_1)$ for some $g_n \in \Gamma$. 
\end{lem}
\begin{proof}
An infinite normal subgroup $G$  contains infinitely many pairwise independent loxodromic elements. Let us fix one such $h_1\in G$.  As $\Gamma/G$ is infinite and the action $\Gamma\act X$ is proper, there exists a sequence of  right cosets $Gg_n$, $g_n\in \Gamma$ so that $d(Gg_no, Go)\to \infty$. Then $h_n:=g_n h_1g_n^{-1}$ are   loxodromic elements in $G$ with axis $\ax(h_n)=g_n\ax(h_1)$.  We claim that the axis of $h_n$  escapes to infinity, i.e.\ $d(o,\Ax_G(h_n))\to \infty$.  

Note that $\Ax_G(h_n)=G\cdot \ax(h_n)=Gg_n\ax(h_1)$.  Since   $\langle h_1\rangle o\subseteq Go$, it follows that   $$d(g_n \langle h_1\rangle  o, G o) \ge d(g_nGo, Ho)=d(Gg_no, Go)$$ and the last term  tends to $ \infty$ as $n$  tends to $ \infty$. Note that $\ax(h_1)$ stays within an $R$-neighborhood of $\langle h_1\rangle o$ for some $R>0$. Thus,    $\ax(h_n)=g_n \ax(h_1)\subset N_R(g_n\langle h_1\rangle o)\subset N_R(g_n Go)$. We then obtain $$d(\Ax_G(h_n),  o)=d(g_n \ax(h_1), G o)\ge d(g_nGo, Go)-R.$$ The last term tends to infinity, concluding the proof. 
\end{proof}

It would be interesting to note that escaping loxodromic elements also exist in confined subgroups.
\begin{defn}\label{def-conf}
A subgroup $G$ is called \textit{confined} in $\Gamma$ if there exists a finite subset $P$ in $\Gamma$ so that $gGg^{-1}$ intersects $P\setminus 1$ for any $g\in \Gamma$. The set $P$ is called the \emph{confining subset}.
\end{defn}

\begin{lem}\label{EscapingLoxoInConfinedSubgroup}
Suppose that  $X$ is  a   hyperbolic space equipped with a proper and non-elementary isometric action of  $\Gamma$.
Let $G$ be an infinite confined subgroup
of infinite index in $\Gamma$ with a finite confining subset $P$. Assume that each non-trivial element in $P$  is loxodromic.    Then $G$ has infinitely many loxodromic elements $h_n$ which escapes to infinity.   
\end{lem}
\begin{proof}
The proof follows a similar outline as  Lemma \ref{EscapingLoxoInNormalSubgroup}. 
As the index $[\Gamma:G]$ is infinite and the action $\Gamma\act X$ is proper, we take a sequence of  right cosets $Gg_n$ for $g_n\in \Gamma$ so that $d(Gg_no, Go)\to \infty$.

By definition of confined subgroups, for each   $g_n$ there exists $h_n\in G$ and $p_n\in P\setminus 1$ so that $g_n^{-1}h_ng_n=p_n$.    As $P$ is finite, we may assume $p_n=p$ for each $n\ge 1$ after passing to a subsequence. Thus, $g_npg_n^{-1}=h_n$. By assumption, $p$ is a loxodromic element. Hence  each $h_n$ is loxodromic  with axis $g_n\ax(p)$.  
 We then obtain $$d(\Ax_G(h_n),  o)=d(Gg_n \ax(p), o)\ge d(Gg_no, o)-d(Gg_no,Gg_n \ax(p))\ge d(Gg_no, o)-d(o,\ax(p)).$$ 
 The last term tends to infinity, concluding the proof. 
\end{proof}
 
The following is the main result of this subsection. It gives a lower bound on the Hausdorff dimension of non-conical limit sets for a large class of geometrically infinite groups.  

\begin{thm}\label{NormalSubgroupCase}
Suppose $\Gamma$ is a discrete    group acting on a hyperbolic space $X$. If $G$ is an infinite normal subgroup with infinite index in $\Gamma$, then 
$\HD (\Lambda G)=\HD (\Lambda^{nc}G) \ge \e G/\epsilon$, where $\ep$ is the parameter for the visual metric. 
\end{thm}
\begin{proof}
By Lemma \ref{EscapingLoxoInNormalSubgroup}, there exists an escaping sequence of  loxodromic elements $g_n\in G$  with the same stable translation length. Let $\gamma_n$ be an axis of $g_n$. For any $m, n\ge 1$, the set of shortest arcs from $\gamma_n$ to itself can be sent by an isometry to the set of shortest arcs from $\gamma_m$ to itself. (This property fails for escaping sequence of  loxodromic elements in Lemma \ref{EscapingLoxoInConfinedSubgroup}.) For any $L>0$, the set $A$ of such arcs with length in $[L-\Delta,L+\Delta]$ have cardinality $\mathrm{e}^{L \e \Gamma}$. 

Set $\widetilde L_n=d(o,\gamma_n)$ and $\Delta_n=\Delta$. Fix a bridge $\mathfrak b_n$ from $\gamma_n$ to $\gamma_{n+1}$. We fix any divergent sequence $L_n\to\infty$. Choose a set $A_n$ of shortest arcs from $\gamma_n $ to itself  with length in $[L_n-\Delta_n,L_n+\Delta_n]$ has cardinality approximately $\mathrm{e}^{L_n \e G}$.

Since $L_n$ is a fixed large constant independent of $\gamma_n$, we may extract a subsequence of $\gamma_n$  so that $\widetilde L_n-L_n-\Delta_n\to \infty$. Note that this may change the length $B_n$ of the bridge $\mathfrak b_n$ from $\gamma_n$ to $\gamma_{n+1}$ to larger values after passing to a subsequence.  We may then  choose a sufficiently large number $K_n$ of repetitions  of looping arcs in $A_n$ so that (\ref{ChoiceKnEq2}) is satisfied. This compensates for the effect of a larger $B_n$.   

Thus, all the conditions of Proposition \ref{CriterionHDNonConical2} are fulfilled. Hence $\HD (\Lambda^{nc}G) \ge \e G/\epsilon$. 
By Bishop-Jones' theorem \cite{BJ97}, the equality  $\HD (\Lambda G)=\max\{\HD (\Lambda^c G), \HD (\Lambda^{nc} G)\}$ implies $\HD (\Lambda^c G)=\e G/\epsilon$. Thus, we have $\HD (\Lambda^{nc}G) =\HD (\Lambda G)\ge \e G/\epsilon$. The conclusion follows.  
\end{proof}

 We equip  the Gromov boundary of  a complete simply connected Riemannian manifold
 of pinched negative curvature (or a CAT(-1) space) with the Bourdon  metric. Roughly speaking, the Bourdon  metric is a class of visual metrics where $\epsilon$ could be chosen to be $1$.   
\begin{cor}\label{cor-normalsubgp}
Let $N$ be a finite volume Riemannian manifold with pinched negative curvature. Let $M$ be an infinite sheeted regular cover of $N$. Set $G=\pi_1(M)$ and $\Gamma=\pi_1(N)$.  Then $\HD(\Lambda^{nc}G)=\HD(\partial\widetilde M)=\HD(\Lambda \Gamma)$.    
\end{cor}
\begin{proof}
Note that $\Lambda G=\Lambda \Gamma$ and  $\HD(\Lambda G) =\HD (\Lambda^c G)=\e G$ by Bishop-Jones' theorem \cite{BJ97}. If $\e G<\e \Gamma$, the proof finishes by noting $\HD (\Lambda G)=\max\{\HD (\Lambda^c G), \HD (\Lambda^{nc} G)\}$. Otherwise, $\e G=\e \Gamma$ and the proof is completed by Theorem \ref{NormalSubgroupCase}.    
\end{proof}
\begin{rem}
By the Amenability Theorem in \cite{CDST}, we have $\e G<\e \Gamma$ if and only if $\Gamma/G$ is non-amenable.  The real crux of the above statement and its proof lies in the case when $N$ is an amenable cover of $M$.     
\end{rem}

\subsection{Discrete groups acting on trees}
Let  $\Gamma$ be an infinite $d$-regular graph with $d\ge 3$ and $A$ be a finite subset of vertices in $V(\Gamma)$. Let $\partial A$ denote the set of edges $e$ such that $e$ connects $x, y$ with $x \in A$ and $y \in V(\Gamma) \setminus  A$.
The \textit{isoperimetric constant} of $\Gamma$ is given by
$$i(\Gamma) := \inf_{A} \frac{1}{d}\frac{|\partial A|}{|A|}$$
where the infimum  is over finite subsets $A$ of vertices in $V(\Gamma)$.
It is known that $i(T_d)=(d - 1)/d$ if $T_d$ is a $2d$-regular tree.

We say that the graph $\Gamma$ is \textit{amenable} if $i(\Gamma)=0$. That is to say, there exists a sequence of finite subsets $A_n$ called \textit{Folner sets}   with $|\partial A_n|/|A_n|\to 0$.

We now explain several relations between the isoperimetric constant,  the bottom of the  spectrum of the Laplacian on the graph $\Gamma$, and the co-growth of $\Gamma$.\\ 

\noindent\textbf{Mohar inequality.} Mohar \cite{Moh88} proved  an analog of the well-known Cheeger inequality  for infinite graphs.
Let $\lambda_0(\Gamma)$ be the  bottom of the spectrum of the discrete Laplacian   on $\Gamma$. Let $r(\Gamma)$ be the spectral radius of the simple random walk on $\Gamma$. It is known that $r(\Gamma)=1-\lambda_0(\Gamma)$, and $\lambda_0(T_d)=1 - (\sqrt{2d - 1})/d$.

\begin{prop}\cite[Theorem 3.1(b)]{Moh88}\cite[Theorem 5.5(b)]{MW89}\label{MoharIneq} We have the following inequalities.
$\frac{d-d\lambda_0(\Gamma)}{d-2+\lambda_0(\Gamma)}\le i(\Gamma) \le \sqrt{1 -  \lambda_0(\Gamma))^2}.$    
\end{prop} 

Proposition~\ref{MoharIneq} implies that $\lambda_0(\Gamma)\to 1$ if and only if $i(\Gamma) \to 0$. Moreover, a graph $\Gamma$ is amenable if and only if  $\lambda_0(\Gamma)=0$ or $r(\Gamma)=1$. \\

\noindent\textbf{Co-growth formula.}
The co-growth of $\Gamma$ is the growth rate of the fundamental group  $\pi_1(\Gamma)$ acting on the universal covering $\widetilde \Gamma$. (In a sense, this is a dual notion to the growth of the quotient $\Gamma=\widetilde \Gamma/\pi_1(\Gamma)$). More precisely,  let $H$ be a group acting isometrically and properly on a  $d$--regular tree $X$. The growth rate  $\e H$ of the action $H\act X$ is referred as the co-growth of $\U/H$. 

The co-growth formula of Grigorchuk \cite{GriH} relates  the co-growth  to 
  the spectral radius $r(\U/H)$ of the simple random walk on  the graph $\U/H$ as follows
\begin{equation}\label{Grigorchuk}
r(\U/H)=\begin{cases}
\frac{\sqrt{d-1}}{d}\left(\frac{\sqrt{d-1}}{\mathrm{e}^{\e H}}+\frac{\mathrm{e}^{\e H}}{\sqrt{d-1}}\right), &\mathrm{e}^{\e H} \ge \sqrt{d-1}\\
\frac{2\sqrt{d-1}}{d},&\mathrm{e}^{\e H} \le \sqrt{d-1}
\end{cases}    
\end{equation}

Fixing a basepoint $o\in X$, the space of ends of the tree $\U$ could be identified with the set of geodesic rays issuing from $o$. The visual metric between two geodesic rays $\alpha, \beta$   is defined to be $\mathrm{e}^{n}$, where $n$ is the length of $\alpha\cap \beta$. Then $\log (d-1)$ is the Hausdorff dimension of the space of ends of $\U$ equipped with the   visual metric.  

The following is the main theorem of this subsection.
\begin{thm}\label{thm-amenable-graph}
Assume that $G$ acts isometrically and properly on a regular tree $X$ of degree $d\ge 3$. Assume that $X/G$ is an infinite amenable graph. Then $\HD(\Lambda^{nc} G)=\log (d-1)$.    
\end{thm}
\begin{proof}
The graph $X/G$ is amenable, so there exists a sequence of {Folner sets} $F_n$  with $|\partial F_n|/|F_n|\to 0$. 
Here, we take each $F_n$ to be an induced subgraph on its vertex set. We may assume without loss of generality that $F_n$ is connected by choosing  connected components with minimal $|\partial F_n|/|F_n|$.  It is clear that $|F_n|\to\infty$. 

By passing to a subsequence,  we may assume further that $F_n$ is escaping. Indeed, if $d(o, F_n)$ is bounded for all $n$, we consider the sequence of subsets $F_n\setminus B(o,m)$ for a large fixed $m$. As $|F_n|\to \infty$ and $|B(o,m)|\ll |F_n|$, we may extract   a subsequence of $F_n$  so that  $F_n$  are still Folner sets.  Letting $m\to \infty$, a  diagonal argument produces an escaping subsequence of Folner sets $F_n$.  

We now complete $F_n$ to  a $d$-regular graph $\widetilde F_n$ by attaching trees. Namely, for each vertex $v$ in $F_n$ with degree $d_v$ less than $d$, we adjoin  $(d-d_v)$ subtrees of degree $d$ 
to $F_n$ by means of an edge to $v$. The degree of $v$ 
after adjunction is thus $d$. As a result,  $F_n$ is the core of $\widetilde F_n$.

By definition, the isoperimetric constant $i(\widetilde F_n)$ is at most $|\partial F_n|/|F_n|$, which tends to $0$ and thus $\lambda_0(\widetilde F_n)\to 1$ by Proposition \ref{MoharIneq}. By the above co-growth formula, the growth rate $\omega_n:=\e {H_n}$ of the action of the fundamental group $H_n:=\pi_1(\widetilde F_n)$ on $T$ tends to $\log (d-1)$. 

With the above preliminaries in place, we are ready to complete the proof analogous to that of Theorem \ref{NonConicalFromGeometricLimit}. 
First of all, we fix a hyperbolic element $h_n$ in $H_n$ for each $n\ge 1$. The axis of $h_n$ projects to an immersed loop $\gamma_n$ in $\widetilde F_n$. Since $F_n$ is a deformation retract of $\widetilde F_n$ obtained by collapsing the attached subtrees to the vertices in $F_n$, the immersed loop $\gamma_n$ is entirely contained in $F_n$. Thus, $h_n$ is an escaping sequence in $G$ (i.e. $\gamma_n$ is escaping in $X/G$).

 Let $\mathfrak b_n$ be the bridge given by a shortest  path from $\gamma_n$ to $\gamma_{n+1}$. In graphical terms,  $\mathfrak b_n$ is just an immersed path so that it intersects $\gamma_n$ and $\gamma_{n+1}$ only at the endpoints. Set $B_n=\ell(\mathfrak b_n)$.

By Theorem \ref{ShortestArcsonGraphs} applied to $\widetilde F_n$,  there is a family of  shortest paths $A_n$ from $\gamma_n$ to $\gamma_n$ with length in $[L_n-\Delta_n,L_n+\Delta_n]$, such that $A_n$ has cardinality at least $\mathrm{e}^{\omega_n L_n}$. By the same reasoning, each path in $A_n$ is immersed and is thus contained in $F_n$.   
Note that $\Delta_n$ depends on $\gamma_n$, but $L_n$ could be arbitrary large. We choose $L_n\to\infty$ so that $\Delta_n/L_n\to 0$ and further $K_n\to \infty$ so that $B_n/K_nL_n\to 0$ and (\ref{ChoiceKnEq2}) holds.

We then follow the earlier construction:  loop $K_n$ times about shortest arcs in $A_n$, and go to $\gamma_{n+1}$ via the bridge $\mathfrak b_n$, and loop $K_{n+1}$ times about the arcs in $A_{n+1}$.  By Lemma \ref{LargeTreePtsVersion}, the union of all these immersed  rays in $X$ lifts to a quasi-radial tree $T$. The Hausdorff dimension of the ends of $T$ is at least $\liminf_{n\to \infty} \omega_n$. As $\omega_n\to \log (d-1)$, this proves that $\HD(\partial T)\ge \log (d-1)$. Since $\partial T$ consists of non-conical limit points of $G$ and $\HD(\partial X)=\log (d-1)$,   the proof is complete.
\end{proof}

\subsection{Hyperbolic surfaces}

The \textit{Cheeger constant} of a (possibly infinite volume) Riemannian $n$-manifold $M$ is given by
$$h(M) := \inf \frac{\mathrm{Vol}_{n-1}(\partial D)}{\mathrm{Vol}_n(D)}$$
where $D \subseteq M$ is a smooth compact $n$-submanifold with boundary and $0 < \mathrm{Vol}_n(D) \le \mathrm{Vol}_n(M)/2$.  If there exists a submanifold $A$ so that 
$$
h(M)=\frac{\mathrm{Vol}_{n-1}(\partial A)}{\mathrm{Vol}_n(A)}
$$ then $\partial A$ and $A$ are referred as $(n-1)$ and $n$ dimensional \textit{Cheeger minimizers}. In analogy with amenable graphs,  if $h(M)=0$ we say that $M$ is \textit{amenable}. It is a classical result of Kanai \cite{Kanai} that if $M$ has bounded geometry, then amenability of $M$ is equivalent to the amenability of a graph whose vertices form a net in $M$. 

In what follows, we only consider  hyperbolic surfaces $S$. We may write $S$ as a space form $\mathbb H^2/G$ where $G\cong \pi_1(S)$ is a discrete subgroup in the isometry group of $\mathbb H^2$.  A \textit{cusp neighborhood} in $S$ means a neighborhood of an end of $S$ which is covered by a horoball in $\mathbb H^2$. A \textit{funnel} in $S$ is a neighborhood of an end of $S$ which is covered by a half-plane in $\mathbb H^2$.

The \textit{convex core} of $S$ is the minimal convex subsurface that is a deformation retract of $S$. Equivalently, it is the quotient of the convex hull of the limit set of $G$. We can explicitly obtain the convex core of $S$ by removing all the funnels of $S$.  A hyperbolic surface $S$ is called \textit{geometrically finite} if its convex core has finite area.

In \cite{AdamsM}, Adams-Morgan give a complete classification of \textit{boundary minimizers}  in geometrically finite hyperbolic surfaces, i.e. they find regions $A$ with fixed area and least length of boundary. Before stating their results, we introduce the following  terminology. 

Let $A$ be a (possibly disconnected) convex subsurface of $S$ bounded by simple closed geodesics. Given a real number $s$, an \textit{$s$-neighboring} of $A$ is the subset of $S$ obtained  by adding the $s$-neighborhood for each boundary component of $A$ if $s\ge 0$ or removing the $|s|$-neighborhood of  each boundary component of $A$ if $s<0$. Each boundary component of an   $s$-neighboring has constant curvature.

\begin{thm}\cite[Theorem 2.2]{AdamsM}\label{IsopermeterMinimizer} 
Let $S$ be a connected, geometrically finite hyperbolic surface. For a given $t \in (0,Area(S))$, a collection of embedded rectifiable curves bounding a region $A$ of area $t$ which minimizes $\partial A$ consists of regions of the following four types:
\begin{enumerate}
    \item 
     a metric ball,
    \item 
    a cusp neighborhood,
    \item 
    an $s$-neighboring of a closed geodesic,
    \item 
    or an $s$-neighboring of a convex subsurface for some $s\in \mathbb R$.
\end{enumerate}
Further, 
$l(\partial A) \le \sqrt{Area(A)^2 + 4\pi Area(A)}$
with equality in the case of a circle bounding a metric ball. If $S$ has at least one cusp, cases (1) and (3) do not occur and $l(\partial A) \le Area(A)$ with equality for horocycles. Finally, if $Area(A) < \pi$ and $S$ has cusps, then a minimizer consists of any collection of horoball neighborhoods of cusps with boundary having total length  $Area(A)$.
\end{thm}

\noindent\textbf{Elstrodt-Patterson-Sullivan-Corlette formula.} Let $\lambda_0(M)\ge 0$ be  the bottom of the $L_2$-spectrum for the Laplace-Beltrami operator on a Riemannian manifold $M$.  Alternatively, $\lambda_0(M)$ is  given by a variational formula 
$$
\lambda_0(M) = \inf_{f\in C_c^\infty(M)\setminus 0} \frac{\int_M |\mathrm{grad} f|^2}{\int_M f^2},
$$
where $C_c^\infty(M)$ denotes  compactly supported smooth functions on $M$. If $M$ is compact, then $\lambda_0(M)=0$.

Let $X$ be a rank-1 symmetric space (i.e. the real, complex, quaternionic hyperbolic spaces or the Cayley plane). We equip   the visual boundary $\partial X$ with the visual metric and denote by $\HD(\pU)$ the Hausdorff dimension. If $G$ is a lattice in $\isom(X)$, then $\HD(\pU)=\e G$.  The  Elstrodt-Patterson-Sullivan-Corlette formula below relates the growth rate to the  bottom of the spectrum (\cite{Sul,Corlette}): 
\begin{equation}\label{EPSC}
\lambda_0(\U/G)=\begin{cases}
\e G(\HD(\pU)-\e G), &\e G\ge \HD(\pU)/2\\
\frac{\HD(\pU)^2}{4},&\e G \le \HD(\pU)/2
\end{cases}    
\end{equation}

\noindent\textbf{Cheeger-Buser inequality.}
The  Cheeger-Buser inequality bounds the first non-zero eigenvalue from below and above in terms of the Cheeger constant $h(M)$.  Cheeger  showed  that if $\lambda_0(M)\ne 0$ for a Riemannian $n$-manifold $M$, then
$$\lambda_0(M) \ge h(M)^2/4.$$ 
When $M=X/G$ is locally symmetric, the above formula shows $\lambda_0(M)\ne 0$  if and only if  $\e G<\HD(\pU)$.  

If $M$ is a Riemannian $n$-manifold with Ricci curvature bounded below by $-\delta^2(n - 1)$ where $\delta\ge 0$, Buser then showed that   
$$\lambda_0(M) \le 2\delta(n - 1)h(M) + 10h^2(M)$$

We have the following improvement to the main result of \cite{FM01}. 
\begin{thm}\label{thm-cheeger-surface}
Assume that $G$ acts isometrically and properly on the hyperbolic plane $\mathbb H^2$. Assume that $\mathbb H^2/G$ has   Cheeger constant 0. Then $\HD(\Lambda^{nc} G)=1$.    
\end{thm}
\begin{rem}\label{rem-amenable-recurrent}
We say that a Riemann surface is of \textit{parabolic type} if  the Brownian motion on $\mathbb H^2/G$  is recurrent; equivalently $\mathbb H^2/G$ admits no  Green functions.    By \cite[Theorem 2.1]{HP97}, this happens exactly when the Poincar\'e series of $G$ is divergent at $1$ (and thus $\e G=1$). In \cite{FM01}, Theorem \ref{thm-cheeger-surface} is proved for  hyperbolic surfaces $\mathbb H^2/G$ with infinite area and of parabolic type.  As $\e G=1$,  the Cheeger constant is zero by the formula (\ref{EPSC}). Thus, a recurrent hyperbolic surface with infinite area must be amenable. Conversely, it is easy to construct an amenable hyperbolic surface with funnels, which thus admits transient Brownian motions.
\end{rem}
\begin{proof}
Let $\Sigma$ denote the hyperbolic surface $\mathbb H^2/G$.  It is known that the Cheeger constant of a geometrically finite hyperbolic surface is non-zero. Thus, $\Sigma$ must be of infinite type with infinite area. 

Let $S_n$ be a sequence of compact subsurfaces  in $\Sigma$, so that $l(\partial S_n)/Area(S_n)\to 0$. Since $\Sigma$ is of infinite type, $S_n$ must have non-empty boundary.  We may  assume that $S_n$ is \textit{essential}: no boundary components are peripheral and no two boundary components are homotopic. Indeed, if a boundary component bounds a disk or a cusp, we may include the disk or the cusp into $S_n$. Similarly, if two boundary components are homotopic, we add the bounding annulus to $S_n$. The resulting surfaces, still denoted by  $S_n$,   have $l(\partial S_n)/Area(S_n)\to 0$.  

We now  modify $S_n$ further as follows.
\begin{claim}
There exists an escaping sequence of convex compact subsurfaces $S_n^\star\subset\Sigma$ so that the Cheeger constant of  $S_n^\star$ tends to 0.    
\end{claim}
\begin{proof}[Proof of Claim:] 
Let $S_n$ be a sequence of compact subsurfaces as above. As $S_n$  is essential, the natural inclusion into $\Sigma$ induces an embedding at the level of fundamental groups. 
Consider the cover  $p: \tilde S_n\to \Sigma$ of $\Sigma$ associated to the subgroup $\pi_1(S_n)<\pi_1(\Sigma)$ (with a fixed basepoint).  As $S_n$  is a  finite type surface with boundary,  $\tilde S_n$ is a geometrically finite hyperbolic surface with infinite area. The convex-core $C(\tilde S_n)$ of  $\tilde S_n$ is the minimal convex subsurface which is a deformation retract. It is obtained from  $\tilde S_n$ by cutting out finitely many funnels with geodesic boundary. 

As $\pi_1(S_n)$ is identified with $p_\star(\pi_1(\tilde S_n))$, we may lift $S_n$ to get a subsurface $\hat S_n$ in $\tilde S_n$, such that  $\hat S_n$ is homeomorphic to $S_n$  and has  the same area  as $S_n$. Let us denote this area as $t_n$. By Theorem \ref{IsopermeterMinimizer}, the boundary {minimizer}  in $\tilde S_n$  with the area $t_n$ and with least boundary length  is a compact subsurface with constant curvature boundary. Equivalently, it is obtained from a convex compact subsurface $S_n^\star$ bounded by closed geodesics, in the following two ways. For some fixed $s\ge 0$,  \begin{enumerate}
    \item 
    either add  the $s$-neighborhood of each boundary component,
    \item 
    or remove the $s$-neighborhood of each boundary component.
\end{enumerate} By definition, the Cheeger constant $h(S_n^\star)$   of $S_n^\star$ is less than $l(\partial S_n)/Area(S_n)$. Thus,  $h(S_n^\star)$  also tends to $0$ as $n\to\infty$. In the case (2), we adjoin the removed $s$-neighborhood to  $S_n^\star$
and continue to denote the resulting surface as  $S_n^\star$. The boundary length decreases, so the Cheeger constant of $S_n^\star$  decreases and still tends to $0$ as $n\to \infty$. We project $S_n^\star$ to give the desired subsurface on  $\Sigma$ in the claim.  Moreover,  $h(S_n^\star)\to 0$ as $n\to \infty$ implies that the first three cases of Theorem \ref{IsopermeterMinimizer} are impossible.   

It remains to get an escaping sequence of the $S_n^\star$'s. If $S_n^\star$ is not escaping,  we may excise a large convex subsurface  from each $S_n^\star$  so that $S_n^\star$ is still convex and $d_\Sigma(o,S_n^\star)\to \infty$. As $Area(S_n^\star)\to\infty$ and the excised large  subsurface is of fixed area,  we may extract a subsequence so that $h(S_n^\star)\to 0$.  Hence the claims follows.
\end{proof}

We are ready to complete the proof along the lines of Theorem \ref{NonConicalFromGeometricLimit} or \ref{thm-amenable-graph}. 
By  formula (\ref{EPSC}), if  $H_n<G$ denotes the fundamental group of $S_n^\star$,   the critical exponent  $\omega_n$ of  the action $H_n\act \mathbb H^2$ tends to  $1$.

We fix the following:
\begin{itemize}
    \item a sequence of closed geodesics $\gamma_n$ in $S_n^\star$
    such that $\gamma_n$ is an escaping sequence on $\Sigma$,
    \item a bridge $\mathfrak b_n$, which is a shortest  path from $\gamma_n$ to $\gamma_{n+1}$. 
\end{itemize}  

We complete $S_n^\star$ to  a complete hyperbolic surface $\overline S_n$, by adjoining funnels along closed geodesics. In other words, $S_n^\star$ is the convex-core of the completion.
By Theorem \ref{ShortestArcsonMfd} applied to $\overline S_n$, there exists  a family of  shortest paths $A_n$ from $\gamma_n$ to $\gamma_n$ with length in $[L_n-\Delta_n,L_n+\Delta_n]$, such that $A_n$ has cardinality at least $\mathrm{e}^{\omega_n L_n}$. As the convex core $S_n^\star$  contains all closed geodesics in  $\overline S_n$ and thus every shortest arcs between them, we see that $A_n$ is entirely contained in $S_n^\star$ for any $L_n$. 

Note that $\Delta_n$ depends on $\gamma_n$, but $L_n$ could be arbitrary large. We choose $L_n\to\infty$ so that $\Delta_n/L_n\to 0$ and further $K_n\to \infty$ so that $B_n/K_nL_n\to 0$ and (\ref{ChoiceKnEq2}) holds.

We then follow the construction in \textsection \ref{sub-constr-mfld}:  loop $K_n$ times about the shortest arcs in $A_n$, and go to $\gamma_{n+1}$ via the bridge $\mathfrak b_n$, and loop $K_{n+1}$ times  about the arcs in $A_{n+1}$.  By Lemma \ref{LargeTreePtsVersion}, the union of all these   rays   lifts to a quasi-radial tree $T$ in $\mathbb H^2$. The Hausdorff dimension of the ends of $T$ is at least $\liminf_{n\to \infty} \omega_n$. As $\omega_n\to 1$, this proves that $\HD(\partial T)\ge 1$ and thus $\HD(\partial X)=1$.  The proof is complete.
\end{proof}

\section{Hausdorff dimension of non-conical points of Kleinian groups}\label{sec-nonconical-kl}
For the purposes of this section, $\Gamma$ will denote a finitely generated geometrically infinite Kleinian group. Let $M^h = \Hyp^3/\Gamma$. If $\Gamma$ has parabolics, then $M^h$ has cusps. Let $M$ denote $M^h$ minus a small neighborhood of the cusps. 
We assume that the neighborhoods of distinct cusps are chosen small enough to have disjoint closures. Note that $M=M^h$ if $\Gamma$ has no parabolics.

\begin{defn}\label{def-truncate}
	Let $M^h$ be either a complete hyperbolic manifold or a convex
    codimension zero submanifold of a complete hyperbolic manifold.
    Then $M^h$ minus a small neighborhood of its cusps (if any) , denoted by $M$, will be referred to as a \emph{truncated} 
	hyperbolic manifold. We shall refer to $M$ as the \emph{truncation} of $M^h$.
\end{defn}

The purpose of Definition~\ref{def-truncate} is to be able to deal with hyperbolic 3-manifolds with or without cusps on the same footing.
We now state the main Theorem of this section.

\begin{thm}\label{thm-kl-nonconical}
	Let $\Gamma$  denote a finitely generated geometrically infinite Kleinian group, $M^h = \Hyp^3/\Gamma$, and  $M$ denote the associated truncated 3-manifold. Then there exists an unbounded sequence of points $x_n \in M$, such that $(M,x_n)$ converges geometrically to  a geometrically infinite 
	truncated hyperbolic 3-manifold $N$. Further, if $\Gamma_\infty$ denote the Kleinian group associated to $N$, then the limit set {$\Lambda_{\Gamma_\infty}$} equals the whole Riemann sphere.
\end{thm}

We refer the reader to \cite[Ch. 8,9]{thurstonnotes} for the original source on geometric limits and to \cite[Section 3]{mms} for an exposition suited to the needs of the present paper.
Combining Theorem~\ref{thm-kl-nonconical} with Theorem~\ref{NonConicalFromGeometricLimit}, we immediately have the following.
\begin{cor}\label{cor--kl-nonconical}
	Let $\Gamma$  denote a finitely generated geometrically infinite Kleinian group, and $M^h = \Hyp^3/\Gamma$. Then the Hausdorff dimension of non-conical points for {$\Gamma$} equals 2. 
\end{cor}

The rest of this section is devoted to proving Theorem~\ref{thm-kl-nonconical}. The geometric limit $N$ that is the output of Theorem~\ref{NonConicalFromGeometricLimit} has more structure, and is a  variant of a doubly degenerate hyperbolic 3-manifold. The proof will involve a detour through models of ends of geometrically infinite manifolds, notably
the work of Minsky from \cite{minsky-elc1}.  
For purposes of exposition, we split the proof  into two cases.
\begin{enumerate}
	\item The bounded geometry case, where the truncated manifold $M$ has injectivity radius bounded below by some $\ep > 0$,
	This is dealt with in Section~\ref{sec-bddgeo} below.
	\item The complementary unbounded geometry case, where no such lower bound exists. 
\end{enumerate}
The bounded geometry case is considerably easier (see Corollary~\ref{cor--kl-nonconical-bddgeo} below) and demonstrates some features of the general case.

Let $M$ be the truncation of the convex core of a  hyperbolic 3-manifold in the sense of Definition~\ref{def-truncate}.  The resolution  of the tameness conjecture
\cite{agol-tameness, gab-cal} shows that  any end $E$ of $M$ has a neighborhood 
homeomorphic to $S \times [0, \infty )$, where $S$ is a compact surface, possibly with boundary. In other words ends of hyperbolic 3-manifolds are {\bf topologically tame}. Further, Thurston-Bonahon \cite{thurstonnotes,bonahon-bouts} and Canary \cite{canary} establish that topologically tame  ends are geometrically tame, i.e.\ there exists a sequence of pleated surfaces exiting them. However, the geometry of such ends can be quite complicated. We shall now proceed to 
describe model geometries of ends of hyperbolic 3-manifolds following \cite{minsky-bddgeom, minsky-elc1, minsky-elc2, mahan-bddgeo, mahan-ibdd, mahan-amalgeo, mahan-split}. For now, we start with the following general definition. We do not specify for now what a \emph{prescribed geometry} is. 
For now, it will suffice for the reader to assume that any prescribed geometry specifies a finite or countable collection of metrics on $S \times [0,1]$ for $S$ a  fixed truncated hyperbolic surface.

\begin{defn} \label{def-glue-end2end}
	We say that a geometrically infinite end $E$ of a truncated hyperbolic 3-manifold $M$ is built up of blocks of some prescribed geometries \emph{glued end to end}, if 
	\begin{enumerate}
		\item $E $ is homeomorphic to $S \times [0, \infty)$, and 
		\item There exists $L \geq 1$ such that $S \times [i, i+1]$ (equipped with the metric induced from $E$) is  $L-$bi-Lipschitz to a block of the prescribed
		geometry.
	\end{enumerate} 
	We shall refer to	$S \times [i, i+1]$ as the { $(i+1)$th block} of  $E$.	
\end{defn}

\subsection{The bounded geometry case}\label{sec-bddgeo}

\begin{defn}\label{def-bddgeo}\cite{minsky-bddgeom, minsky-jams}
	An end $E$ of a truncated hyperbolic $M$ has \textit{bounded geometry}
	if there exists $\ep >0$ such that the injectivity radius of $M$ at $x \in M$ is bounded below by $\ep$ for all $ x \in M$.
\end{defn}

\begin{defn}\label{def-bddgeoblock} Let $S$ be a fixed truncated hyperbolic surface. Equip $B_0 = S \times [0,1]$ with the product metric.
	If $B$ is $L-$bi-Lipschitz homeomorphic to $B_0$, for some $L\geq 1$, it is called an $L-$thick  block.   
	
	If a geometrically infinite end $E$ is built up of $L-$thick  blocks 
	glued end to end (in the sense of Definition~\ref{def-glue-end2end}) 
	for some $L \geq 1$ then we say that $E$ admits an $L-$thick bounded geometry model.
	If a geometrically infinite end $E$ admits an $L-$thick bounded geometry model for some $L\geq 1$, we say that $E$ admits a \textit{bounded geometry model}.
\end{defn}

In the following Definition, we do not assume that $E$ admits a bounded geometry model.
This notion will be used in Section~\ref{sec-splitt}.
\begin{defn}\label{def-bddgeoblocksub} Let $E$ be any geometrically infinite end.
	Let $\Sigma$ be an essential subsurface of $S$. if $\Sigma \times [0,n] \subset E$, 
	$n \in \natls$  equipped with the metric induced from $E$
	is built up of $L-$thick  blocks of the form $\Sigma \times [i,i+1]$
	glued end to end (in the sense of Definition~\ref{def-glue-end2end}), then we say that  $\Sigma \times [0,n]$ admits a  \textit{bounded geometry sub-model of length $n$}.
\end{defn}

The following statement is now a consequence of work of Minsky
\cite{minsky-top,minsky-jams} (see also \cite{mitra-trees, mahan-bddgeo}).
\begin{thm}\label{thm-minsky-bddgeo}
	Let $E$ be an end of a truncated hyperbolic $M$ such that $E$ has bounded geometry in the sense of Definition~\ref{def-bddgeo}. Then there exists $L \geq 1$ such that $E$ admits an $L-$thick bounded geometry model.
\end{thm}

We then observe the following.  
\begin{prop}\label{prop-bddgeoseq}
	Let $E$ be a simply degenerate end of a truncated hyperbolic $M$ such that 
	$E\big(\cong {S \times [0,\infty)}\big)$   has bounded geometry. Let $x_n \in E$ such that $x_n \to \infty$. Then, after passing to a subsequence if necessary, $(M,x_n)$ geometrically converges to $N$ where $N$ is 
	a (truncated) doubly degenerate hyperbolic manifold of bounded geometry homeomorphic to
	$S \times (-\infty,\infty)$.
\end{prop}

\begin{proof}
	The proof is essentially the same as that in \cite[Remark 3.2]{mms}. Let $N$ be a (subsequential) geometric limit of $(M,x_n)$. Since $x_n \to \infty$, we can assume, by passing to a further subsequence if necessary, that $N$ is also a geometric limit of compact hyperbolic manifolds of the form $S \times [-n,n]$, where $x_n \in S \times \{0\}$. Further,  by Theorem~\ref{thm-minsky-bddgeo}, there exists $L \geq 1$ such that each $S \times [i, i+1]\subset S \times [-n,n]$ is an $L-$thick block. Passing to the limit, it follows that 
	$S \times [-n,n]$, and hence $(M,x_n)$ converges to a hyperbolic manifold $N$ homeomorphic to $S \times (-\infty,\infty)$ admitting an $L-$thick bounded geometry model. It follows again from work of Minsky \cite{minsky-jams,minsky-bddgeom} that  $N$ is of bounded geometry. Thus, 
	$N$ is 
	a (truncated) doubly degenerate hyperbolic manifold of bounded geometry homeomorphic to
	$S \times (-\infty,\infty)$.
\end{proof}

We thus have the following special case of Corollary~\ref{cor--kl-nonconical}.

\begin{cor}\label{cor--kl-nonconical-bddgeo}
	Let $\Gamma$  denote a finitely generated geometrically infinite Kleinian group, and $M^h = \Hyp^3/\Gamma$. Further, assume that one of the geometrically infinite ends of $M^h$ (after truncation if necessary) has bounded geometry.
	Then the Hausdorff dimension of non-conical points for $M^h$ equals 2. 
\end{cor}

\begin{proof}
	Let $E$ be the truncation of the geometrically infinite end of bounded geometry.
	Let $x_n \in E$ such that $x_n \to \infty$.  Then by Proposition~\ref{prop-bddgeoseq}, after passing to a subsequence if necessary, $(M,x_n)$ geometrically converges to $N$ where $N$ is 
	a (truncated) doubly degenerate hyperbolic manifold of bounded geometry homeomorphic to
	$S \times (-\infty,\infty)$. Theorem~\ref{NonConicalFromGeometricLimit} now applies to furnish the conclusion.
\end{proof}

\subsection{i-bounded Geometry} \label{sec-ibdd} The next model geometry is satisfied by degenerate Kleinian punctured-torus groups as shown by Minsky in \cite{minsky-torus}.

\begin{defn}\cite{mahan-ibdd}\label{def-ibdd}
	An end $E$ of a hyperbolic $M$ has \emph{i-bounded geometry} 
	if the boundary torus of every Margulis tube in $E$ has bounded diameter. 
\end{defn}

We will need to generalize Definition~\ref{def-ibdd} to allow rank 2 cusps in place of Margulis tubes.
Towards this, we need an i-bounded geometry analog of Definition~\ref{def-bddgeoblock}.
Fix a truncated hyperbolic surface $S$. Let 
$\CC=\{\sigma_i\}$ be a finite collection of disjoint simple closed geodesics on
$S$. Let
$N_\epsilon ( \sigma_i )$  denote the $\epsilon-$neighborhood of
$\sigma_i$, ($\sigma_i \in \mathcal{C}$), where we choose $\epsilon$ 
small enough so
that the neighborhoods are disjoint.

\begin{defn}\label{def-dthinblock}
	Let $S, \CC, \sigma_i, \ep$ be as above.
	Let $I = [0,3]$. Equip $S \times I$ with the product metric. Let
	$B^c = (S \times I - \cup_i N_{\epsilon} ( \sigma_i ) \times
	[1,2]$.  Equip $B^c$ with the induced path-metric.
	Then $B^c$ is referred to as a \emph{drilled thin block.}
	
	Let $\Sigma$ be an essential subsurface of $S$. Repeat the above construction with $S$ replaced by $\Sigma$. Then the output of this construction will be referred to 
	as a \emph{drilled thin block} associated to $\Sigma \subset S$.
\end{defn}

We now proceed to Dehn fill a drilled thin block.
For each resultant torus component $\T_i$ of
the boundary of $B^c$, perform Dehn filling on some $(1,n_i)$ curve by attaching a 
solid torus  $\Theta_i$  whose meridian is the $(1,n_i)$ curve. Let $B$ denote the
result of Dehn filling. Note that $B$
is homeomorphic to $S \times I $.
Note that the $n_i$'s are allowed to be quite arbitrary.
We refer to  $n_i$ as a \emph{twist coefficient}. Equip $\Theta_i$ with a hyperbolic 
metric such that it is foliated by totally geodesic  hyperbolic disks whose centers lie on a core geodesic in $\Theta_i$. 

\begin{defn}\label{def-thinblock}
	The resulting copy of $S \times I$ obtained, equipped with the metric just
	described, is called a \emph{filled thin block}, or simply a \emph{thin  block}. 
	
	The hyperbolic solid torus $\Theta_i$ is referred to as a \emph{Margulis tube} of the 
	thin  block.
	
	Let $\Sigma$ be an essential subsurface of $S$. Repeat the above construction with $S$ replaced by $\Sigma$, i.e.\ perform the Dehn filling on a drilled thin block
	associated to $\Sigma \subset S$ (in the sense of Definition~\ref{def-dthinblock}). Then the output of this construction will be referred to 
	as a \emph{filled thin block} or simply a \emph{thin  block} associated to $\Sigma \subset S$.
\end{defn}

\begin{defn}\label{def-ibddgeomodel}
	An end $E$ of a hyperbolic 3-manifold $M$  is said to admit an \emph{i-bounded geometry model} if 	it is bi-Lipschitz homeomorphic to a model manifold $E_m$ 
	consisting of gluing  thick and thin blocks end-to-end.
\end{defn}

The following statement is a consequence of the model in
\cite[Section 9]{minsky-elc1} and the bi-Lipschitz model theorem of \cite{minsky-elc2}. The  complex structure for boundary tori of Margulis tubes is encoded in terms of certain
meridian coefficients that are of the form $a+bi$, where $a, b$ are integers. If $a, b$ are both uniformly bounded for all blocks, we get back
the models of bounded geometry. If there is  a uniform bound on only the imaginary part of
these coefficients, we obtain the models of i-bounded geometry. See Figure~\ref{fig-ibdd} below.

\begin{prop}\label{prop-ibdd-model}
	An end $E$ of a hyperbolic 3-manifold $M$ has
	{i-bounded geometry} in the sense of Definition~\ref{def-ibdd} if and only if
	it admits an i-bounded geometry model in the sense of Definition~\ref{def-ibddgeomodel}.
\end{prop}

\begin{figure}[H]
	
	\includegraphics[height=5cm]{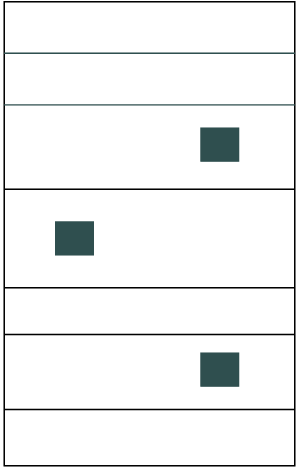}
	
	\smallskip

	\caption{Model of i-bounded geometry: black squares denote Margulis tubes \cite{mahan-ibdd}}
	\label{fig-ibdd}
\end{figure}

Recall that an end $E$ of a truncated hyperbolic manifold is homeomorphic to
$S \times [0, \infty)$, where 
 $S$ is a  topological surface, possibly with boundary, underlying a truncated hyperbolic surface.
Let $J$ denote either $(-\infty,\infty)$ or $[0,\infty)$. Let $J \cap \Z = J_\Z$ denote the integer points in $J$. 
Let $S_\Z := S \times (J_\Z + \frac{1}{2}) \subset S \times J$. Let $\CC$ be some collection of  simple closed curves contained in 
$S_\Z $, such that for all $n \in \Z$, the collection of curves in $\CC$ contained 
in  $S \times \{n+ \frac{1}{2} \}$ are disjoint, We shall then refer to 
$S \times J$ minus a small neighborhood of the curves $\sigma \in \CC$ as a \emph{drilled product of $S$ and $J$.} The closures of the small neighborhoods are required to be disjoint, and contained in $S \times [n+\frac{1}{3}, n+\frac{2}{3}]$ for some $n$. 
 
\begin{defn}\label{def-genlzdibddgeo} 
	Let $E$ be a truncated hyperbolic manifold homeomorphic to a drilled product of $S$ and $J$ for some $S,J$ as above.
	Suppose that $E$ is bi-Lipschitz homeomorphic to a model manifold $E_m$ built out of
	\begin{enumerate}
		\item $L-$thick  blocks for some $L \geq 1$ in the sense of Definition~\ref{def-bddgeoblock},
		\item drilled thin blocks in the sense of Definition~\ref{def-dthinblock}, and 
		\item thin blocks in the sense of Definition~\ref{def-thinblock}, 
	\end{enumerate}
	glued end to end (in the sense of Definition~\ref{def-glue-end2end}). Then we say that $E$ admits a \emph{generalized i-bounded geometry model}.
\end{defn}

The difference between an  i-bounded geometry model (Definition~\ref{def-ibddgeomodel}) and  a
generalized i-bounded geometry model (Definition~\ref{def-genlzdibddgeo}) is that in the latter drilled thin blocks are allowed. 
{Such generalized i-bounded geometry models arise naturally as follows.
Start with a degenerate end $E^h$ of a hyperbolic manifold $M^h$. Let $E$ denote the
truncation of $E^h$. Note that $E$ is obtained from $E^h$ by removing a small neighborhood of rank one cusps. 
Assume that $E$ has i-bounded geometry. Then,
after removing the interiors of some disjoint Margulis tubes from $E$ we obtain a manifold of generalized i-bounded geometry.}

In Section~\ref{sec-splitt} below, we shall need  a notion of generalized i-bounded geometry sub-models associated
to essential subsurfaces (cf.\ Definition~\ref{def-bddgeoblocksub}). We point out that in Definition~\ref{def-ibddgeoblocksub} below, as in Definition~\ref{def-bddgeoblocksub}, we do not impose any restriction on the model geometry of $E$ itself.

\begin{defn}\label{def-ibddgeoblocksub} Let $E$ be any geometrically infinite end.
	Let $\Sigma$ be an essential subsurface of $S$. If $\Sigma \times [0,n] \subset E$, 
	$n \in \natls$   equipped with the metric induced from $E$
	is built up of $L-$thick and filled thin  blocks of the form $\Sigma \times [i,i+1]$
	glued end to end (in the sense of Definition~\ref{def-glue-end2end}), then we say that  $\Sigma \times [0,n]$ admits an  \textit{i-bounded geometry sub-model of length $n$}.
\end{defn}

\begin{prop}\label{prop-geoltofibddgibdd}
	Let $E^h$ be a degenerate end of i-bounded geometry, and let $x_n \in E^h$ be a sequence of points such that each $x_n$ lies in the thick part of $E^h$
	and $x_n \to \infty$ in $E^h$. 
	After passing to a subsequence if necessary, assume that $(N^h, x_\infty)$ 
	is the geometric limit of $(E,x_n)$. Then the truncation $N$ admits a model of
	generalized i-bounded geometry.
\end{prop}

\begin{proof}
	By Proposition~\ref{prop-ibdd-model}, the truncation $E$ of $E^h$ admits a model of i-bounded geometry. 
    Since each $x_n$ lies in the thick part of $E^h$, it follows that $x_n$ lies in the thick part of $E$, and $x_n \to \infty$ in $E$. 
	
	Note now that 
	a geometric limit of a drilled thin block continues to be a drilled thin block. This continues to be true for any finite concatenation of drilled thin blocks glued end to end. Hence, away from Margulis tubes, any geometric limit of a sequence $\{(E,x_n)\}$ admits
	a model of generalized i-bounded geometry.
	Further, the boundary $\partial \mathbb T$ of any Margulis tube $\mathbb T$ has uniformly bounded  diameter in $E$ and hence in
	$N$. It follows that 
	$N$ admits a model of
	generalized i-bounded geometry. 
    
    {To see the last claim, consider a geometrically convergent sequence of Margulis tubes
    $\mathbb T_n$ in $\{(E,x_n)\}$. As noted above, $\partial \mathbb T_n$ has uniformly bounded  diameter in $E$. If, 
    in addition,  $\mathbb T_n$ has uniformly
    bounded diameter, then so does the limiting  Margulis tube $\mathbb T_\infty$. On the other hand, if $\mathbb T_n$ has  diameter tending to infinity, then any geometric limit 
    $\mathbb T_\infty$ gives a \emph{rank two} cusp whose boundary $\partial \mathbb T_\infty$ is a torus. It is
    precisely in  the latter case, that truncation yields a
    {drilled} thin block, and hence a model
    generalized i-bounded geometry. In the former case, the
    limiting block is simply a thin block.} 
\end{proof}

We are now in a position to state the main technical theorem of this section. The proof will occupy the rest of this section.

\begin{thm}\label{thm-maintechnicalgeolt}
	Let $E^h$ be a degenerate end of a  hyperbolic $M^h$, so that $E^h$ is homeomorphic to $S^h \times [0,\infty)$, where $S^h$ is a complete hyperbolic surface possibly with cusps. Let $E, M, S$ denote the truncations of $E^h, M^h, S^h$ respectively, so that $E$ is homeomorphic to $S \times [0,\infty)$.
    There exists a sequence $x_n \in E (\subset M \subset M^h)$ such that $(M^h,x_n)$ geometrically converges to a complete hyperbolic 3-manifold $N^h$, such that the following holds.\\
    There exists an essential subsurface $\Sigma$ of $S$ such that the truncation $N$ of $N^h$ satisfies the following:
	\begin{enumerate}
		\item $N$ is homeomorphic to a drilled product of $\Sigma$ and $J$, where   $J=\R$,
		\item  $N$ admits a {generalized i-bounded geometry model} $N_m$.
	\end{enumerate}
\end{thm}

We can now complete the proof of Theorem~\ref{thm-kl-nonconical} modulo Theorem~\ref{thm-maintechnicalgeolt}.

\begin{proof}[Proof of Theorem~\ref{thm-kl-nonconical} assuming Theorem~\ref{thm-maintechnicalgeolt}:] $ $\\
	Note that, in Theorem~\ref{thm-maintechnicalgeolt}, $N$ is a deformation retract of $N^h$. Further, $N^h$ is allowed to have infinitely generated fundamental group $\pi_1(N)$. Indeed, each rank 2 cusp of $N^h$ corresponds to a torus boundary component of a drilled thin block of $N_m$. 
	
	Next we observe that if $\rho(\pi_1(N))$ is the Kleinian group with $N^h = \Hyp^3/\rho(\pi_1(N))$, then its limit set is necessarily 
	all of the sphere $S^2$ at infinity. 
	Indeed, $N$ admits sequences of closed geodesics exiting 
	in the $+\infty$ and  $-\infty$ directions of $J=\R$ since $N$ is homeomorphic to a drilled product of $S$ and $J$ by Theorem~\ref{thm-maintechnicalgeolt}. Hence 
	$N^h$ equals its own convex core, and so the limit set of $\rho(\pi_1(N))$  is  the sphere $S^2$ at infinity.
\end{proof}

To prove  Theorem~\ref{thm-maintechnicalgeolt}, we shall
\begin{enumerate}
	\item recall the necessary background on the model geometry of ends from \cite{minsky-elc1,minsky-elc2} in Section~\ref{sec-splitt},
	\item use this background to construct the relevant geometric limit in Section~\ref{sec-geolts}.
\end{enumerate}

\noindent {Scheme of proof of  Theorem~\ref{thm-maintechnicalgeolt}:}\\
For now, we indicate the two major steps of the argument referring the reader to 
Section~\ref{sec-splitt} for necessary background on hierarchies and model geometries.
Given a degenerate end $E$ and its ending lamination, Minsky associates  with it a hierarchy $\HH$ of tight geodesics $g_Y$ corresponding to essential subsurfaces $Y$ of $S$. Let $\zeta(Y)$ denote the complexity of the subsurface $Y$. (Recall that $\zeta(Y)$ equals 3 times the genus of $Y$ plus the number of boundary components.)
Then there exists a minimal $\zeta_0 \geq 4$ such that any tight geodesic 
$g_Y \in \HH$ supported on a subsurface $Y$ of complexity \emph{strictly less than $\zeta_0$} is bounded independent of $Y$. 

\begin{proof}[Proof of Theorem~\ref{thm-maintechnicalgeolt} when $\zeta_0=4$.] We provide here a proof of Theorem~\ref{thm-maintechnicalgeolt} when $\zeta_0=4$ so that the main idea is explicated without getting into technicalities.
Note that an essential surface of complexity $\zeta_0=4$ is either a  4-punctured sphere, or a punctured torus.

Thus, there exist subsurfaces $Y_i\subset S$ of complexity $\zeta(Y_i)=4$ and tight geodesics $g_i \in \HH$ supported on $Y_i$ with length $\ell(g_i)$ tending to infinity. Such tight geodesics are referred to as 4-geodesics.
In this case, after passing to a subsequence if necessary, we can assume that each $Y_i$ is a copy of  $\Sigma$, where  $\Sigma$ is either a truncated 4-punctured sphere, or a truncated punctured torus.
Further, the combinatorial model manifold for $E$ contains
combinatorial sub-models $E(Y_i)$ consisting of $\ell(g_i)$ drilled thin blocks
(in the sense of Definition~\ref{def-dthinblock}) glued end to end. Let $N_i$ denote this concatenation of drilled thin blocks.
Each $N_i$ is obtained from $\Sigma \times [0, \ell(g_i)]$ after drilling.
We choose a sequence $x_i$ such that $x_i$ lies in the $[\ell(g_i) / 2]-$th block and let $i \to \infty$. Then the geometric limit of $(M, x_i)$ equals the geometric limit of $(E,x_i)$. Recall that $M$ (resp. $E$) is the truncation of $M^h$ (resp.\ $E^h$).
Finally, the geometric limit of $(E,x_i)$ agrees with the geometric limit 
$(N_\infty, x_\infty)$ of 
$(N_i,x_i)$ away from Margulis tubes.  Proposition~\ref{prop-geoltofibddgibdd} now furnishes the conclusion when $\zeta_0=4$. 
\end{proof} When $\zeta_0> 4$, the proof is similar and will be given in Section~\ref{sec-geolts} below.

\subsection{Hierarchies, subsurface projections and the Ending Lamination Theorem} \label{sec-splitt}
In this subsection, we quickly recall 
the essential aspects of hierarchies and subsurface projections from \cite{masur-minsky,masur-minsky2, minsky-elc1,minsky-elc2} that we shall need. 
We shall cull out, particularly from \cite[Sections 8,9]{minsky-elc1}, the 
necessary aspects of the relationship between subsurface projections and the combinatorial model manifold built there. In \cite{minsky-elc2}, it is established that this 
combinatorial model  is bi-Lipschitz homeomorphic to the truncation of a simply or doubly degenerate hyperbolic 3-manifold with the same end-invariants. We refer to 
\cite{minsky-elc1} for details.

Recall that for a compact surface $S(=S_{g,b})$ of genus $g$  with $b$ boundary components, $\xi ( S_{g,b} ) =
3g+b$ denotes its 	{\it complexity}. Let $Y$ be an essential subsurface  of $S$
(possibly an annulus).   Its curve complex is denoted as $\CC(Y)$, and its arc-and-curve complex by $\AAA\CC(Y)$.
The distance
in $\CC(Y)$ will be denoted as $d_Y$. Also, if $\eta$ is a simple closed curve or a lamination,
$\eta|_Y$ will denote its projection to $\AAA\CC(Y)$.
By performing surgery on the arcs of $\eta|_Y$ along boundary components of $Y$ 
(cf.\ \cite[Section 2.2]{minsky-bddgeom}) we obtain an element of $\CC(Y)$ that we refer to as the \emph{subsurface projection} of $\eta$ to $Y$. We denote it as $\pi_Y(\eta)$.

Now, let $E$ denote a truncated simply degenerate  end of a complete hyperbolic 3-manifold $M^h$. Then $E$ is homeomorphic to $S \times [0,\infty)$. Let $\tau$ denote a marking 
on $S \times \{0\}$, and $\LL$ denote the ending lamination for $E$. 
The following theorem can be culled out of \cite[Theorem 9.1]{minsky-elc1}. It characterizes bounded geometry ends $E$ (see also \cite[p. 150-151]{minsky-bddgeom}).

\begin{thm} \label{thm-bddgeoss}
	The truncated end $E$  is of bounded geometry if and only if there exists $D>0$ such that for every proper essential subsurface $W$ of $S$ (including annular domains), $d_W(\tau,\LL) \leq D$.
\end{thm}

A similar characterization of i-bounded geometry ends $E$ can be culled out of \cite[Theorem 9.1]{minsky-elc1}.

\begin{thm} \label{thm-ibddgeoss}
	The truncated end $E$  is of i-bounded geometry if and only if there exists $D>0$ such that for every proper \emph{non-annular} essential subsurface $W$ of $S$, $d_W(\tau,\LL) \leq D$.
\end{thm}

More information can be culled out of \cite[Theorem 9.1]{minsky-elc1}. 
Let  $E_m$ denote the combinatorial model manifold for $E$ constructed in \cite{minsky-elc1}. It is established in \cite{minsky-elc2} that there exists a bi-Lipschitz homeomorphism
$\Phi: E_m \to E$.

We refer the reader to \cite{masur-minsky2} for details about hierarchies and to \cite[pgs. 6-8]{minsky-elc1} for a quick overview of the construction of the combinatorial model $E_m$.
For the pair
$(\tau,\LL)$, let $\HH(\tau,\LL)$ denote the associated hierarchy of geodesics (the existence of $\HH(\tau,\LL)$ is guaranteed by
\cite[Lemma 5.13]{minsky-elc1}).  
For our purposes, we shall need the following:
$\HH(\tau,\LL)$  consists  of a family of tight geodesics $g_Y$ supported on essential
non-annular subsurfaces $Y$ of $S$. In \cite[Lemmas 5.7, 5.8]{minsky-elc1}, Minsky constructs a \emph{resolution} of $\HH(\tau,\LL)$, i.e.\ a sequence of markings, separated by elementary moves, sweeping through $\HH(\tau,\LL)$. 

For a subsurface $Y$,
let $u_Y, v_Y$ denote initial and terminal vertices for $g_Y$ in $\CC(Y)$. 
Let $\ell(g_Y)$ denote the length of $g_Y$. Then the model manifold $E_m$ contains a \emph{sub-model} 
$E_m(Y)$   for $Y \times [0, \ell(g_Y)]$ with initial and terminal vertices   $u_Y, v_Y$. The construction of the sub-model
$E_m(Y)$ can be culled out of \cite[pgs. 37-40]{minsky-elc1},   to
which we refer for details on resolutions of hierarchies and slices.
Indeed,
the collection of tight geodesics $g_W$ supported on subsurfaces of $Y$ is used to construct the  model
$E_m(Y)$. 
Further, $E_m(Y)$ embeds locally isometrically in 
the full model manifold $E_m$. Hence, $\Phi: E_m \to E$ restricts to an embedding of
$E_m(Y)$ so that $\Phi (E_m(Y)) (\subset E)$ is bi-Lipschitz homeomorphic to   
$E_m(Y)$ with bi-Lipschitz constant depending only on $\Phi$ (but not on $Y$).

Recall Definitions~\ref{def-bddgeoblocksub} and \ref{def-ibddgeoblocksub}. Then the following refines of one direction each of Theorems~\ref{thm-bddgeoss} and \ref{thm-ibddgeoss}. Again, the proof of \cite[Theorem 9.1]{minsky-elc1} contains its proof.
Assume that the bi-Lipschitz constant $L$ occurring 
in Definition~\ref{def-bddgeoblock} is fixed below. Our quantification will be in terms of a new bi-Lipschitz constant $K$.

\begin{thm}\label{thm-effectivesub}  Let $E$ denote any truncated degenerate
	end of a hyperbolic 3-manifold $M^h$. 
	Let  $E_m$ denote the \emph{combinatorial} model manifold for $E$ constructed in \cite{minsky-elc1}
	and $\Phi: E_m \to E$ denote the bi-Lipschitz homeomorphism furnished by \cite{minsky-elc2}. 
	Let $\HH(\tau,\LL)$ denote the hierarchy of tight geodesics  associated with $E_m$. 
	Given $R \geq 1$, there exists $K \geq 1$ such that the following holds.
	Let $Y$ be any essential subsurface of $S$ with $\zeta(Y) \geq 4$ and $g_Y \in \HH(\tau,\LL)$ be a tight geodesic supported on $Y$, and let 
	$E_m(Y)$, homeomorphic to $Y \times [0, \ell(g_Y)]$ be a 
    {model manifold constructed from $g_Y$ and all tight geodesics subordinate to $g_Y$ in the sense of \cite{minsky-elc1}. (Note that $E_m(Y) \subset E_m$.)}
	\begin{enumerate}
		\item Suppose that for every proper essential subsurface $W$ of $Y$ (including annular domains), $d_W(\tau,\LL) \leq R$.
		Then $\Phi(E_m(Y))$ is $K-$bi-Lipschitz to a bounded geometry sub-model of length 
		$\ell(g_Y)$ in the sense of Definition~\ref{def-bddgeoblocksub}.
		
		\item Suppose that for every proper \emph{non-annular} essential subsurface $W$ of $Y$, $d_W(\tau,\LL) \leq R$.
		Then $\Phi(E_m(Y))$ is $K-$bi-Lipschitz to an i-bounded geometry sub-model of length 
		$\ell(g_Y)$ in the sense of Definition~\ref{def-ibddgeoblocksub}.
	\end{enumerate}
\end{thm}

\subsection{Proof of Theorem~\ref{thm-maintechnicalgeolt}}\label{sec-geolts} With the background on model geometries of Section~\ref{sec-splitt} in place, the proof of Theorem~\ref{thm-maintechnicalgeolt} now follows the scheme sketched at the end of Section~\ref{sec-ibdd}.

\begin{proof}[Proof of Theorem~\ref{thm-maintechnicalgeolt}:] We continue with the notation used in Theorem~\ref{thm-effectivesub}. 
	We observe first that $d_S (\tau, \LL) = \infty$, and that for any proper essential subsurface $W$ of $S$, $d_W (\tau, \LL)$ is finite.
	Hence 
	\begin{enumerate}
		\item either there exist tight geodesics $g_W\in \HH(\tau,\LL)$ supported on subsurfaces $W$ of complexity $\zeta_0 (W) = 4$, such that $\{\ell(g_W)\}$ is unbounded,
		\item or 	there exists a minimal $\zeta_0 >  4$ and $R \geq 1$, such that  the following happens:
		\begin{itemize}
			\item   there exist tight geodesics $g_W\in \HH(\tau,\LL)$ supported on subsurfaces $W$ of complexity $\zeta (W) = \zeta_0$, such that $\{\ell(g_W)\}$ is unbounded,
			\item  for all $Y$   satisfying $4\leq \zeta (Y) < \zeta_0$, any tight geodesic $g_Y\in \HH(\tau,\LL)$ supported on $Y$ satisfies $\ell(g_Y) \leq R$.
		\end{itemize}
	\end{enumerate}
    In either case, we show now that Theorem~\ref{thm-effectivesub}(2) furnishes $K \geq 1$ and a sequence of subsurfaces $W_n$ with $ \zeta (W_n) = \zeta_0$ and $E_m(W_n)$ embedded in $E_m$ such that   $\Phi(E_m(W_n))$ is $K-$bi-Lipschitz to an i-bounded geometry sub-model of length 
	$\ell(g_{W_n})$ in the sense of Definition~\ref{def-ibddgeoblocksub}.
    
    The first case was dealt with at the end of Section~\ref{sec-ibdd}. In the second case, we follow the same scheme. 
    Choose $x_n$ to lie in the thick part of the  $[\frac{\ell(g_{W_n})}{2}]$-th block. The number of topological types of surfaces with a fixed complexity $\zeta_0$ is finite.
    Hence, after passing to a subsequence if necessary, we can assume that the subsurfaces $W_n$ with $ \zeta (W_n) = \zeta_0$  are homeomorphic to a fixed surface $\Sigma$ with
    $\zeta (\Sigma) = \zeta_0$. 
    
	As in the argument for $\zeta_0=4$,
the combinatorial model manifold for $E$ contains
combinatorial sub-models $E_m(W_n)$ consisting of $\ell(g_{W_n})$ drilled thin blocks
(in the sense of Definition~\ref{def-dthinblock}) glued end to end. Here, each $E_m(W_n)$ is a concatenation of drilled thin blocks
 obtained from $\Sigma \times [0, \ell(g_{W_n})]$ after drilling.

We choose a sequence $x_n$ such that $x_n$ lies in the (necessarily thick part of the) $[\ell(g_{W_n}) / 2]-$th block 
of $E_m(W_n)$. Finally, let $n \to \infty$. Let 
$(N_\infty, x_\infty)$ denote the geometric limit of 
$E_m(W_n)$ (after passing to a subsequence if necessary).

Recall that $M$ (resp. $E$) is the truncation of $M^h$ (resp.\ $E^h$).
Hence, the geometric limit of $(M, x_n)$ equals the geometric limit of $(E,x_n)$. 
Finally, the geometric limit of $(E,x_n)$ agrees with the geometric limit 
$(N_\infty, x_\infty)$  away from Margulis tubes.  Proposition~\ref{prop-geoltofibddgibdd} now  shows that the geometric limit $N^h$ of  $\big(\Phi(E_m(W_n)), x_n \big)$ has a truncation $N$ that admits a model of generalized i-bounded geometry, completing the proof.
	\end{proof}

\begin{rem}\label{rmk-ct} Recall that Corollary~\ref{cor--kl-nonconical} tells us that the Hausdorff dimension of the non-conical limit set $\Lambda^{nc}(G)$ of a finitely generated geometrically infinite Kleinian group $G$ is 2 and the above proof of Theorem~\ref{thm-maintechnicalgeolt} 
completed the proof of Corollary~\ref{cor--kl-nonconical}.

There is, however, a much more elementary statement that can be deduced much more easily from results in the existing literature: Let $G$ be a finitely generated 
Kleinian group. Then  $\Lambda^{nc}(G)$  is countable if and only if $G$ is geometrically finite. Indeed, if $G$ is geometrically finite, non-conical limit points agree with parabolic fixed points in $\Lambda(G)$, and this collection is countable.
On the other hand,
when $G$ is   geometrically infinite, there is a  Cannon-Thurston map $\partial i$ from the Gromov boundary (when $G$ has no parabolics) or the Bowditch boundary (when $G$ has  parabolics) onto $\Lambda(G)$ \cite{mahan-split,mahan-kl}. Further, 
\begin{enumerate}
\item  points in   $\Lambda(G)$ with multiple pre-images under $\partial i$
are non-conical,
\item $\partial i$ identifies precisely the ideal end-points of leaves of ending laminations \cite{mahan-elct,mahan-kl}, and
\item Any ending lamination has uncountably many leaves.
\end{enumerate} 
Hence $\Lambda^{nc}(G)$  is uncountable when $G$ is   geometrically infinite.
\end{rem}

\begin{eg}\label{eg-2dsurf}
We finally construct an example of a geometrically infinite hyperbolic surface to show that the sufficient conditions of Theorem~\ref{NonConicalFromGeometricLimit}
are \emph{not necessary}. We will construct a geometrically infinite hyperbolic surface surface $S$  such that as $x\in S$ tends to infinity, the injectivity radius
$\inj_x$ at $x$ also  tends to infinity. Hence, any geometric limit $(S,x_n)$ with 
$x_n\to \infty$ is necessarily the full hyperbolic plane $\Hyp^2$. This violates the hypotheses of Theorem~\ref{NonConicalFromGeometricLimit}. Nevertheless, the Hausdorff dimension of the non-conical limit set of the subgroup of $PSL(2, \R)$ corresponding
to $S$ is one. This will follow from Theorem~\ref{thm-cheeger-surface} once the construction is done.

We proceed with our construction. For each $n \in \natls$, construct a closed hyperbolic surface $S_n$ with injectivity radius at least $n$. Such $S_n$'s may be constructed as covers of a fixed closed hyperbolic surface $\Sigma$ by using the residual finiteness of surface groups, and constructing covers corresponding to subgroups that exclude small elements. Next, let $\sigma_n \subset S_n$ denote
a simple closed non-separating geodesic, with length $l_n$. Clearly $l_n \to \infty$ as
$n\to \infty$. Passing to a subsequence if necessary, we can assume that 
$l_n < l_{n+1}$ for all $n$. Let $S_n^c$ denote $S_n$ cut open along $\sigma_n$.
Then  $S_n^c$ has two boundary components $\sigma_n^\pm$, both of length $l_n$.
We will modify $S_n^c$ minimally to change the length of  $\sigma_n^+$ to $l_{n+1}$.
Towards this, let $\PP_n \subset S_n^c$ be an embedded pair of pants with one boundary component $\sigma_n^+$ and the other boundary components $\alpha_n, \beta_n$ say.
Let  $\PP_n^\prime$ denote the unique pair of pants with boundary components of length
$l_{n+1}$,  $\ell(\alpha_n), \ell(\beta_n)$. Let $S_n^\prime = (S_n \setminus \PP_n) \cup
\PP_n^\prime$ denote the hyperbolic surface with totally geodesic boundary obtained by replacing $\PP_n$ with  $\PP_n^\prime$. Note that the lengths $\ell(\alpha_n), \ell(\beta_n)$ are both at least $2n$ by the assumption on injectivity radius,
as is $l_{n+1}$. For $n$ large, the pair of pants $\PP_n\prime$ is `skinny', i.e.\ the distance between any pair of its boundary geodesics tends to zero as $n \to \infty$.
Hence, there exists $\ep_n$ with $\ep_n\to 0$ as $n \to \infty$ such that for each point $p$ on the boundary of $S_n^\prime$, there exists a hyperbolic half-disk $\HH_p$ with boundary of radius at least $(n-\ep_n)$ such that  the boundary of $\HH_p$ is
 contained in the totally geodesic  boundary of $S_n^\prime$. Denote the boundary component of $S_n^\prime$ of length $l_{n+1}$ by $\sigma_n^m$, so that 
 $S_n^\prime$ has two boundary components $\sigma_n^-$ and $\sigma_n^m$ with lengths $l_n$ and $l_{n+1}$ respectively.
 For all $n \geq 1$, glue $S_n^\prime$ and $S_{n+1}^\prime$ together along the boundary components of length $l_{n+1}$. This gives us a surface $S^\prime$ with
 one geodesic boundary component of length $l_1$ corresponding to $\sigma_1^-$.
 Finally, double $S^\prime$  along  $\sigma_1^-$ to obtain $S$.

 A caveat: it is possible, a priori, that the length of a geodesic in  $S^\prime$ intersecting each $S_n^\prime$ and escaping to infinity still has finite length. However, since each $S_n$ has injectivity radius at least $n$, this is not possible in the above example. In particular,  $S$ has a complete hyperbolic structure.
 Further, as promised, the construction shows that the injectivity radius
 $\inj_x$ at $x$   tends to infinity  as $x\in S$ tends to infinity.
 
 Finally, we exploit the freedom in the construction of $S_n$ to ensure that 
 the lengths $l_n$ grow slowly with respect to the areas of $S_n$, i.e.\ we demand that $(l_n + l_{n+1})/A(S_n) \to 0$ as $n \to \infty$. This can be arranged for instance by increasing the area of each $S_n$ arbitrarily by increasing its topology as follows.
Let $\beta_n$ be an auxiliary non-separating curve in $S_n$ disjoint from $\sigma_n$. Then cutting $S_n$ open along  $\beta_n$ and gluing finitely many of these cyclically end to end, we can construct a finite cyclic cover
 (of as large a degree as we like) of $S_n$. Since $\sigma_n$ is unaffected by this cyclic cover construction, we can arrange so that $(l_n + l_{n+1})/A(S_n) \to 0$ as $n \to \infty$. It follows that the Cheeger constant of $S$ is zero.
  Hence, by  Theorem~\ref{thm-cheeger-surface} the Hausdorff dimension of the non-conical limit set of the subgroup of $PSL(2, \R)$ corresponding
 to $S$ is one.  
 \end{eg}

\section{Hausdorff dimension of Myrberg limit sets}\label{sec-Myrberg}

Let $G$ be any non-elementary discrete group acting
properly by isometries on a Gromov hyperbolic space $X$. Let $\mG$ be the set of Myrberg limit points as in Definition \ref{MyrbergDefn}. The goal of this section is to prove that the Hausdorff dimension of the Myrberg limit set is the same as that of the whole conical limit set. In the next section, we will explain how to prove the same result for the Myrberg limit 
set in the Floyd boundary.

\begin{thm}\label{MyrbergHdimHypThm}
The Hausdorff dimension of the Myrberg limit set $\mG$  is equal to $\frac{\e G}{\epsilon}$, where $\epsilon$ is the parameter for the visual metric in Lemma \ref{VisualMetric}.    
\end{thm}

The remainder of this section is devoted to the proof of this theorem. The scheme is analogous to the one followed in the construction of non-conical limit points in Section  \ref{sec-nonconical}. We construct a sequence $A_n$ of  \textit{large annular sets} of elements and a sequence of \textit{bridges} $ b_n$ inserted between $A_n$ and $A_{n+1}$.
We then proceed to concatenate these appropriately.
We now make this precise.\\

\noindent\textbf{Bridges.} As $G$ is a countable group, we list all loxodromic  elements in $G\act X$ as follows.
$$
\mathcal B=\{ b_1,   b_2, \cdots,  b_n, \cdots \}
$$
We include all non-trivial powers of loxodromic elements in $\mathcal B$. 
Fix a basepoint $o\in X$. Let $B_n=d(o, b_no)$ be the  length of elements $b_n$ in $\mathcal B$.
We fix  a  set    $F=\{f_1,f_2,f_3\}$   of three pairwise independent loxodromic elements in $G$. The following lemma will be useful.

\begin{lem}\label{ChooseAi}
Let $b\in\mathcal B$ and $\tilde A$ be a subset in $G$.
There exist a subset $A\subseteq \tilde A$ and a loxodromic element $f\in F$  with the following properties:
\begin{enumerate}
    \item $3|A|\ge |\tilde A|$.
    \item For   any $a\in A$, we have $[o,ao]$ has $\tau$-bounded projection to $\ax(f)$.
    \item  $[o,bo]$ has $\tau$-bounded projection to $\ax(f)$.
 
\end{enumerate}  
where the constant $\tau>0$ depends only on the axis of $f\in F$ and $\mathcal B$.
\end{lem}
\begin{proof}
	By applying Lemma \ref{BddProjLem}  twice, we have the following.
For each $a\in A$ there exists $f\in F$ so that $[o,ao]$ and $[o,bo]$ have $\tau$-bounded projection to $\ax(f)$. As $F$ consists of three elements,  (1) follows by picking a common $f$ for a subset   $A$ of  $\tilde A$
 of cardinality at least one-third.  
\end{proof}

\noindent\textbf{Large annular sets.} 
Fix $\Delta\ge 1$. Recall the annular set with parameter  $n, \Delta$ (Definition \ref{def-annularset}): 
$$
A(n, \Delta,o)=\{g\in G: |d(o, go)-n|\le\Delta\}
$$
for which we have
\begin{equation}\label{AnnulusGrowth}
\e G=\limsup_{n\to\infty}\frac{\log |A(n, \Delta,o)|}{n}
\end{equation}


We fix $\tau>0$ as in Lemma \ref{ChooseAi} and let  $L, R>0$ be given by Lemma \ref{QuasiRadialTree}.  We assume that $d(o,fo)>L$ for each $f\in F$ by taking high powers if necessary. Note that $\tau$ remains the same as it depends only on the axes $\ax(f)$.

Fix a divergent sequence of numbers  $L_n\to \infty$ with $L_n\ge L$ so that 
\begin{enumerate}
    \item 
    $\frac{L_{m+1}+\Delta_{m+1}}{\sum_{n=1}^m L_n+\Delta+B_n}\to 0$.  
    \item 
    $|A(L_n, \Delta,o)|\ge \mathrm{e}^{L_n\omega_n}$.
    \item 
    $\omega_n\to \e G$
\end{enumerate} 
where  Item (2) follows by (\ref{AnnulusGrowth}). Thus, the parameters $(L_n,\Delta,K_n)$ with $K_n=1$ satisfy the assumptions  (\ref{ChoiceKnEq}) (\ref{ChoiceKnEq2}) of Lemmas \ref{LargeTreeGrpVersion} and \ref{HDLargeTree}.

By a covering argument, we  see that $A(L_n, \Delta,o)$ contains a maximal $(2R+2\Delta)$-separated subset $\tilde A_n$ so that  $$|\tilde A_n|\ge N_0 \mathrm{e}^{L_n\omega_n}$$ where $N_0:=|\{g\in G: d(o,go)\le 2R+2\Delta\}|$ depends only on $R,\Delta$. 

By Lemma \ref{ChooseAi}, there exist a sequence of subsets $A_n\subseteq \tilde A_n$ and $f_n, h_n\in F$ so that
\begin{enumerate}
    \item 
    $|A_n|>\mathrm{e}^{L_n \omega_n}$.
    \item
    for each $a\in A_n$, $[o,ao]$ and $[o, b_no]$ have $\tau$-bounded projection to $\ax(f_n)$.
    \item 
    for each $a\in A_{n+1}$, $[o,ao]$ and $[o, b_no]$ have $\tau$-bounded projection to $\ax(h_n)$.
\end{enumerate}
\begin{rem}
By Lemma \ref{ChooseAi}, we should have $|A_n|>(N_0/3)\mathrm{e}^{L_n \omega_n}$ in Item (1). We may take   even larger  values of $L_n$ to absorb the coefficient $N_0/3$ before $\mathrm{e}^{L_n \omega_n}$.    
\end{rem} 
We may drop (and re-index) finitely many elements $b_n\in \mathcal B$  so that $B_n=d(o,b_no)\ge L$ for any $n\ge 1$. The last two criteria above imply the following analog of Lemma \ref{lem:localqginhyp}. 
\begin{lem}\label{lem:localqg-Myrberg}
There exist $c, \tau'$ depending only $\tau$ satisfying the following. For any $a_n\in A_n,a_{n+1}\in A_{n+1}$,  the path labeled by $a_nf_n  b_n  h_{n} a_{n+1}$, i.e.\  $$[o,a_no]\cup (a_n[o, f_no])\cup(a_nf_n[o, b_no])\cup(a_nf_n b_n[o,h_{n}o])\cup(a_nf_n b_nh_n[o,a_{n+1}o])$$ is an $L$-local $\tau'$-quasi-geodesic. Hence it is a $c$-quasi-geodesic by Lemma \ref{localtoglobal}. 
\end{lem}

\noindent\textbf{Construction of Myrberg limit points.} 
Set
$$
\mathcal W:=\bigcup_{m\ge 0} \prod_{n=1}^m A_nf_n    b_n h_{n}. 
$$
Then $\mathcal W$ consists of \textit{admissible}  words 
alternating over $A_n$ ($n\ge 1$) and $\mathcal B$ as follows:
$$
W_m= (a_1f_1   b_1h_1)  \cdots (a_nf_n  b_n  h_{n})\cdots (a_mf_m  b_m h_m)
$$
By construction  $\mathcal W$ has a natural tree structure. Let $\mathcal W_\infty$ denote the set of infinite  words whose prefixes are all admissible. 
Let $c$ be given by Lemma \ref{lem:localqg-Myrberg}. The following  key  fact
will be useful.
\begin{lem}\label{EndingMyrbergPts}
Let $W_\infty=\prod_{n=1}^\infty A_nf_n   b_nh_{n}$ be an infinite admissible word. Then 
the sequence of points $W_m o $ with $m\ge 1$ forms a $c$-quasi-geodesic ray ending at a Myrberg point denoted by $\xi_W$. 
\end{lem}
\begin{proof}
By Lemma \ref{lem:localqg-Myrberg}, the  path $\gamma$ labeled by an infinite word $W_\infty$ is an $L$-local $\tau'$-quasi-geodesic, so it is a $c$-quasi-geodesic ray in $X$. Let $\xi\in \Lambda G$ be the end point of $\gamma$. Of course, $\xi$ is necessarily a limit point. 

To see that $\xi$  is a Myrberg limit point, we make use of Lemma \ref{lem:CharMyrberg}. For any given $ b\in\mathcal B$, we need to find a sequence of translates $g_n b$ so that for some $R>0$, $N_R(g_n b)$ intersects $\gamma$ in an unbounded set as $n\to\infty$. This is guaranteed by the nature of the construction. Indeed,   all powers  $\{ b^i: i\ge 1\}$ of $ b$ are contained in $\mathcal B$. So they appear in the infinite word $W_\infty$. Let $g_n$ denote the element represented by the prefix subword just before the occurrence of $ b^n$ in $W_\infty$.  Thus, there exists 
$R$ depending on $c$, such that for all $n$,
   $N_R(g_n\ax( b_n))$    intersects $\gamma$ in a set of diameter comparable to $d(o, b^no)$. The endpoint of $\gamma$ is thus a Myrberg limit point by Lemma \ref{lem:CharMyrberg}.
\end{proof}

Let us define the map $\Phi: \mathcal W \longrightarrow X$ as follows $$\begin{aligned}
W=\prod_{i=1}^m a_nf_n    b_n h_{n} &\longmapsto Wo
\end{aligned}$$

\begin{lem}\label{lem:qrtreeforMyrbergHyp}
The map $\Phi$  is injective and the limit set of the image of $\Phi$ has Hausdorff dimension $\frac{\e G}{\epsilon}$.
\end{lem}
\begin{proof}
The injectivity of  $\Phi$ in the proof of Lemma \ref{QuasiRadialTree} relies on the following two facts : 
\begin{enumerate}
    \item 
    Lemma \ref{lem:localqginhyp} shows that for every $W\in\mathcal W$ and every prefix $W_m$, the path labeled by $W$ is a $c$-quasi-geodesic and intersects the $R_0$-neighborhood of  $W_mo$. This is proved here in Lemma \ref{EndingMyrbergPts}.
\item 
$A_n$ consists of $2(\Delta+R)$-separated elements.
\end{enumerate}
The same argument then proves the injectivity, and thus the image $\Phi(\mathcal W)$ is a quasi-radial tree.
By the above choice of $L_n$ and using the fact that $\omega_n\to\omega$, the statement about Hausdorff dimension follows by Lemma \ref{HDLargeTree}.
\end{proof}

 Theorem \ref{MyrbergHdimHypThm} now follows. \hfill $\Box$

\section{Further generalizations: groups with contracting elements}\label{sec-contracting}

In this section, we explain how the main construction in Section~\ref{sec-hausd} generalizes to groups with contracting elements. In particular, this allows us to compute the Hausdorff dimension of the Myrberg limit set in the Floyd boundary (Theorem \ref{MyrbergHdimFloyd} below).

\subsection{Preliminaries on contracting elements}
Let $Z$ be a closed subset of $ X$ and let $x$ be a point in $ X$. We define the set of nearest-point projections from $x$ to $Z$ as follows
\[ \pi_{Z}(x) : = \big \{ y\in Z: d(x, y) = d(x, Z) \big \} \]
where  $d(x, Z) : = \inf \big \{ d(x, y): y \in Z \big \}.$
Since $X$ is a proper metric space, 
$\pi_{Z}(x)$ is non empty.  Denote $\proj_Z(x,y)=\mathrm{diam}(\pi_{Z}(x)\cup \pi_{Z}(y))$.

\begin{defn} \label{Def:Contracting}
We say that a closed subset $Z \subset X$ is \emph{$C$--contracting} for a constant $C>0$ if,
for all pairs of points $x, y \in X$, we have
\[
d(x, y) \leq d(x, Z) \quad  \Longrightarrow  \quad \proj_Z(x,y) \leq  C.
\]
Any such $C$ is called a \emph{contracting constant} for $Z$.
\end{defn}

The property (1) actually characterizes the contracting property.  
\begin{lem}\cite[Corollary 3.4, Lemma 3.8]{BF1}\label{lem-contracting-property}
Let $Z$ be a closed $C$-contracting subset. Then
the following hold.
\begin{enumerate}
    \item  There exists $C'$ such that
    any geodesic outside $N_{C'}(Z)$ has $C'$-bounded projection to $Z$. 
    \item Given $D > 0$, there exists $C'$ such that
    if $W$ is a closed subset with Hausdorff distance at most $D$ from $Z$, then $W$ is $C'$-contracting.
\end{enumerate}
\end{lem}

An isometry $h$ of infinite order is called \textit{contracting} if for some  $o\in X$, the orbital map $n\in\mathbb Z\mapsto h^no\in X$ is a quasi-isometric embedding and the image $\{h^no: n\in \mathbb Z\}$ is a contracting subset in $X$. The definition does not depend on $o$ by Lemma \ref{lem-contracting-property}.

A group $G$ is called \textit{elementary} if it is virtually a cyclic group.
Let us  consider a proper and isometric action of a group $G$ on $X$. 

\begin{lem}\cite[Lemma 2.11]{YANG10}\label{elementarygroup}
A contracting element $h\in G$ is contained in a unique  maximal elementary subgroup denoted by $E(h)$. Moreover,
$$
E(h)=\{g\in G: \exists n\in \mathbb N_{> 0}, (\;gh^ng^{-1}=h^n)\; \lor\;  (gh^ng^{-1}=h^{-n})\}.
$$
\end{lem}

In contrast to the axis $\ax(h)$ in hyperbolic space (Definition \ref{def-qa}), we take the following definition of axis depending on the basepoint $o\in X$. Define the \textit{axis} of $h$ to be the following quasi-geodesic 
\begin{equation}\label{axisdefn}
\ax(h)=\{f o: f\in E(h)\}.
\end{equation} Notice that $\ax(h)=\ax(k)$ and $E(h)=E(k)$    for any contracting element   $k\in E(h)$.
 

Two  contracting elements $h_1, h_2\in G$  are called \textit{independent} if the collection $\{g\ax(h_i): g\in G;\ i=1, 2\}$ is a contracting system with bounded intersection. Note that two conjugate contracting elements with disjoint fixed points are not independent in our sense.
 
\begin{lem}\cite[Lemma  2.12]{YANG10}
Assume that $G$ is a non-elementary group with a contracting element. Then $G$ contains infinitely many pairwise independent contracting elements.     
\end{lem}


\subsubsection{Convergence boundary}\label{sub-conv-bdry}
Consider a metrizable compactification $\bU:=\pU\cup \U$, so that $\U$ is open and dense in $\bU$. We also assume that the action of $\isom(\U)$ extends by homeomorphism to  $\pU$. We follow the exposition in \cite{YANG22} closely and refer to it for additional details.

We   equip $\pU$    with an  $\isom(\U)$--invariant  partition $[\cdot]$:   $[\xi]=[\eta]$ implies $[g\xi]=[g\eta]$ for any $g\in \isom(\U)$.  We say that $\xi$ is \textit{minimal} if $[\xi]=\{\xi\}$, and a subset $U$ is \textit{saturated} if $U=[U]$. In general, $[\cdot]$ may not be closed, \emph{e.g.}, the horofunction boundary with finite difference relation.  

We say that $x_n$ \textit{tends} to (resp. \textit{accumulates} on)  $[\xi]$ if the limit point (resp. any accumulation point) is contained in the subset $[\xi]$. This implies that $[x_n]$ tends to or accumulates on $[\xi]$ in the quotient space $[\pG]$. So, an infinite ray $\gamma$ \textit{terminates} at $[\xi] \in \pU$ if any sequence of points in $\gamma$ accumulates on $[\xi]$.  We say that $\xi$ is \textit{non-pinched} if whenever $x_n, y_n\in \U$ are two sequences of points converging to $[\xi]$, the 
sequence of geodesic segments $[x_n,y_n]$ is an escaping.

\begin{defn}\label{ConvBdryDefn}
    
We say that $(\bU,[\cdot])$ is a \textit{convergence compactification} of $\U$ if the following  hold.
\begin{enumerate}
    \item[\textbf{(A)}]Any contracting geodesic ray $\gamma$ accumulates on a closed subset $[\xi]$ for some  $\xi\in \pU$; and any sequence $y_n\in \U$ with escaping projections $\pi_\gamma(y_n)$ tends to $[\xi]$. 

    \item[\textbf{(B)}]
    Let $\{Z_n\subseteq \U :n\ge 1\}$ be an escaping sequence of $C$--contracting  quasi-geodesics for some $C>0$. Then for any given $x\in \U$, there exists a  subsequence of $ Y_n$ defined as follows  
    $$Y_n:=Z_n\cup \{y\in \U:  \|[x,y]\cap N_C(Z_n)\|\ge 10C\}$$ 
    and $\xi\in \pU$ such that $Y_n$  accumulates to $[\xi]$, i.e.\  any convergent sequence of points $y_n\in Y_n$ tends to a point in $[\xi]$.
    \item[\textbf{(C)}]
    The set $\mathcal C$ of {non-pinched} points  $\xi\in \pU$ is non-empty. 
\end{enumerate}  
\end{defn} 
Assumption (C) excludes trivial examples given by the one-point compactification. Note that any Hausdorff  quotient of a convergence boundary is again a convergence boundary.
The convergence boundary in Definition~\ref{ConvBdryDefn} allows us to treat the following examples in a unified language.  
\begin{examples}\label{ConvbdryExamples}
The first three convergence boundaries below are equipped with  a \textit{maximal} partition $[\cdot]$ (that is, $[\cdot]$--classes are singletons).
    \begin{enumerate}
        \item 
        Hyperbolic space $\U$ with Gromov boundary $\pU$, where  all boundary points are non-pinched.
        \item 
        CAT(0) space $\U$ with visual boundary $\pU$ (homeomorphic to the horofunction boundary), where all boundary points are non-pinched.
        \item
        The Cayley graph $\U$ of a relatively hyperbolic group equipped with the Bowditch or Floyd boundary $\pU$, where  conical limit points are non-pinched. See \textsection \ref{sub-app-Floyd} for more details.

        If $\U$ is infinite ended, we could also take $\pU$ as the 
        space of ends. The same conclusions hold.
        \item
        Teichm\"{u}ller space $\U$ with the Thurston  boundary $\pU$, where $[\cdot]$ is given by the Kaimanovich-Masur partition \cite{KaMasur}. Uniquely ergodic   points are non-pinched, and their $[\cdot]$-classes are singleton.
        \item
        Any proper metric space $\U$ with the horofunction boundary $\pU$, where $[\cdot]$ is given by finite difference partitions and all boundary points are non-pinched (\cite[Theorem 1.1]{YANG22}).
        If $\U$ is a  CAT(0) cubical space, a result of Bader-Guralnik says that the horofunction boundary is exactly the Roller boundary (\cite[Prop. 6.20]{FLM}).  
        If $\U$ is the Teichm\"{u}ller space with Teichm\"{u}ller metric, the horofunction boundary is  the Gardiner-Masur  boundary (\cite{LS14,Walsh19}). 
    \end{enumerate}
\end{examples}
 
\subsubsection{Limit set and Myrberg limit points}\label{sub-Myrberg}
The \textit{limit set} $\pG$ of $G$ is defined to be the union of $[\cdot]$-classes of accumulation points of some (any) orbit $Go$ in $\pU$. The limit set is independent of the basepoint $o\in X$ by Assumption \textbf{(B)} in Definition \ref{ConvBdryDefn}.
Let $h$ be a contracting element.  By Assumption \textbf{(A)}, the two half-rays of the axis $\ax(h)$ accumulate on two $[\cdot]$-classes of boundary points denoted by $[h^-]$ and $[h^+]$. By definition, the union $[h^\pm]$ belongs to $\pG$.
We say that $h$ is \textit{non-pinched} if $[h^-]\ne [h^+]$. Equivalently,  $[h^\pm]$ are non-pinched points by \cite[Lemma 3.19]{YANG22}. We are only interested in  convergence boundaries with non-pinched contracting elements. This is the case for all examples as above.

Let $h$ be a non-pinched contracting element. The assumptions \textbf{(A)} and \textbf{(C)} allow us to extend the nearest
point projection $\pi_{\ax(h)}$ to the boundary. 
\begin{lem}\cite[Lemma 3.24]{YANG22}
The  projection   $\pi_{\ax(h)}: X\to \ax(h)$ extends to boundary points in $\pU\setminus [h^\pm]$ in the following sense.
There exists a constant $D$ depending on $\ax(h)$ so that if $x_n\in X\to \xi\in \pU\setminus [h^\pm]$, then  $\pi_{\ax(h)}(\xi)$ is contained in a $D$-neighborhood of  $\pi_{\ax(h)}(x_n)$  for all sufficiently large $n$. 
\end{lem}
From this we obtain the North-South dynamics \cite[Lemma 3.27, Corollary 3.28]{YANG22}.
\begin{lem}\label{SouthNorthLem}
The action of $\langle h\rangle$ on $\pU\setminus [h^\pm]$ has the North--South dynamics: for any two open sets   $[h^+] \subseteq U$ and $[h^-]\subseteq V$ in $\pU$, there exists an integer $n>0$ such that $h^n (\pU \setminus V)\subseteq U$ and $h^{-n} (\pU \setminus U)\subseteq V$.  In particular, if $\proj_{\ax(h)}(x_n,y_n)\to \infty$ for $x_n, y_n\in X\cup \pU\setminus [h^\pm]$, we have $(x_n,y_n)$ converges to $([h^-],[h^+])$.
\end{lem}

We now formulate the analog  of Myrberg limit points in a general convergence boundary. Let $\pG\Join\pG$ denote the distinct $[\cdot]$-pairs in $\pG$. We equip $\pG$ with the quotient topology by identifying each $[\cdot]$ to a point. 
\begin{defn}\label{defn-Myrberg-general}
A non-pinched point $\xi \in  \pU$ is called a \textit{Myrberg limit point} if for any $x\in \U$, the set of $G$-translates of the ordered pair $(x, \xi)$ is dense in the space  $\pG\Join \pG$  in the following sense:
\begin{itemize}
    \item For any $[\zeta]\ne[\eta]\in [\pG]$ there exists $g_n\in G$ so that $g_nx\to [\zeta]$ and $g_n\xi\to [\eta]$ in the quotient topology.
\end{itemize}
\end{defn}    
\begin{rem}
By definition, the property of being a Myrberg point is a property of the $[\cdot]$-equivalence class.  When the $[\cdot]$ partition is maximal as in the first three Examples \ref{ConvbdryExamples} then the definition of Myrberg limit points coincides with Definition \ref{MyrbergDefn}.   
\end{rem}

In \cite[Lemma 3.15]{YANG22}, the fixed point pairs $([h^+], [h^-])$ of all non-pinched elements $h\in G$ are dense in the set $\pG\Join\pG$ of distinct pairs  of limit points.  Along similar lines in  Lemma \ref{lem:CharMyrberg} with  Lemma \ref{SouthNorthLem}, we could then prove the following.
\begin{lem}\cite[Lemma 4.16]{YANG22}\label{CharMyrberg-general}
A  point $\xi\in \pU$ is a Myrberg limit  point if and only if the following holds.
Let   $h\in G$ be a non-pinched contracting element. There is a sequence of elements $g_n\in G$ so that the projection of a geodesic ray $\gamma$ ending at $[\xi]$ to  $g_n\ax(h)$ tends to $\infty$.
\end{lem}

\subsection{Admissible paths and Extension Lemma}

Let $\mathbb F$ be a family of uniformly contracting sets. Assume that $\mathbb F$ has  \textit{bounded intersection} property. That is, for any $r>0$ there exists $D=D(r)$ so that $\mathrm{diam}(N_r(Z)\cap N_r(Z'))\le D$ for any $Z\ne Z'\in \mathbb F$. The notion of admissible paths allows   us  to construct   quasi-geodesics  by concatenating geodesics via $\mathbb F$.
\begin{defn}[Admissible Path]\label{AdmDef} Given $L,\tau\geq0$, a path $\gamma$ is called $(L,\tau)$-\textit{admissible} in $ X$, if $\gamma$ is a concatenation of geodesics $p_0q_1p_1\cdots q_np_n$ $(n\in\mathbb{N})$, where the two endpoints of each $p_i$ lie in some $Z_i\in \mathbb{F}$, and   the following   properties hold:
\begin{enumerate}
\item[(LL)] \textit{Long local property:}  Each $p_i$  for $1\le i< n$ has length bigger than $L$. We allow  the initial and final geodesic segments.  $p_0,p_n$ to be trivial, i.e.\ points.
\item[(BP)] \textit{Bounded Projection property:}  For each $Z_i$, we have $Z_i\ne Z_{i+1}$ and $$\max\{\mathrm{diam}(\pi_{Z_i}(q_i)),\mathrm{diam}(\pi_{Z_i}(q_{i+1}))\}\leq\tau,$$ where $q_0:=\gamma_-$ and $q_{n+1}:=\gamma_+$ by convention.
\end{enumerate} 
The collection $\{Z_i: 1\le i\le n\}$  is referred to as a contracting subset associated to the admissible path.
\end{defn}
\begin{rem}\label{ConcatenationAdmPath}
The path $q_i$ is allowed to be trivial, so that by the (BP) condition, it suffices to check $Z_i\ne Z_{i+1}$. It will be useful to note that admissible paths could be concatenated as follows. Let $p_0q_1p_1\cdots q_np_n$ and $p_0'q_1'p_1'\cdots q_n'p_n'$ be $(L,\tau)$-admissible. If $p_n=p_0'$ has length bigger than $L$, then the concatenation $(p_0q_1p_1\cdots q_np_n)\cdot (q_1'p_1'\cdots q_n'p_n')$ has a natural $(L,\tau)$-admissible structure.  
\end{rem}

\begin{prop}\label{admisProp}\cite[Proposition 3.1]{YANG6}
For any $\tau>0$, there exist $c, L,  R_0>0$ depending only on $\tau, C$ such that the following holds. Let $\gamma = p_0 q_1 p_1 \cdots q_n p_n$ be an $(L, \tau)-$admissible
path. Then  $\gamma$ is a $c$-quasi-geodesic and any geodesic joining $\gamma_-,\gamma_+$ intersects the $R_0$-neighborhood of the endpoints of every $q_i$. 
\end{prop}

Fix a set $\{h_1,h_2,h_3\}$ of three pairwise independent non-pinched contracting elements in $G$. The following is proved in \cite[Lemma 2.14]{YANG10} via similar ingredients (\ref{BddProjEQ2}) in proving Lemma \ref{DCExtension}. 

\begin{lem}[Extension Lemma]\label{extend3}
There exist   $L, \tau>0$ depending only on $C$ with the following property.  
Choose  elements $f_i\in \langle h_i\rangle$ for  $1\le i\le 3$  to obtain a set $F$ satisfying $d(o,f_io)\ge L$. Let $g_1, g_2\in G$ be any two elements.
There exists an element $f \in F$ such that $\mathrm{diam}(\pi_{\ax(f)}([o,g_io]))\le \tau$ for each $i=1,2$. In particular,  the path  $$\gamma:=[o, g_1o]\cdot(g_1[o, fo])\cdot(g_1f[o,g_2o])$$ is an $(L, \tau)$-admissible path relative to $\mathbb F$.
\end{lem}

\begin{rem}
Since admissible paths are given by local conditions, we can use $F$
to connect   any number of elements  $g\in G$ to get an $(L, \tau)$-admissible path. We refer the reader to \cite[Lemma 2.16]{YANG10} for a precise formulation. 
\end{rem}


The main result of this subsection reads as follows.

\begin{thm}\label{thm-qrtree-Myrberg-general}
Suppose that $G$ act properly on a proper geodesic metric space $X$ with   a convergence boundary $\pU$.  Assume that $G$ contains non-pinched  contracting elements. Then there exists a quasi-radial tree $T$ with vertices in the orbit $Go$ rooted at $o$ so that the growth rate of $T$ is equal to $\e G$ and the limit set of $T$ consists of Myrberg limit points.     
\end{thm}
\begin{proof}
The proof follows closely that of Theorem \ref{MyrbergHdimHypThm} presented  in Section \ref{sec-Myrberg}.
We list all non-pinched contracting  elements in $G\act X$ as follows.
$$
\mathcal B=\{ b_1,   b_2, \cdots,  b_n, \cdots \}
$$
which includes all non-trivial powers of contracting elements. Denote $B_n=d(o,b_no)$.

Fix $\Delta>1$ and let $L,\tau, R_0$ be given by Lemma \ref{extend3}. Choose a divergent sequence $L_n\to \infty$ with $L_n>L$. 
Let $\tilde A_n$ be a maximal  $(2R+2\Delta)$-separated subset of $A(L_n,\Delta,o)$.
With Lemma \ref{ChooseAi} replaced by Lemma \ref{extend3}, we can find as in Section  \ref{sec-Myrberg} a sequence of subsets $A_n\subseteq \tilde A_n$ and $f_n, h_n\in F$ so that
\begin{enumerate}
    \item 
    $|A_n|>\mathrm{e}^{L_n \omega_n}$.
    \item
    for each $a\in A_n$, $[o,ao]$ and $[o, b_no]$ have $\tau$-bounded projection to $\ax(f_n)$.
    \item 
    for each $a\in A_{n+1}$, $[o,ao]$ and $[o, b_no]$ have $\tau$-bounded projection to $\ax(h_n)$.
\end{enumerate}

Let $\mathcal W$ be the set of all words with form $W=\prod_{i=1}^m a_nf_n    b_n h_{n}$, where $a_n\in A_n$. We  define the map $\Phi: \mathcal W \longrightarrow X$ as follows $$\begin{aligned}
\Phi:\;\; &\mathcal W \longrightarrow X\\
  &W\longmapsto Wo
\end{aligned}$$ 
The injectivity of  $\Phi$ follows by a similar argument as  in  Lemma \ref{QuasiRadialTree}. We indicate the two main ingredients. 
\begin{enumerate}
    \item 
    Lemma \ref{lem:localqginhyp} shows that for every prefix $W_m$ of $W$, the geodesic  $[o,Wo]$   intersects the $R_0$-neighborhood of   $W_mo$. Here this follows from Lemma \ref{admisProp}, as $W$ labels a $(L,\tau)$-admissible path relative to $\mathbb F$.
\item 
$A_n$ consists of $2(\Delta+R)$-separated elements.
\end{enumerate}
Thus the image $T:=\Phi(\mathcal W)$ is a quasi-radial tree, and  the growth rate of $T$ is equal to $\e G$. 

Analogous to Lemma \ref{lem:localqg-Myrberg}, it remains to show that each branch in $T$ terminates at a Myrberg point.    
\begin{claim}
Let $W_\infty=\prod_{n=1}^\infty a_nf_n   b_nh_{n}$ be an infinite word. Then 
the sequence of points $W_m o $ for every prefix in $W$ with length $m\ge 1$ forms a $c$-quasi-geodesic ray $\gamma$ which accumulates on the $[\cdot]$-class of  a Myrberg point.     
\end{claim}
\begin{proof}[Proof of the Claim]
By construction, $\gamma$ is a $(L,\tau)$-admissible path relative to $\mathbb F$, so it is a $c$-quasi-geodesic by Proposition \ref{admisProp}. Moreover, if  we denote $g_m=\prod_{n=1}^m A_nf_n$, we have $d(g_mo, [o,\xi])\le R_0$ and $d(g_mb_mo, [o,\xi])\le R_0$. This implies $\pi_{g_m\ax(b_m)}(\gamma)>d(o,b_mo)-2R_0$ so the end point of $\gamma$ is a Myrberg point  by Lemma \ref{CharMyrberg-general}.    
\end{proof}

The proof is complete.
\end{proof}

Compared with Theorem \ref{MyrbergHdimHypThm}, we do not have here the estimate on the Hausdorff dimension, as there is no known visual metric on $\pU$ with properties as in Lemma \ref{VisualMetric} and Lemma \ref{ShadowApproxBalls}. However, in the special case of the Floyd metric, 
we can indeed apply Theorem \ref{thm-qrtree-Myrberg-general} to  compute the Hausdorff dimension of the Myrberg limit set in the Floyd boundary.

\subsection{Applications: Floyd boundary}\label{sub-app-Floyd} 
We first introduce the compactification of a locally finite graph due to W. Floyd \cite{Floyd}. The Cayley graph of a finitely generated group shall be our main focus. We follow closely the  exposition in      \cite{Ge2}, \cite{GePo2} and \cite{Ka}.

Let $G$ be a group with a finite generating set $S$. Assume that
$1\notin S$ and $S=S^{-1}$.  Let  $\Gx$ denote  the \textit{Cayley graph} of $G$ with respect to $S$, equipped with   the word  metric  $d$. We define a Floyd metric  on $\Gx$ by rescaling the word metric as follows.

Fix $0 < \lambda <1$ throughout the construction. The \textit{Floyd
length} $\ell_\lambda^o(e)$ of an edge $e$ in $\Gx $ is $\lambda^n$, where $n =
d(o, e)$. The Floyd length $\ell_\lambda^o(\gamma)$ of a path $\gamma$ is the sum of Floyd lengths of its edges. This induces a length metric $\rho_\lambda^o$ on $\Gx$, which is the infimum of Floyd lengths of all possible paths between two points.

Let $\Gf$ be the Cauchy completion of $G$ with respect to $\rho_\lambda^o$.
The complement $\pGf$ of $\Gx$ in $\Gf$ is called \textit{Floyd
boundary} of $G$. The boundary $\pGf$ is called \textit{non-trivial} if
$\sharp \pGf>2$. Non-triviality of the Floyd boundary does not depend on the choice of generating sets \cite[Lemma~7.1]{YANG6}. Most  groups have trivial Floyd boundary \cite{KN04, Lev20}. Currently, the  most general class of groups known  to have non-trivial Floyd boundary are relatively hyperbolic groups \cite{Ge2}.


By construction, we have the following equivariant property 
\begin{align}
\label{equiv}
\rho_\lambda^o (x, y) =
\rho_\lambda^{go}(gx, gy)\\
\label{lambdabilip}
\lambda^{d(o, o')} \le \frac{\rho_\lambda^{o}(x, y)}{\rho_\lambda^{o'}(x, y)} \le
\lambda^{-d(o, o')}    
\end{align}
for any two points $o, o' \in G$. So for different basepoints, the corresponding Floyd compactifications are bi-Lipschitz. Hence, the left-multiplication by each $g\in G$ on $G$ extends to the boundary as a bi-Lipschitz homeomorphism.  Note  that the topology  may  depend on the choice of the rescaling function and the generating set. 
When $G$ is hyperbolic, the Floyd metric $\rho_\lambda^o$ is, up to bi-Lipschitz equivalence,  the same as the visual metric $\rho_\epsilon^o$
(Section~\ref{sec-prel}) with $\epsilon:=-\log \lambda$ in  \cite[Appendix]{PYANG}. We shall write the Floyd metric ${\rho}_\lambda$ when the basepoint is identity.

The action on the Floyd boundary provides an important source of convergence group actions. If $|\pGf|\ge 3$, Karlsson proved in \cite{Ka}  that $\Gamma$ acts  by homeomorphism on $\pGf$ as a convergence group action. 
Moreover, the cardinality of $\pGf$ is either 0, 1, 2 or uncountably infinite. By   \cite[Proposition~7]{Ka}, $|\pGf|= 2$ exactly when the group $\Gamma$ is virtually infinite cyclic.
These follow from the following fundamental fact  in   \cite{Ka}.
\begin{lem}[Visibility lemma] \label{karlssonlem} 
For any $c\ge 1$, there is a function $\varphi: \mathbb R_{\ge 0} \to \mathbb R_{\ge 0}$ such that for
any $v \in G$ and any $c$-quasi-geodesic $\gamma$ in $\Gx$, 
$\ell_\lambda^v(\gamma) \ge \kappa$ implies that $d(v,
\gamma) \le \varphi(\kappa)$.
\end{lem}

By the theory of  convergence groups, elements in $G$ can be divided into the categories of elliptic, parabolic and hyperbolic elements. Hyperbolic elements in $G$ are infinite order elements with exactly two fixed points in $\pGf$. Moreover, they are contracting by \cite[Lemma 7.2]{YANG6}, so the previous discussion applies in the current setup.  

The Floyd boundary is \textit{visual}: any quasi-geodesic ray converges to a boundary point, and any two points $x, y\in G\cup\pGf$ can be connected by a bi-infinite or semi-infinite geodesic. See \cite[Prop. 2.4]{GePo2} for a proof.
For $x,y\in G$, we define  the \emph{shadow} of a ball  $B(y,r)$ from the source $x$ to be
$$\Pi_{x}(y, r) := \{\xi\in \pGf: \exists [x,\xi]\cap B(y,r)\ne\emptyset\}$$
We have the following analog of Lemma \ref{ShadowApproxBalls}, which compares   balls with shadows at large Floyd distance.  When $G$ is a relatively   hyperbolic group, Property (2) is proved  in \cite[Lemma~3.15]{PYANG}  for transitional points $v$ on $\gamma$. With the same proof, we generalize it to any group for  points $v$ with large Floyd distance. Property (1) is proved in \cite[Lemma~3.14]{PYANG}. We provide their short proofs for completeness. 
\begin{lem}\label{lemma3.16PY}
Given $\xi\in \pGf$, let $\gamma$ be a geodesic  between $1$ and $\xi$. Let $v$ be  any point on $\gamma$ and denote  $r=\lambda^{d(1,v)}$. Then
\begin{enumerate}
    \item For any $R>0$, there exist $C_1=C_1(R)>0$ so that $\Pi_1(v, R)\subset B_{{\rho}_\lambda}(\xi,C_1 r)$.
    \item 
    For any $\kappa>0$, there exist $R=R(\kappa), C_2=C_2(\kappa)>0$ so that if $\rho_\lambda^v(1,\xi)\ge \kappa$ then
$B_{{\rho}_\lambda}(\xi, C_2r)\subset  \Pi_1(v, R)$.
\end{enumerate}
\end{lem}
\begin{proof}
(1) Let $\eta\in \Pi_1(v, R)$ and  $w\in [1,\eta)$ so that $d(1,v)=d(1,w)$. As $d(v,[1,\eta])\le R$ we have $d(v,w)\le 2R$ and thus $\rho_\lambda(v,w)\le 2R\cdot \lambda^{d(1,v)-2R}$ by definition of Floyd metric. Note that $[v,\xi]$ is a $\rho_\lambda$-geodesic from $v$ to $\xi$ (\cite[Lemma 2.7]{PYANG}), so $\rho_{\lambda}(v, \xi)\le \frac{\lambda^{d(1,v)}}{1-\lambda}$. The same holds for $\rho_{\lambda}(w, \eta)$. Thus we  obtain
$$
\begin{array}{lll}
 \rho_{\lambda}(\xi, \eta) &\le \rho_{\lambda}(v, \xi)+\rho_{\lambda}(w, \eta) + \rho_{\lambda}(g, w) \\
\\
&\le (\frac{2}{1-\lambda}+\frac{2R}{\lambda^{2R}})\cdot \lambda^{d(1,v)}.
\end{array}
$$ Setting $C_1:=\frac{2}{1-\lambda}+\frac{2R}{\lambda^{2R}}$ completes the proof.

(2). Let   $\eta \in B_{\rho_{\lambda}}(\xi, \kappa r/2)$. Using
the property (\ref{lambdabilip}), we have $$\rho_{\lambda}^v(\eta, \xi)
\le \lambda^{-d(v,1)}\rho_{\lambda}(\eta, \xi) \le \kappa /2$$ Thus, $\rho_{\lambda}^v(1, \eta)\ge \rho_{\lambda}^v(1, \xi)-\rho_{\lambda}^v(\xi, \eta) \ge C_2:= \kappa/2$, and   $d(g, [1, \eta]) \le R:= \phi(\kappa/2)$ by  Lemma \ref{karlssonlem}. Hence,  $B_{\rho_{\lambda}}(\xi, C_2r) \subset \Pi_1(v, R)$, proving the lemma.    
\end{proof} 

The following easy consequence of Lemma \ref{karlssonlem} will be used.
\begin{lem}\label{LargeFloydDist}
Given $c>1, \kappa>0$ there exists $L=L(c,\kappa)>0$ with the following property.
Let $\gamma=\gamma_1\alpha\gamma_2$ be a $c$-quasi-geodesic. Assume that $\len(\alpha)>L$ and $\rho_\lambda^x(\alpha_-,\alpha_+)\ge \kappa$ for the midpoint $x$ of $\alpha$. Then $\rho_\lambda^x(\gamma_-,\gamma_+)\ge \kappa/2$. 
\end{lem}
\begin{proof}
Since $\gamma$ is a $c$-quasi-geodesic,  $d(x,\gamma_1), d(x,\gamma_2)$ is large compared with $L$. Choose $L$ large enough depending on $c$ and $\kappa$ so that $\fl^x(\gamma_1)\le \kappa/4$ and $\fl^x(\gamma_2)\le \kappa/4$ by Lemma \ref{karlssonlem}. The triangle inequality shows that   $\rho_\lambda^x(\gamma_-,\gamma_+)\ge \rho_\lambda^x(\alpha_-,\alpha_+)-\fl^x(\gamma_1)-\fl^x(\gamma_2) \ge \kappa/2$.  
\end{proof}

The action of $G$ on the Floyd boundary is a convergence group action, so we could define the Myrberg limit set as in Definition \ref{MyrbergDefn}. From an alternate point of view, the Floyd boundary satisfies the assumptions (A)(B)(C) in Definition \ref{ConvBdryDefn} where the partition is maximal. That is, $[\cdot]$-classes are singletons and we could omit $[\cdot]$ in Lemma \ref{CharMyrberg-general}. Recall the notion of family paths from the discussion preceding Lemma~\ref{LargeTreePtsVersion}

\begin{prop}\label{qrtreeinFloyd}
There exist a quasi-radial tree $T$ rooted at $1$  and a constant $\kappa>0$ with the following properties 
\begin{enumerate}
    \item   $\e T=\e G$;
    \item 
    each family path $v_n\in T$ $(n\ge 0)$ is a $c$-quasi-geodesic ray ending at a Myrberg point $\xi$ so that $\rho_\lambda^{v_n}(v_0,\xi)\ge \kappa$. 
    
\end{enumerate}     
\end{prop}
\begin{proof}
The construction of the quasi-radial tree has been described in Theorem \ref{thm-qrtree-Myrberg-general}. In particular,    $\e T=\e G$ and each family path $v_n$ $(n\ge 0)$ is a $c$-quasi-geodesic ray $\gamma$ ending at a Myrberg point $\xi$. We now prove $\rho_{v_n}(v_0,\xi)\ge \kappa$  by 
using Lemma \ref{LargeFloydDist}. Indeed,   by construction  $v_n$ is an end point of a contracting segment $\alpha$ labeled by a loxodromic element $f\in F$. The set $F$ is finite, so $\rho_x(\alpha^-,\alpha^+) $ has a uniform lower bound  $\kappa>0$. This implies that $\rho_x(1,\xi)\ge \kappa/2$ by  Lemma \ref{LargeFloydDist}.  Up to rescaling $\kappa$ again depending on $L=\max\{d(1,f):f\in F\}$, we can move $x$ to the vertex $v$ by the bi-Lipschitz inequality (\ref{lambdabilip}). The proof is then complete. 
\end{proof}
\begin{thm}\label{MyrbergHdimFloyd}
Assume that $\pGf$ is nontrivial for $1>\lambda>0$. Then the Hausdorff dimension of the Myrberg limit set in the Floyd boundary $\pGf$ is equal to $\omega_G/-\log \lambda$.    
\end{thm}
\begin{proof}
The upper bound is due to Marc Bourdon and a proof is given in \cite[Lemma 4.1]{PYANG}. We only need to prove the lower bound.

Let $T$ be the quasi-radial tree given by Proposition \ref{qrtreeinFloyd}, whose accumulation points are Myrberg points. The argument for the Hausdorff dimension is along the same lines as Lemma \ref{HDLargeTree}. We indicate the modifications.
 Lemma \ref{HDLargeTree} was stated for the visual metric on the Gromov boundary.
 However, we only used the visual metric there  to establish bounds for  shadows of vertices in the quasi-radial tree $T$.   By Lemma \ref{ShadowApproxBalls} shadows in the hyperbolic situation are roughly the same as balls with appropriate radius. In the Floyd metric, we have  the same estimates as in Lemma \ref{ShadowApproxBalls} for the vertices with large Floyd distance by Lemma \ref{lemma3.16PY}. Note that the vertices on each family path have large Floyd distance by Proposition \ref{qrtreeinFloyd}. Thus the lower bound on $\mG$ follows exactly as Lemma \ref{HDLargeTree}. 
\end{proof}
\subsection{Applications: mapping class groups} This subsection sketches an application of the construction in Theorem \ref{thm-qrtree-Myrberg-general} to  the mapping class group action on Teichm\"uller space.

Let $G=\mcg$ denote the orientation-preserving mapping class group of a closed surface $\Sigma_g$ with $g\ge 2$. The group $G$ acts properly on the Teichm\"uller space $\T_g$ equipped with the Teichm\"uller metric.  Pseudo-Anosov elements are strongly contracting  \cite{Minsky}.  Thurston showed that $\T_g$ can be naturally compactified by the space of projective measured foliations $\pmf$. In \cite{YANG22}, the second author 
studied a partition $[\cdot]$ of $\pmf$ due to Kaimanovich-Masur \cite{KaMasur}  from the point of view of topological dynamics. It was shown there that  Assumptions (A)(B)(C) in Definition \ref{ConvBdryDefn}
are satisfied. The partition $[\cdot]$ restricts to singletons on uniquely ergodic points. We can then use Definition \ref{defn-Myrberg-general} to study Myrberg points in $\pmf$. 

By Lemma \ref{CharMyrberg-general},   for any geodesic ray $\sigma$ ending at a Myrberg point, there exists $R>0$ satisfying the following. Let $\{\gamma_n\}$ be an enumeration of closed
geodesics in moduli space. Let $N_R(\gamma_{n})$ denote its $R-$neighborhood.
Then $\sigma$ spends 
arbitrarily long times in $N_R(\gamma_{n})$.
  Masur's criterion \cite{Ma80}   then shows that Myrberg points are necessarily uniquely ergodic points.
By the above discussion, the next result follows from Theorem \ref{thm-qrtree-Myrberg-general}.
\begin{thm}\label{MyrberginMCG}
Fix a basepoint $o\in \T_g$.
There exists a quasi-radial tree $T$ rooted at $o$ in $\mathcal T_g$ with vertices contained in $Go$ so that $\e T=6g-6$ and each radial ray issuing from $o\in T$ ends at a Myrberg point.   
\end{thm}



\bibliographystyle{alpha}
 \bibliography{bibliography}

\end{document}